\documentclass[final,leqno,onefignum,onetabnum]{siamltex1213}

\usepackage{amsmath}
\usepackage{amssymb}
\usepackage{MnSymbol}
\usepackage{hyperref,url}
\usepackage{bm}
\usepackage{placeins}
\usepackage{color}

\newcommand{\N}{\mathbb{N}}

\newcommand{\R}{\mathbb{R}}

\newcommand{\dHaus}{\, \mathrm{d} \mathcal{H}^{n-1} \,}
\newcommand{\dd}{\, \mathrm{d} \,}
\newcommand{\dz}{\, \mathrm{dz}\,}
\newcommand{\dx}{\, \mathrm{dx}\,}
\newcommand{\ds}{\, \mathrm{ds}\,}
\newcommand{\dt}{\, \mathrm{dt}\,}
\newcommand{\dzeta}{\, \mathrm{d} \zeta\,}
\newcommand{\pd}{\partial}
\newcommand{\abs}[1]{\left| #1 \right|}
\newcommand{\norm}[1]{\| #1 \|}

\newcommand{\inner}[2]{\langle #1 , #2 \rangle}

\newcommand{\eps}{\varepsilon}

\newcommand{\surf}{\nabla_{\Gamma}}
\newcommand{\surfz}{\nabla_{\Gamma^{\eps z}}}
\newcommand{\surfzloc}{\tilde{\nabla}_{\Gamma^{\eps z}}}
\newcommand{\surfloc}{\tilde{\nabla}_{\Gamma}}

\newcommand{\dist}{\text{dist}}
\newcommand{\esssup}{\mathop{\mathrm{ess\,sup}}}
\newcommand{\sech}{\mathrm{sech}}

\newcommand{\tr}[1]{\text{tr} \left ( #1 \right )}

\newtheorem{thm}{Theorem}[section]
\newtheorem{cor}[thm]{Corollary}

\newtheorem{remark}{Remark}[section]

\newtheorem{defn}{Definition}[section]
\newtheorem{assump}{Assumption}[section]
\numberwithin{equation}{section}

\title{Analysis of the diffuse domain approach for a bulk-surface coupled PDE system \thanks{}} 

\author{Helmut Abels \thanks{Fakult\"at f\"ur Mathematik, Universit\"at Regensburg, 93040 Regensburg, Germany (\email{Helmut.Abels@mathematik.uni-regensburg.de})} \and Kei Fong Lam \thanks{Fakult\"at f\"ur Mathematik, Universit\"at Regensburg, 93040 Regensburg, Germany (\email{Kei-Fong.Lam@mathematik.uni-regensburg.de}) This author has been supported by the Engineering and Physical Sciences Research Council (EPSRC) grant EP/H023364/1.} \and Bj\"{o}rn Stinner \thanks{Mathematics Institute, Zeeman Building, University of Warwick, Coventry, CV4 7AL, United Kingdom (\email{Bjorn.Stinner@warwick.ac.uk})}}

\begin{document}
\maketitle
\slugger{mms}{xxxx}{xx}{x}{x--x}

\begin{abstract}
We analyse a diffuse interface type approximation, known as the diffuse domain approach, of a linear coupled bulk-surface elliptic partial differential system.  The well-posedness of the diffuse domain approximation is shown using weighted Sobolev spaces and we prove that the solution to the diffuse domain approximation converges weakly to the solution of the coupled bulk-surface elliptic system as the approximation parameter tends to zero.  Moreover, we can show strong convergence for the bulk quantity, while for the surface quantity, we can show norm convergence and strong convergence in a weighted Sobolev space.  Our analysis also covers a second order surface elliptic partial differential equation and a bulk elliptic partial differential equation with Dirichlet, Neumann and Robin boundary condition. 
\end{abstract}

\begin{keywords}
diffuse domain method, weighted Sobolev spaces, well-posedness, diffuse interface approximation, asymptotic analysis, bulk-surface elliptic equations.
\end{keywords}

\begin{AMS}
35J25, 35J50, 35J70, 46E35, 41A30, 41A60
\end{AMS}

\pagestyle{myheadings}
\thispagestyle{plain}
\markboth{H.ABELS, K.F.LAM, B.STINNER}{ANALYSIS OF THE DIFFUSE DOMAIN APPROACH}

\section{Introduction}
The diffuse domain approach \cite{article:RatzVoigt06, article:LiLowengrubRatzVoigt09} is a method originating from the phase field methodology which approximates partial differential equations posed on domains with arbitrary geometries.  The method embeds the original domain $\Omega^{*} \subset \R^{n}$ with complicated geometries into a larger domain $\Omega$ with a simpler geometry.  Drawing on aspects of the phase field methodology, the diffuse domain method replaces the boundary $\Gamma$ of $\Omega^{*}$ with an interfacial layer of thickness $0 < \eps \ll 1$, denoted by $\Gamma_{\eps}$.  The original PDEs posed on $\Omega^{*}$ will have to be extended to $\Omega$ and any surface quantities or boundary terms on $\Gamma$ have to be extended to fields defined on $\Gamma_{\eps}$.  The resulting PDE system, which we denote as the diffuse domain approximation, is defined on $\Omega$ and will have the same order as the original system defined in $\Omega^{*}$, but with additional terms that approximate the original boundary conditions on $\Gamma$.

Our model problem will be the following elliptic coupled bulk-surface system:
\begin{equation*}
(\textrm{CSI}) \quad \begin{alignedat}{2}
-\nabla \cdot (\mathcal{A} \nabla u) + a u &= f && \text{ in } \Omega^{*}, \\
-\surf \cdot (\mathcal{B} \surf v) + bv + \mathcal{A} \nabla u \cdot \nu &= \beta g && \text{ on } \Gamma, \\
\mathcal{A} \nabla u \cdot \nu &= K(v - u) && \text{ on } \Gamma.
\end{alignedat}
\end{equation*}

Here, $\surf  v$ and $\surf \cdot \bm{v}$ denote the surface gradient of $v$ and the surface divergence of $\bm{v}$ on $\Gamma$, respectively.   For a precise definition, we refer the reader to \cite[Section 2]{article:DziukElliott13} or Section \ref{sec:ScaledNbd} below.  Meanwhile, $K,\beta \geq 0$ are non-negative constants, and $\mathcal{A} = (a_{ij})_{1 \leq i,j \leq n}$, and $\mathcal{B} = (b_{ij})_{1 \leq i,j\leq n}$ denote the matrices with function coefficients $a_{ij}$, and $b_{ij}$, respectively.  The precise assumptions on the data and the domain will be given in Section \ref{sec:dataassump}.

We now embed $\Omega^{*} \cup \Gamma$ into a larger domain $\Omega \subset\R^{n}$.  The location of the original boundary $\Gamma$ is encoded in an order parameter $\varphi$ as its zero-level set.  A typical choice for $\varphi$ is a function of the signed distance function of $\Gamma$, $d : \Omega \to \R$, which is defined as
\begin{align}\label{defn:signeddist}
d(x) = \begin{cases}
-\inf_{z \in \Gamma} \abs{x-z} & \text{ for } x \in \Omega^{*}, \\
0 & \text{ for } x \in \Gamma, \\
\inf_{z \in \Gamma} \abs{x-z} & \text{ for } x \in \Omega \setminus \overline{\Omega^{*}}.
\end{cases}
\end{align}

Let $\chi_{\Omega^{*}}$ denote the characteristic function of $\Omega^{*}$ and $\delta_{\Gamma} := \mathcal{H}^{n-1} \lefthalfcup \Gamma$ denote the Hausdorff measure restricted to $\Gamma$ (see Theorem \ref{thm:Alt} for a definition in the sense of distributions given by the Dirac measures of $\Omega^{*}$ and $\Gamma$).  An equivalent distribution form for (CSI) is (see Appendix, Theorem \ref{thm:Alt})
\begin{equation*}
\begin{alignedat}{2}
-\nabla \cdot (\chi_{\Omega^{*}} \mathcal{A} \nabla u) + \chi_{\Omega^{*}} a u &= \chi_{\Omega^{*}}f + \delta_{\Gamma} K(v-u) && \text{ in } \mathcal{D}'(\Omega), \\
-\nabla \cdot (\delta_{\Gamma} \mathcal{B} \nabla v) + \delta_{\Gamma} bv &= \delta_{\Gamma} \beta g - \delta_{\Gamma} K(v-u) && \text{ in } \mathcal{D}'(\Omega).
\end{alignedat}
\end{equation*}

The diffuse domain approximation of (CSI) is derived by approximating $\chi_{\Omega^{*}}$ and $\delta_{\Gamma}$ with more regular functions $\xi_{\eps}(\varphi)$, $\delta_{\eps}(\varphi)$, indexed by $\eps$, which is related to the thickness of the interfacial layer $\Gamma_{\eps}$.  In order words, the diffuse domain approximation of (CSI) is
\begin{equation*}
(\mathrm{CDD}) \quad \begin{alignedat}{2}
 - \nabla \cdot (\xi_{\eps} \mathcal{A}^{E} \nabla u^{\eps}) + \xi_{\eps} a^{E} u^{\eps} & = \xi_{\eps} f^{E}  + \delta_{\eps} K (v^{\eps} - u^{\eps}) && \text{ in } \Omega, \\
 - \nabla \cdot (\delta_{\eps} \mathcal{B}^{E} \nabla v^{\eps}) + \delta_{\eps} b^{E} v^{\eps} &=  \delta_{\eps} \beta g^{E} - \delta_{\eps} K (v^{\eps} - u^{\eps}) && \text{ in } \Omega,
\end{alignedat}
\end{equation*}
where $\mathcal{A}^{E}, \mathcal{B}^{E}, a^{E}, b^{E}, f^{E}$ and $g^{E}$ denote suitable extensions of $\mathcal{A}, \mathcal{B}, a, b, f$ and $g$ to the larger domain $\Omega$.  The precise assumptions on the extensions will be outlined in Section \ref{sec:dataextension}.

Formally, $\xi_{\eps}(\varphi) \to \chi_{\Omega^{*}}$, $\delta_{\eps}(\varphi) \to \delta_{\Gamma}$ as $\eps \to 0$, and so in the limit of vanishing interfacial thickness, we recover the distributional form for (CSI).

For here onwards, we refer to the original problem posed on $\Omega^{*}$ as the SI (sharp interface) problem, and the corresponding diffuse domain approximation posed on $\Omega$ will be denoted as the DD (diffuse domain) problem.

We remark that replacing the original boundary $\Gamma$ with an interfacial layer $\Gamma_{\eps}$ transforms the diffuse domain approximation into a two-scale problem.  This is similar to phase field approximations of free boundary problems.  The idea of adopting the phase field methodology to approximate partial differential equations has been applied to study diffusion inside a stationary cell \cite{article:KockelkorenLevineRappel03}, Turning patterns on a membrane \cite{article:LevineRappel05}, wave propagation in the heart \cite{article:BuzzardFoxSisoNadal08, article:FentonCherryKarmaRappel05}, two-phase flow \cite{article:AlandLowengrubVoigt10} and soluble surfactants \cite{article:GarckeLamStinner14, article:TeigenSongLowengrubVoigt11}.  In \cite{article:BuenoOrovioPerezGarca06, article:BuenoOrovioPerezGarceFenton06}, a diffuse domain type method, denoted as the smoothed boundary method, has been applied to solve partial differential equations on irregular domains with homogeneous Neumann boundary conditions using spectral methods.  A generalised formulation of the smoothed boundary method can be found in \cite{article:YuChenThornton12}.

The phase field methodology provides us with two candidates for $\varphi(x)$.  The first is based on the smooth double-well potential $\psi_{DW}(\varphi) = \frac{1}{4} (1-\varphi^{2})^{2}$ and leads to
\begin{align*}
 \varphi_{DW}(x) := \tanh \left ( \frac{d(x)}{\sqrt{2}\eps} \right ) \text{ for } x \in \Omega.
\end{align*}
The other is based on the double-obstacle potential \cite{incoll:BloweyElliott93,article:ChenElliott94}
\begin{align*}
 \psi_{DO}(\varphi) = \frac{1}{2}(1-\varphi^{2}) + I_{[-1,1]}(\varphi), \quad I_{[-1,1]}(\varphi) = \begin{cases}
                                                                                          0, & \text{ if } \varphi \in [-1,1], \\
+\infty, & \text{ otherwise},
                                                                                         \end{cases}
\end{align*}
and leads to
\begin{align*}
 \varphi_{DO}(x) := \begin{cases}
                     +1, & \text{ if } d(x) > \eps \frac{\pi}{2}, \\
\sin(d(x)/\eps), & \text{ if } \abs{d(x)} \leq \eps \frac{\pi}{2}, \\
-1, & \text{ if } d(x) < -\eps \frac{\pi}{2}.
                    \end{cases}
\end{align*}
Note that in the case of the double-obstacle potential, the interfacial layer $\Gamma_{\eps}$ has finite thickness of $\eps \pi$.  We remark that for general phase field models $\varphi_{DW}$ or $\varphi_{DO}$ is the leading order approximation for the order parameter $\varphi$.  In our setting, the location of the boundary $\Gamma$ is known and hence we can use $\varphi_{DW}$ or $\varphi_{DO}$ as the order parameter.  A common regularisation of $\chi_{\Omega^{*}}$ based on the smooth double-well potential used in  \cite{article:LiLowengrubRatzVoigt09, article:TeigenLiLowengrubWangVoigt09, article:TeigenSongLowengrubVoigt11} is
\begin{align*}
 \xi_{\eps}^{(1)}(x) = \frac{1}{2} (1 - \varphi_{DW}(x) ) = \frac{1}{2} \left (1 - \tanh \left ( \frac{d(x)}{\sqrt{2}\eps} \right ) \right ).
\end{align*}
While an alternative based on the double-obstacle potential is 
\begin{align*}
 \xi_{\eps}^{(2)}(x) = \frac{1}{2} (1 - \varphi_{DO}(x) ).
\end{align*}
There are many regularisations of $\delta_{\Gamma}$ available from the literature \cite{article:TeigenSongLowengrubVoigt11, article:ElliottStinnerStylesWelford11, article:RatzVoigt06, article:LeeJunseok12}.  One well known approximation of $\delta_{\Gamma}$ is a multiple of the Ginzburg--Landau energy density (see \cite{article:ModicaMortola77})
\begin{align*}
\frac{\eps}{2} \abs{\nabla \varphi}^{2} + \frac{1}{\eps} \psi(\varphi).
\end{align*}
From the above discussions regarding the double-well and the double-obstacle potentials, we have two candidates for the regularisation to $\delta_{\Gamma}$:
\begin{align*}
 \delta_{\eps}^{(1)}(x) = \frac{3}{2\sqrt{2}\eps}\sech^{4} \left (\frac{d(x)}{\sqrt{2}\eps} \right ), \quad \delta_{\eps}^{(2)}(x) = \frac{2}{\pi \eps}\cos^{2} \left ( \frac{d(x)}{\eps} \right )  \chi_{\{ x \in \Omega \, : \, \abs{d(x)} \leq \eps \frac{\pi}{2}\}}.
\end{align*}

The convergence analysis of the diffuse domain approach (with the smooth double-well potential), in the limit $\eps \to 0$, has only been done in the context of recovering the original equations via formally matched asymptotics (see \cite{article:LiLowengrubRatzVoigt09, article:TeigenLiLowengrubWangVoigt09,article:TeigenSongLowengrubVoigt11,preprint:LervayLowengrub14}).  A first analytical treatment of convergence in one dimension and on a half-plane in two dimension can be found in \cite{article:FranzGartnerRoosVoigt12}, where the error between the solution to a second order system and the diffuse domain approximation in the $L^{\infty}$ norm is of order $\mathcal{O}(\eps^{1-\mu})$, where $\mu > 0$ arbitrarily small.  

In \cite{article:ElliottStinner09}, a diffuse domain type approximation for an advection-diffusion equation posed on evolving surfaces is considered.  Motivated by modelling and numerical simulations, the diffuse domain approximation utilises a double-obstacle type regularisation.  Note that the regularisations from the double-obstacle potential are degenerate in certain parts of the larger domain $\Omega$ (in particular they are zero outside $\Gamma_{\eps}$ for the set-up in \cite{article:ElliottStinner09}).  Consequently the corresponding diffuse domain approximation becomes a degenerate equation and weighted Sobolev spaces are employed.  The chief results in \cite{article:ElliottStinner09} are the well-posedness of the diffuse domain approximation and weak convergence to the solution of the original system.

For a bulk second order elliptic boundary value problem, the convergence analysis has been studied in \cite{preprint:BurgerElvetunSchlottbom14}.  There, a double-obstacle type diffuse domain approximation to a Robin boundary value problem is studied, with the specific choice $\delta_{\eps} = \abs{\nabla \xi_{\eps}}$.  The authors are able to deduce trace theorems, embedding theorems and Poincar\'{e} inequalities for Sobolev spaces weighted with $\xi_{\eps}$.  Under suitable assumptions on the decay of $\xi_{\eps}$ near $\pd \Gamma_{\eps}$, the authors showed that there exists a $p > 2$ such that the error between the solution to the diffuse domain approximation and the solution to the Robin problem is of order $\mathcal{O}(\eps^{\frac{1}{2} - \frac{1}{p}})$ in a weighted Sobolev norm.  Similar results for the Neumann and Dirichlet boundary conditions are also given.

In this work, we show that weak solutions to (CDD) converge to 
the unique weak solution to (CSI) under appropriate assumptions.  Thanks to the property of the problem, we are also able to show strong convergence in $H^{1}(\Omega^{*})$ for the bulk quantity, while for the surface quantity we have norm convergence and strong convergence in a weighted Sobolev space.  The techniques we develop here can also be applied to a second order surface linear elliptic equation and a second order linear elliptic equation with Robin, Neumann and Dirichlet boundary conditions.

The structure of this article is as follows:  In Section \ref{sec:AssumpMain} we introduce the assumptions and the main results.  In Section \ref{sec:Technical} we prove several technical results that will simplify the proof of the main results, which are contained in Section \ref{sec:Mainproof}.  In Section \ref{sec:Discussion}, we compare our results with the results in \cite{article:ElliottStinner09, preprint:BurgerElvetunSchlottbom14} and discuss some insights and issues we encountered in the analysis.  In the appendix we collect a few results regarding the distributional formulations and properties of the trace operator.

\section{General assumptions and main results}\label{sec:AssumpMain}
We observe that setting $\mathcal{A} = \bm{0}^{n \times n}$, $a = 0$, $f = 0$, $K = 0$, $\beta = 1$ in (CSI) leads to an elliptic equation on $\Gamma$, which we denote as (SSI):
\begin{align*}
 \text{(SSI)} \quad - \surf \cdot (\mathcal{B} \surf v) + bv = g \text{ on } \Gamma.
\end{align*}
Meanwhile, setting $\mathcal{B} = \bm{0}^{n \times n}$, $b(x) \equiv \beta$ and formally sending $K \to \infty$ in (CSI) leads to the Robin boundary condition (RSI):
\begin{equation*}
 \text{(RSI}) \quad 
\begin{alignedat}{2}
 - \nabla \cdot (\mathcal{A} \nabla u) + au &= f && \text{ in } \Omega, \\
\mathcal{A} \nabla u \cdot \nu + \beta u &= \beta g && \text{ on } \Gamma.
\end{alignedat}
\end{equation*}
Neglecting the second equation in (CSI), setting $v  = g$ and formally sending $K \to \infty$ gives the Dirichlet boundary condition (DSI):
\begin{equation*}
\text{(DSI}) \quad \begin{alignedat}{2}
 - \nabla \cdot (\mathcal{A} \nabla u) + au &= f && \text{ in } \Omega, \\
u &= g && \text{ on } \Gamma,
\end{alignedat}
\end{equation*}
while the Neumann boundary condition (NSI) can be obtained by setting $\mathcal{B} = \bm{0}^{n \times n}$, $b = 0$, $\beta = 1$ and neglecting the last boundary condition in (CSI):
\begin{equation*}
\text{(NSI}) \quad  \begin{alignedat}{2}
 - \nabla \cdot (\mathcal{A} \nabla u) + au &= f && \text{ in } \Omega, \\
\mathcal{A} \nabla u \cdot \nu  &= g && \text{ on } \Gamma.
\end{alignedat}
\end{equation*}
In order to cover these derived sharp interface problems we make the following assumptions.

\subsection{Assumptions on the data}\label{sec:dataassump}
\
\begin{assump}[Assumptions on domain]\label{assump:Gamma}
We assume that $\Omega^{*}$ is an open bounded domain in $\R^{n}$ with compact $C^{2}$ boundary $\Gamma$ and outward unit normal $\nu$.  Let $\Omega$ be an open bounded domain in $\R^{n}$ with Lipschitz boundary $\pd \Omega$ such that $\overline{\Omega^{*}} \subset \Omega$ and $\Gamma \cap \pd \Omega = \emptyset$.  In addition, for $(\mathrm{CSI})$ and $(\mathrm{SSI})$, we assume that $\Gamma$ is of class $C^{3}$.\end{assump}

\begin{assump}[Assumptions on the data]\label{assump:General}
We assume that for $1 \leq i, j \leq n$,
\begin{align*}
a_{ij}, a \in L^{\infty}(\Omega^{*}),  \quad f \in L^{2}(\Omega^{*}), \quad b_{ij}, b \in L^{\infty}(\Gamma),  \quad g \in L^{2}(\Gamma),
\end{align*}
and there exist positive constants $\theta_{0}, \theta_{1}, \theta_{2}, \theta_{3}$ such that
\begin{align*}
 (\mathcal{A}(x) \zeta_{1}) \cdot \zeta_{1} \geq \theta_{0} \abs{\zeta_{1}}^{2}, \quad (\mathcal{B}(p) \zeta_{2}) \cdot \zeta_{2} \geq \theta_{1} \abs{\zeta_{2}}^{2}, \quad a(x) \geq \theta_{2}, \quad b(p) \geq \theta_{3},
\end{align*}
for all $x\in \Omega^{*}$, $p \in \Gamma$, $\zeta_{1} \in \R^{n}$ and $\zeta_{2} \in T_{p} \Gamma \subset \R^{n}$.
\end{assump}

\begin{assump}[Specific assumption for (DSI)]\label{assump:DSI}
For $(\mathrm{DSI})$, we assume
\begin{align*}
g \in H^{\frac{1}{2}}(\Gamma).
\end{align*}
\end{assump}

Here, for a bounded Lipschitz domain $D \subset \R^{n}$, the fractional-order Sobolev spaces $H^{\frac{1}{2}}(\pd D)$ is defined as
\begin{align*}
H^{\frac{1}{2}}(\pd D) := \left \{ u  \in L^{2}(\pd D) : \int_{\pd D} \int_{\pd D} \frac{\abs{u(x) - u(y)}^{2}}{\abs{x-y}^{n}} \dHaus(x) \dHaus(y) < \infty \right \},
\end{align*}
equipped with the inner product
\begin{align*}
\inner{u}{v}_{H^{\frac{1}{2}}(\pd D)} & := \int_{\pd D} uv \dHaus \\
& + \int_{\pd D} \int_{\pd D} \frac{(u(x)-u(y))(v(x) - v(y))}{\abs{x-y}^{n}} \dHaus(x) \dHaus(y),
\end{align*}
where $\dHaus$ denotes the $n-1$ dimensional Hausdroff measure.

\begin{assump}[Specific assumption for (NSI)]\label{assump:NSI}
For $(\mathrm{NSI})$, if $g \neq 0$ then we assume that
\begin{align*}
g \in H^{\frac{1}{2}}(\Gamma), \text{ and } a_{ij} \in C^{0,1}(\overline{\Omega^{*}}) \text{ for } 1 \leq i, j \leq n.
\end{align*}
\end{assump}

\subsection{Results on the sharp interface (SI) problems}
Let $\gamma_{0} : W^{1,1}(\Omega^{*}) \to L^{1}(\Gamma)$ denote the trace operator.  If $\Gamma$ is Lipschitz and $g \in H^{\frac{1}{2}}(\Gamma)$, then by Theorem \ref{thm:trace}, there exists $\tilde{g} \in H^{1}(\Omega^{*})$ such that $\gamma_{0}(\tilde{g}) = g$, and (DSI) is equivalent to the following homogeneous problem (DSIH):
\begin{equation*}
(\mathrm{DSIH}) \quad \begin{alignedat}{2}
-\nabla \cdot (\mathcal{A} \nabla w) + aw & = f + \nabla \cdot (\mathcal{A} \nabla \tilde{g}) - a\tilde{g} && \text{ in } \Omega^{*}, \\
w &= 0 && \text{ on } \Gamma.
\end{alignedat}
\end{equation*}
If $w$ is a weak solution to (DSIH), then $u := w + \tilde{g}$ is a weak solution to (DSI).

We can also consider (DSIH) as the limiting problem $\beta \to \infty$ of the following homogeneous Robin problem (RSIH):
\begin{equation*}
(\mathrm{RSIH}) \quad \begin{alignedat}{2}
-\nabla \cdot (\mathcal{A} \nabla w) + aw &= f + \nabla \cdot (\mathcal{A} \nabla \tilde{g}) - a\tilde{g} && \text{ in } \Omega^{*}, \\
\mathcal{A} \nabla w \cdot \nu + \beta w &= 0 && \text{ on } \Gamma.
\end{alignedat}
\end{equation*}

In fact, we have (see also \cite{article:MarusicPaloka99}):
\begin{lemma}\label{lem:DirichletfromRobin}
Let $\Omega^{*}$ be an open bounded domain with $C^{1}$ boundary $\Gamma$.  Let the data satisfy Assumption \ref{assump:General} with $\tilde{g} \in H^{1}(\Omega^{*})$.  For each $\beta > 0$, $(\mathrm{RSIH})$ is well-posed with unique weak solution $w^{\beta} \in H^{1}(\Omega^{*})$, and $(\mathrm{DSIH})$ is well-posed with unique weak solution $w_{D} \in H^{1}_{0}(\Omega^{*})$ (see Theorem \ref{thm:wellposedSI} below for the well-posedness).  Then, as $\beta \to \infty$, $w^{\beta}$ converges weakly to $w_{D}$ in $H^{1}(\Omega^{*})$.
\end{lemma}

The proof of Lemma \ref{lem:DirichletfromRobin} will be outlined in Section \ref{sec:proof:DirichletfromRobin}.

Similarly, by Assumption \ref{assump:NSI} and Theorem \ref{thm:conormal}, there exists $h \in H^{2}(\Omega^{*})$ such that $\mathcal{A} \nabla h \cdot \nu = g$ in the sense of traces.  Then, (NSI) is equivalent to the following homogeneous problem (NSIH):
\begin{equation*}
(\mathrm{NSIH}) \quad \begin{alignedat}{2}
-\nabla \cdot (\mathcal{A} \nabla w) + aw &= f + \nabla \cdot (\mathcal{A} \nabla h) - a h && \text{ in } \Omega^{*}, \\
\mathcal{A} \nabla w \cdot \nu &= 0 && \text{ on } \Gamma.
\end{alignedat}
\end{equation*}
If $w$ is a weak solution to (NSIH), then $u := w + h$ is a weak solution to (NSI).  We point out that, if $h \neq 0$, then the coefficient $\mathcal{A}$ in $(\mathrm{NSIH})$ is as assumed in Assumption \ref{assump:NSI}, i.e., $a_{ij} \in C^{0,1}(\overline{\Omega^{*}})$ for $1 \leq i, j \leq n$.

We now state the well-posedness results for the SI problems.  For convenience, let us define the following bilinear forms $a_{B} : H^{1}(\Omega^{*}) \times H^{1}(\Omega^{*}) \to \R$, $a_{S} : H^{1}(\Gamma) \times H^{1}(\Gamma) \to \R$, $l_{B} : L^{2}(\Omega^{*}) \times L^{2}(\Omega^{*}) \to \R$, and $l_{S} : L^{2}(\Gamma) \times L^{2}(\Gamma) \to \R$:
\begin{alignat}{4}
 \label{bulkbilinear} a_{B}(\varphi, \psi) & := \int_{\Omega^{*}} \mathcal{A} \nabla \varphi \cdot \nabla \psi + a \varphi \psi \dx, && \quad l_{B}(\varphi, \psi ) && := \int_{\Omega^{*}} \varphi \psi \dx,  \\
\label{surfacebilinear} a_{S}(\varphi, \psi) &:= \int_{\Gamma} \mathcal{B} \surf \varphi \cdot \surf \psi + b \varphi \psi \dHaus, && \quad l_{S}(\varphi, \psi ) &&:= \int_{\Gamma} \varphi \psi \dHaus.
\end{alignat}

\begin{thm}[Well-posedness for the SI problems]\label{thm:wellposedSI}
Let Assumptions \ref{assump:Gamma}, \ref{assump:General}, \ref{assump:DSI} and \ref{assump:NSI} be satisfied.  Then there exist unique weak solutions
\begin{alignat*}{3}
(u,v) \in H^{1}(\Omega^{*}) \times H^{1}(\Gamma) & \text{ for } \mathrm{(CSI)}, &&\quad v_{S} \in H^{1}(\Gamma) && \text{ for } \mathrm{(SSI)},\\
u_{R}  \in H^{1}(\Omega^{*}) & \text{ for } \mathrm{(RSI)}, && \quad w_{R} \in H^{1}(\Omega^{*}) && \text{ for } \mathrm{(RSIH}), \\
u_{N} \in H^{1}(\Omega^{*}) & \text{ for } \mathrm{(NSI)}, && \quad w_{N} \in H^{1}(\Omega^{*}) && \text{ for } \mathrm{(NSIH)}, \\
& && \quad w_{D} \in H^{1}_{0}(\Omega^{*}) && \text{ for } \mathrm{(DSIH)},
\end{alignat*}
such that for all $\varphi \in H^{1}(\Omega^{*})$, $\varphi_{0} \in H^{1}_{0}(\Omega^{*})$, $\psi \in H^{1}(\Gamma)$,
\begin{align*}
 a_{B}(u,\varphi) + a_{S}(v, \psi) + K l_{S}(v - \gamma_{0}(u),\psi - \gamma_{0}(\varphi)) & = l_{B}(f, \varphi) + \beta l_{S}(g, \psi), \\
a_{S}(v_{S}, \psi) & = l_{S}(g, \psi), \\
a_{B}(u_{R}, \varphi) + \beta l_{S}(\gamma_{0}(u_{R}), \gamma_{0}(\varphi)) & = l_{B}(f, \varphi) + \beta l_{S}(g, \gamma_{0}(\varphi)), \\
a_{B}(w_{R}, \varphi) + \beta l_{S}(\gamma_{0}(w_{R}), \gamma_{0}(\varphi)) & = l_{B}(f, \varphi) - a_{B}(\tilde{g}, \varphi), \\
a_{B}(u_{N}, \varphi) &= l_{B}(f, \varphi) + l_{S}(g, \gamma_{0}(\varphi)), \\
a_{B}(w_{N}, \varphi) &= l_{B}(f + \nabla \cdot (\mathcal{A} \nabla h) - a h, \varphi), \\
a_{B}(w_{D}, \varphi_{0}) & = l_{B}(f, \varphi_{0}) - a_{B}(\tilde{g}, \varphi_{0}).
\end{align*}
Moreover, there exist constants $C$, independent of $(u, v)$, $v_{S}$, $u_{R}$, $w_{R}$, $w_{D}$, $u_{N}$, $w_{N}$, such that
\begin{align*}
\norm{u}_{H^{1}(\Omega^{*})}^{2} + \norm{v}_{H^{1}(\Gamma)}^{2} & \leq C (\norm{f}_{L^{2}(\Omega^{*})}^{2} + \norm{g}_{L^{2}(\Gamma)}^{2}), \\
\norm{v_{S}}_{H^{1}(\Gamma)}^{2} & \leq C \norm{g}_{L^{2}(\Gamma)}^{2}, \\
\norm{u_{R}}_{H^{1}(\Omega^{*})}^{2} + \beta \norm{\gamma_{0}(u_{R})}_{L^{2}(\Gamma)}^{2} & \leq C( \norm{f}_{L^{2}(\Omega^{*})}^{2} + \norm{g}_{L^{2}(\Gamma)}^{2}), \\
\norm{w_{R}}_{H^{1}(\Omega^{*})}^{2} + \beta \norm{\gamma_{0}(w_{R})}_{L^{2}(\Gamma)}^{2} & \leq C(\norm{f}_{L^{2}(\Omega^{*})}^{2} + \norm{\tilde{g}}_{H^{1}(\Omega^{*})}^{2}), \\
\norm{u_{N}}_{H^{1}(\Omega^{*})}^{2} & \leq C(\norm{f}_{L^{2}(\Omega^{*})}^{2} + \norm{g}_{L^{2}(\Gamma)}^{2}), \\
\norm{w_{N}}_{H^{1}(\Omega^{*})}^{2} & \leq C(\norm{f}_{L^{2}(\Omega^{*})}^{2} + \norm{h}_{H^{2}(\Omega^{*})}^{2}), \\
\norm{w_{D}}_{H^{1}(\Omega^{*})}^{2} & \leq C( \norm{f}_{L^{2}(\Omega^{*})}^{2} + \norm{\tilde{g}}_{H^{1}(\Omega^{*})}^{2}).
\end{align*}

\end{thm}
We note that the constant $C$ for $w_{N}$ depends on $\norm{\mathcal{A}}_{C^{0,1}(\overline{\Omega^{*}})}$, while the constant for $w_{R}$ is independent of $\beta$ (see Section \ref{sec:proof:DirichletfromRobin}).

We will outline the proof for (CSI) in Section \ref{sec:proof:CSI}.  The well-posedness for the other SI problems follow similarly from the application of the Lax--Milgram theorem, and hence we will omit the details.

\subsection{Derivation of the diffuse domain approximations}
Motivated by Theorem \ref{thm:Alt}, replacing $\chi_{\Omega^{*}}$ and $\delta_{\Gamma}$ by $\xi_{\eps}$ and $\delta_{\eps}$, we obtain the corresponding diffuse domain approximations for (SSI), (RSI), (RSIH), (NSI) and (NSIH), which we denote by (SDD), (RDD), (RDDH), (NDD) and (NDDH):
\begin{alignat*}{3}
(\mathrm{SDD}) &  - \nabla \cdot (\delta_{\eps} \mathcal{B}^{E} \nabla v_{S}^{\eps}) + \delta_{\eps} b^{E} v_{S}^{\eps} = \delta_{\eps} g^{E} && \text{ in } \Omega, \\
(\mathrm{RDD}) &  - \nabla \cdot (\xi_{\eps} \mathcal{A}^{E} \nabla u_{R}^{\eps}) + \xi_{\eps} a^{E} u_{R}^{\eps} + \beta \delta_{\eps} u_{R}^{\eps} = \xi_{\eps} f^{E} + \beta \delta_{\eps} g^{E} && \text{ in } \Omega, \\
(\mathrm{RDDH}) &  - \nabla \cdot (\xi_{\eps} \mathcal{A}^{E} \nabla w_{R}^{\eps}) +  \xi_{\eps} a^{E} w_{R}^{\eps} + \beta \delta_{\eps} w_{R}^{\eps} = \xi_{\eps} (f^{E} - a^{E} \tilde{g}^{E}) + \nabla \cdot (\xi_{\eps} \mathcal{A}^{E} \nabla \tilde{g}^{E})  && \text{ in } \Omega, \\
(\mathrm{NDD}) & - \nabla \cdot (\xi_{\eps} \mathcal{A}^{E} \nabla u_{N}^{\eps}) +  \xi_{\eps} a^{E} u_{N}^{\eps} = \xi_{\eps} f^{E} + \delta_{\eps} g^{E} && \text{ in } \Omega, \\
 (\mathrm{NDDH}) &  - \nabla \cdot (\xi_{\eps} \mathcal{A}^{E}  \nabla w_{N}^{\eps}) + \xi_{\eps} a^{E} w_{N}^{\eps} = \xi_{\eps} f^{E} + \xi_{\eps} \nabla \cdot (\mathcal{A}^{E} \nabla h^{E})  - \xi_{\eps} a^{E} h^{E} && \text{ in } \Omega.
\end{alignat*}
Here, terms with the superscript $E$ denote extensions to the larger domain $\Omega$.  We will outline the specific choices for these extensions in Section \ref{sec:dataextension} below.

We note that Theorem \ref{thm:Alt} does not allow us to infer an approximation for PDEs with Dirichlet boundary conditions.  However, motivated by Lemma \ref{lem:DirichletfromRobin}, we obtain an approximation by scaling $\beta$ appropriately in (RDDH).  Hence, a diffuse domain approximation to (DSIH) is $(\mathrm{DDDH})$:
\begin{align*}
- \nabla \cdot (\xi_{\eps} \mathcal{A}^{E} \nabla w_{D}^{\eps}) +  \xi_{\eps} a^{E} w_{D}^{\eps} + \frac{1}{\eps} \delta_{\eps} w_{D}^{\eps} = \xi_{\eps} (f^{E} - a^{E} \tilde{g}^{E}) + \nabla \cdot (\xi_{\eps} \mathcal{A}^{E} \nabla \tilde{g}^{E}) \text{ in } \Omega.
\end{align*}
Here, one may infer that if $w_{D}^{\eps}$ converges to some function $w^{0}$ as $\eps \to 0$, then $\delta_{\eps} w_{D}^{\eps}$ would converge to the trace of $w^{0}$.  So the term $\frac{1}{\eps} \delta_{\eps} w_{D}^{\eps}$ can be seen as a penalisation to enforce $\gamma_{0}(w^{0}) = 0$ in the limit $\eps \to 0$.

\begin{remark}
We could set $\beta = \frac{1}{\eps^{m}}$ for $m > 0$ to obtain a family of diffuse domain approximations to $(\mathrm{DSIH})$.  Then $(\mathrm{DDDH})$ is the case $m = 1$.  We refer the reader to the discussion in Section \ref{sec:AltDirichlet} for more details.
\end{remark}

\begin{remark}
It will turn out that the problem $(\mathrm{NDDH})$ has better properties than $(\mathrm{NDD})$.  In particular, the subsequent analysis for the Neumann problem is performed with $(\mathrm{NDDH})$ instead of $(\mathrm{NDD})$.  We refer the reader to the discussion in Section \ref{sec:Neumann} for more details.
\end{remark}

Finally, we assume a homogeneous Neumann boundary condition  for our diffuse domain approximations:
\begin{align*}
\mathcal{A}^{E} \nabla u^{\eps} \cdot \nu_{\pd \Omega} = \mathcal{A}^{E} \nabla w^{\eps} \cdot \nu_{\pd \Omega} =  \mathcal{B}^{E} \nabla v^{\eps} \cdot \nu_{\pd \Omega} = 0 \text{ on } \pd \Omega.
\end{align*}

\subsection{Assumptions on data extensions}\label{sec:dataextension}
In general, the extension operator is not unique.  For Sobolev functions $f \in W^{l,p}(\Omega^{*})$, where $l \geq 1, 1 \leq p \leq \infty$, one can use an order $l-1$ reflection about the boundary $\Gamma$ to extend $f$ to the exterior of $\Omega^{*}$.  The case $l = 1$ is outlined in the extension theorem \cite[Theorem 1, p. 254]{book:Evans}, while for higher order reflections, we refer to \cite[Theorem 5.19, p. 148]{book:AdamsFournier} or \cite[p. 43--44]{book:MazjaPoborchi}.  

Thus, for (DSIH) we can extend $\tilde{g} \in H^{1}(\Omega^{*})$ to $\tilde{g}^{E} \in H^{1}(\Omega)$ by the extension theorem, and for (NSIH) we can extend $h \in H^{2}(\Omega^{*})$ to $h^{E} \in H^{2}(\Omega)$ by the reflection method.

For the other bulk data, i.e., $\mathcal{A}$, $a$ and $f$, we make the following assumption:
\begin{assump}[Extension of bulk data]\label{assump:bulkdata}
Let $1\leq i, j \leq n$.  For $a_{ij}, a \in L^{\infty}(\Omega^{*})$ and $f \in L^{2}(\Omega^{*})$, we assume that there exist extensions $a_{ij}^{E}, a^{E} \in L^{\infty}(\Omega)$, and $f^{E} \in L^{2}(\Omega)$ such that $\mathcal{A}^{E} = (a_{ij}^{E})_{1\leq i,j \leq n}$ is uniformly elliptic with constant $\theta_{0}$ and $a^{E}(x) \geq \theta_{2}$ for a.e. $x \in \Omega$.

For $(\mathrm{NSIH})$ we additionally assume that the extension $a_{ij}^{E} \in C^{0,1}(\overline{\Omega})$ for $1\leq i, j \leq n$ and the matrix $\mathcal{A}^{E}$ preserves uniform ellipticity.
\end{assump}

The surface data will be extended in a specific way.  Recall that by Assumption \ref{assump:Gamma}, $\Gamma$ is at least of class $C^{2}$.  We define the tubular neighbourhood $\text{Tub}^{r}(\Gamma)$ of $\Gamma$ with width $r > 0$ as
\begin{align*}
\text{Tub}^{r}(\Gamma) := \{ x \in \Omega : \abs{d(x)} < r \}.
\end{align*}
Then, by \cite[Lemma 14.16]{book:GilbargTrudinger}, there exists $\eta > 0$ such that the signed distance function $d$ to $\Gamma$ is of class $C^{2}(\text{Tub}^{\eta}(\Gamma))$.  Moreover, it can be shown that $d$ is globally Lipschitz with constant 1 (see \cite[Section 14.6]{book:GilbargTrudinger}).

For each $y \in \Gamma$, let $T_{y} \Gamma$ and $\nu(y)$ denote its tangent space and outward pointing unit normal, respectively.  A standard result in differential geometry shows that for $\eta$ sufficiently small, there is a $C^{2}$ diffeomorphism between $\text{Tub}^{\eta}(\Gamma)$ and $\Gamma \times (-\eta, \eta)$ given by
\begin{equation}\label{diffeo}
\begin{aligned}
\Theta^{\eta} : \text{Tub}^{\eta}(\Gamma) \to \Gamma \times (-\eta, \eta), \quad \Theta^{\eta}(x) = (p(x), d(x)),
\end{aligned}
\end{equation}
where, for any $x \in \text{Tub}^{\eta}(\Gamma)$, we define the closest point operator (see \cite{article:MerrimanRuuth07} or \cite[Lemma 2.8]{article:DziukElliott13}) $p : \text{Tub}^{\eta}(\Gamma) \to \Gamma$ by
\begin{align}\label{closestpoint}
p(x):= x - d(x) \nu(p(x)).
\end{align}
Then, we also have
\begin{align}\label{nablad}
\nabla d(x) = \nu(p(x)) \text{ for } x \in \text{Tub}^{\eta}(\Gamma).
\end{align}

\begin{defn}[Constant extension in the normal direction]
For any $g \in L^{q}(\Gamma)$, $1 \leq q \leq \infty$, we define its constant extension $g^{e}$ off $\Gamma$ to $\mathrm{Tub}^{\eta}(\Gamma)$ in the normal direction as
\begin{align}\label{defn:constext}
g^{e}(x) = g(p(x)) \text{ for all } x \in \mathrm{Tub}^{\eta}(\Gamma).
\end{align}
\end{defn}
By Corollary \ref{cor:constextH1} below, we have that $g^{e} \in L^{q}(\mathrm{Tub}^{\eta}(\Gamma))$ if $g \in L^{q}(\Gamma)$ for $1 \leq q \leq \infty$.

For the surface data, we make the following assumption:
\begin{assump}[Extension of surface data]\label{assump:surfacedata}
Let $1\leq i, j \leq n$.  For $b_{ij}, b \in L^{\infty}(\Gamma)$ and $g \in L^{2}(\Gamma)$, let $b_{ij}^{e}, b^{e} \in L^{\infty}(\mathrm{Tub}^{\eta}(\Gamma))$ and $g^{e} \in L^{2}(\mathrm{Tub}^{\eta}(\Gamma))$ denote the constant extensions of $b_{ij}, b$ and $g$ off $\Gamma$ to $\mathrm{Tub}^{\eta}(\Gamma)$ in the normal direction, respectively.  We assume that there exist extensions $b_{ij}^{E}, b^{E} \in L^{\infty}(\Omega)$, and $g^{E} \in L^{2}(\Omega)$ of $b_{ij}^{e}, b^{e}$ and $g^{e}$, respectively, such that $\mathcal{B}^{E} = (b_{ij}^{E})_{1\leq i,j \leq n}$ is uniformly elliptic with constant $\theta_{1}$ and $b^{E}(x) \geq \theta_{3}$ for a.e. $x \in \Omega$.
\end{assump}

\begin{remark}\label{rem:extensionambiguity}
To avoid confusion with the different types of extensions for the bulk and surface data introduced above, we will make the following clarification and use the notation: The superscripts $Ea$, $Er$ and $Ec$ will denote arbitrary extension, extension by reflection and constant extension in the normal direction.  More precisely,
\begin{itemize}
\item For $1 \leq i, j \leq n$, the extensions $b_{ij}^{Ec}, b^{Ec} \in L^{\infty}(\Omega)$ that appear in $(\mathrm{CDD})$ and $(\mathrm{SDD})$ are the extensions of $b_{ij} , b \in L^{\infty}(\Gamma)$  constructed from extending constantly in the normal direction off $\Gamma$ to $\mathrm{Tub}^{\eta}(\Gamma)$, and then arbitrary from $\mathrm{Tub}^{\eta}(\Gamma)$ to $\Omega$, such that Assumption \ref{assump:surfacedata} is satisfied.
\item The extensions $\tilde{g}^{Er} \in H^{1}(\Omega)$ and $h^{Er} \in H^{2}(\Omega)$ that appear in $(\mathrm{DDDH})$ and $(\mathrm{NDDH})$, respectively, are extensions by reflection of $\tilde{g} \in H^{1}(\Omega^{*})$ and $h \in H^{2}(\Omega^{*})$.
\item The extension $g^{Ec} \in L^{2}(\Omega)$ that appears in $(\mathrm{CDD})$, $(\mathrm{SDD})$ and $(\mathrm{RDD})$ is the extension of $g \in L^{2}(\Gamma)$ constructed from extending constantly in the normal direction off $\Gamma$ to $\mathrm{Tub}^{\eta}(\Gamma)$, and then arbitrary from $\mathrm{Tub}^{\eta}(\Gamma)$ to $\Omega$.
\item The extensions $a^{Ea} \in L^{\infty}(\Omega)$ and $f^{Ea} \in L^{2}(\Omega)$ that appear in $(\mathrm{CDD})$, $(\mathrm{RDD})$, $(\mathrm{NDDH})$ and $(\mathrm{DDDH})$ are arbitrary extensions of $a \in L^{\infty}(\Omega^{*})$ and $f \in L^{2}(\Omega^{*})$ such that $a^{Ea} \geq \theta_{2}$ a.e. in $\Omega$.
\item For $1 \leq i, j \leq n$, the extension $a^{Ea}_{ij} \in L^{\infty}(\Omega)$ that appears in $(\mathrm{CDD})$, $(\mathrm{RDD})$ and $(\mathrm{DDDH})$ is an arbitrary extension of $a_{ij} \in L^{\infty}(\Omega^{*})$, such that the matrix $\mathcal{A}^{Ea}$ preserves uniform ellipticity.  Meanwhile, the extension  $a^{Er}_{ij} \in C^{0,1}(\overline{\Omega})$ that appears in $(\mathrm{NDDH})$ is the extension by reflection of $a_{ij} \in C^{0,1}(\overline{\Omega^{*}})$, and the matrix $\mathcal{A}^{Er}$ preserves uniform ellipticity.
\end{itemize}
\end{remark}

\subsection{Assumptions on regularisations}
We first introduce the functions $\xi$ and $\delta$, from which the regularisations $\xi_{\eps}$ and $\delta_{\eps}$ are constructed by a rescaling.

\begin{assump}\label{assump:Xi}
We assume that $\xi : \R \to [0,1]$ is a continuous, nonnegative, monotone function such that
\begin{align}\label{property:Xi}
0 \leq \xi(t) \leq \xi(0) = \frac{1}{2} \leq \xi(s) \leq 1, \quad \forall s \leq 0 \leq t, \text{ and } \lim_{\eps \to 0} \xi \left ( \frac{x}{\eps} \right ) = \begin{cases}
1, & \text{ if } x < 0, \\
0, & \text{ if } x > 0, \\
\frac{1}{2} & \text{ if } x = 0.
\end{cases}
\end{align}
\end{assump}

\begin{assump}\label{assump:delta}
We assume that $\delta : \R \to \R_{\geq 0}$ is a $C^{1}$, nonnegative, even function such that
\begin{align}
\int_{\R} \delta(s) \ds = 1, \quad \delta(s_{1}) \geq \delta(s_{2}) \text{ if } \abs{s_{1}} \leq \abs{s_{2}}, \label{property:delta} \\
\int_{\{s \in \R : \delta(s) > 0 \}} \frac{\abs{\delta'(s)}^{2}}{\delta(s)} \ds + \int_{\R} \sqrt{\delta(s)} + \delta(s) (\abs{s} + \abs{s}^{2}) \ds =: C_{\delta, \mathrm{int}} < \infty, \label{property:delta:bound} 
\end{align}
and for any $q \geq 1$,
\begin{align}\label{property:delta:ptwiseconv}
\lim_{\eps \to 0} \frac{1}{\eps^{q}} \delta \left ( \frac{x}{\eps} \right ) = \begin{cases} 
0, & \text{ if } x \neq 0, \\
+\infty, & \text{ if } x = 0.
\end{cases}
\end{align}
Moreover, we assume there exists a constant $C_{\xi} > 0$ such that
\begin{align}\label{property:Xidelta}
C_{\xi} \delta(t) \leq \xi(t) \text{ for all } t \in \R.
\end{align}
\end{assump}

\begin{defn}
Let $d(x)$ denote the signed distance function to $\Gamma$.  For $x \in \Omega$ and for each $\eps \in (0,1]$, we define
\begin{align}\label{defn:Regularisation}
\xi_{\eps}(x) := \xi \left ( \frac{d(x)}{\eps} \right ), \quad \delta_{\eps}(x) := \frac{1}{\eps} \delta \left ( \frac{d(x)}{\eps} \right ),
\end{align}
with
\begin{align}\label{defn:OmegaepsGammaeps}
\Omega_{\eps} := \{ x \in \Omega : \xi_{\eps}(x) > 0 \}, \quad \Gamma_{\eps} := \{ x \in \Omega : \delta_{\eps}(x) > 0 \}.
\end{align}
\end{defn}
By (\ref{property:Xi}), (\ref{property:delta:ptwiseconv}) and (\ref{property:Xidelta}), we observe that
\begin{align*}
\Omega^{*} \cup \Gamma \subset \Omega_{\eps}, \quad \Gamma \subset \Gamma_{\eps} \subset \Omega_{\eps}, \text{ for all } \eps > 0.
\end{align*}

One can check that Assumptions \ref{assump:Xi} and \ref{assump:delta} are satisfied by our candidate regularisations originating from the double-well potential: 
\begin{align*}
\xi^{DW}(x) := \frac{1}{2} \left ( 1 - \tanh \left ( \frac{x}{\sqrt{2}} \right ) \right ), \quad \delta^{DW}(x) := \frac{3}{2 \sqrt{2}} \sech^{4} \left ( \frac{x}{\sqrt{2}} \right ),
\end{align*}
and from the double-obstacle potential:
\begin{equation}
\label{DWDO} 
\begin{aligned}\xi^{DO}(x) & := \chi_{(-\infty, -\frac{\pi}{2})} (x) + \frac{1}{2} (1-\sin(x)) \chi_{[-\frac{\pi}{2}, \frac{\pi}{2}]}(x), \\
 \delta^{DO}(x) & := \frac{2}{\pi} \cos^{2}(x) \chi_{[-\frac{\pi}{2}, \frac{\pi}{2}]}(x).
\end{aligned}
\end{equation}

\begin{remark}\label{rem:DOdegenerate}
For any $\eps > 0$, $\xi_{\eps}^{DW}$ and $\delta_{\eps}^{DW}$ derived from the double-well potential are non-degenerate in $\Omega$, i.e., $\Omega_{\eps} = \Gamma_{\eps} = \Omega$ for all $\eps > 0$.  However, $\xi_{\eps}^{DO}$ and $\delta_{\eps}^{DO}$ originating from the double-obstacle potential are degenerate in $\Omega$.  In particular, for $(\ref{DWDO})$,
\begin{align*}
\Omega_{\eps} = \Omega^{*} \cup \mathrm{Tub}^{\eps \frac{\pi}{2}}(\Gamma), \quad \Gamma_{\eps} = \mathrm{Tub}^{\eps \frac{\pi}{2}}(\Gamma).
\end{align*}
Moreover, if $\eps_{1} < \eps_{2}$, we have $\Omega_{\eps_{1}} \subset \Omega_{\eps_{2}}$ and $\Gamma_{\eps_{1}} \subset \Gamma_{\eps_{2}}$.  However, the framework of weighted Sobolev spaces is flexible enough to allow us to deduce well-posedness of the diffuse domain approximations with both the double-well and double-obstacle regularisations.
\end{remark}

\subsection{Weighted Sobolev spaces}
Due to the presence of $\xi_{\eps}$ and $\delta_{\eps}$ in the diffuse domain approximations, the natural function spaces to look for well-posedness are Sobolev spaces weighted by $\xi_{\eps}$ and $\delta_{\eps}$. In the following, measurability and almost everywhere are with respect to the Lebesgue measure.

\begin{defn}
For fixed $\eps > 0$, we define
\begin{align*}
L^{2}(\Omega_{\eps}, \xi_{\eps}) & := \left \{ f : \Omega_{\eps} \to \R \text{ measurable s.t. } \int_{\Omega_{\eps}} \xi_{\eps} \abs{f}^{2} \dx < \infty \right \}, \\
L^{2}(\Gamma_{\eps}, \delta_{\eps}) & := \left \{ f : \Gamma_{\eps} \to \R \text{ measurable s.t. } \int_{\Gamma_{\eps}} \delta_{\eps} \abs{f}^{2} \dx < \infty \right \}.
\end{align*}
\end{defn}

By Lipschitz continuity of the signed distance function $d$, and the continuity of $\xi(\cdot)$, we see that $\xi_{\eps}(\cdot)$ is continuous, and consequently $\frac{1}{\xi_{\eps}}$ is bounded in all compact sets $B \subset \Omega_{\eps}$.  Thus, $\frac{1}{\xi_{\eps}} \in L^{1}_{loc}(\Omega_{\eps})$ and by H\"{o}lder's inequality we have the continuous embedding $L^{2}(\Omega_{\eps}, \xi_{\eps}) \subset L^{1}_{loc}(\Omega_{\eps})$ (see also \cite[Theorem 1.5]{article:KufnerOpic84}).  Thus, we can define derivatives for $f \in L^{2}(\Omega_{\eps}, \xi_{\eps})$ in a distributional sense.  I.e., for any multiindex $\alpha$, we call a function $g$ the $\alpha^{th}$ distributional derivative of $f$, and write $g = D^{\alpha} f$, if for every $\phi \in C^{\infty}_{c}(\Omega_{\eps})$, 
\begin{align*}
\int_{\Omega_{\eps}} f D^{\alpha} \phi \dx = (-1)^{\abs{\alpha}} \int_{\Omega_{\eps}} g \phi \dx.
\end{align*}
We define the vector space
\begin{align*}
W^{1,2}(\Omega_{\eps}, \xi_{\eps}) := \{ f \in L^{2}(\Omega_{\eps}, \xi_{\eps}) : D^{\alpha} f \in L^{2}(\Omega_{\eps}, \xi_{\eps}) \text{ for } \abs{\alpha} = 1 \}.
\end{align*}

A similar definition for the vector space $W^{1,2}(\Gamma_{\eps}, \delta_{\eps})$ can be made since $\frac{1}{\delta_{\eps}} \in L^{1}_{loc}(\Gamma_{\eps})$ by the continuity of $\delta_{\eps}(\cdot)$.

For the subsequent analysis, we will present the proofs with the double-well regularisation in mind and detail any necessary modifications for the double-obstacle regularisation afterwards.  

To streamline the presentation, it is more convenient to have a fixed domain when working with weighted Sobolev spaces, hence we introduce the following notation:
\begin{defn}
For fixed $\eps > 0$, we define
\begin{align*}
L^{2}(\Omega, \xi_{\eps}) & := \{ f : \Omega \to \R \text{ measurable s.t. } f \mid_{\Omega_{\eps}} \in L^{2}(\Omega_{\eps}, \xi_{\eps}) \}, \\
H^{1}(\Omega, \xi_{\eps}) & := \{ f : \Omega \to \R \text{ measurable s.t. } f \mid_{\Omega_{\eps}} \in W^{1,2}(\Omega_{\eps}, \xi_{\eps}) \},
\end{align*}
with inner product and induced norms:
\begin{alignat*}{3}
\inner{f}{g}_{L^{2}(\Omega, \xi_{\eps})} & := \int_{\Omega} \xi_{\eps} fg \dx = \int_{\Omega_{\eps}} \xi_{\eps} fg \dx, && \quad \norm{f}_{0,\xi_{\eps}}^{2} && := \inner{f}{f}_{L^{2}(\Omega, \xi_{\eps})}, \\
\inner{f}{g}_{H^{1}(\Omega, \xi_{\eps})} & := \int_{\Omega} \xi_{\eps} (fg + \nabla f \cdot \nabla g) \dx, && \quad \norm{f}_{1,\xi_{\eps}}^{2} && := \inner{f}{f}_{H^{1}(\Omega, \xi_{\eps})}.
\end{alignat*}
Here, we use the identification:
\begin{align*}
f = g \Leftrightarrow f(x) = g(x) \text{ a.e. in } \Omega_{\eps}.
\end{align*}
A similar notation and identification are used for $L^{2}(\Omega, \delta_{\eps})$ and $H^{1}(\Omega, \delta_{\eps})$.  Furthermore, we define
\begin{align*}
\mathcal{V}_{\eps} & := \{ f : \Omega \to \R \text{ measurable s.t. } f \mid_{\Omega_{\eps}} \in W^{1,2}(\Omega_{\eps}, \xi_{\eps}) \} \text{ with } \\
\inner{f}{g}_{\mathcal{V}_{\eps}} & := \int_{\Omega} \xi_{\eps} (fg + \nabla f \cdot \nabla g) + \delta_{\eps} fg \dx, 
\end{align*}
and
\begin{align*}
\mathcal{W}_{\eps} & := \{ f : \Omega \to \R \text{ measurable s.t. } f \mid_{\Omega_{\eps}} \in W^{1,2}(\Omega_{\eps}, \xi_{\eps}) \} \text{ with } \\
\inner{f}{g}_{\mathcal{W}_{\eps}} & := \int_{\Omega} \xi_{\eps} (fg + \nabla f \cdot \nabla g) + \frac{1}{\eps}\delta_{\eps} fg \dx.
\end{align*}
\end{defn}

\begin{remark}
\label{rem:differentspacesdifferentinnerproduct}
We observe that $H^{1}(\Omega, \xi_{\eps})$, $\mathcal{V}_{\eps}$ and $\mathcal{W}_{\eps}$ are the same vector spaces but with different inner products and induced norms.  It will turn out that the well-posedness to $(\mathrm{RDD})$, $(\mathrm{DDDH})$ and $(\mathrm{NDDH})$ depend critically on the choice of the inner product.
\end{remark}

\subsection{Results on the diffuse domain (DD) problems}
Similar to the above, we introduce the following bilinear forms $a_{B}^{\eps} : H^{1}(\Omega, \xi_{\eps}) \times H^{1}(\Omega, \xi_{\eps}) \to \R$, $a_{S}^{\eps} : H^{1}(\Omega, \delta_{\eps}) \times H^{1}(\Omega, \delta_{\eps}) \to \R$, $l_{B}^{\eps} : L^{2}(\Omega, \xi_{\eps}) \times L^{2}(\Omega, \xi_{\eps}) \to \R$, and $l_{S}^{\eps} : L^{2}(\Omega, \delta_{\eps}) \times L^{2}(\Omega, \delta_{\eps}) \to \R$:
\begin{alignat}{4}
\label{bulkbilinear:eps} a_{B}^{\eps}(\varphi, \psi) & := \int_{\Omega} \xi_{\eps} (\mathcal{A}^{Ea} \nabla \varphi \cdot \nabla \psi + a^{Ea} \varphi \psi) \dx, && \quad l_{B}^{\eps}(\varphi, \psi ) && := \int_{\Omega} \xi_{\eps} \varphi \psi \dx,  \\
\label{surfacebilinear:eps} a_{S}^{\eps}(\varphi, \psi) &:= \int_{\Omega} \delta_{\eps}(\mathcal{B}^{Ec} \nabla \varphi \cdot \nabla \psi + b^{Ec} \varphi \psi ) \dx, && \quad l_{S}^{\eps}(\varphi, \psi ) &&:= \int_{\Omega} \delta_{\eps} \varphi \psi \dx.
\end{alignat}
For $(\mathrm{NDDH})$, we replace $\mathcal{A}^{Ea}$ in $a_{B}^{\eps}(\cdot, \cdot)$ with $\mathcal{A}^{Er}$.

Our first main result is the well-posedness of the diffuse domain approximations:

\begin{thm}[Well-posedness for the DD problems]
Let $\Omega \subset \R^{n}$ be an open bounded domain with Lipschitz boundary.  Suppose Assumptions \ref{assump:bulkdata}, \ref{assump:surfacedata}, \ref{assump:Xi} and \ref{assump:delta} are satisfied.  In addition, for $(\mathrm{CDD})$ we assume that $\theta_{3} \geq K$.  Then, for each $\eps > 0$, there exist unique weak solutions
\begin{alignat*}{3}
(u^{\eps},v^{\eps})  \in \mathcal{V}_{\eps} \times H^{1}(\Omega, \delta_{\eps}) & \text{ for } (\mathrm{CDD}), && \quad v_{S}^{\eps}  \in H^{1}(\Omega, \delta_{\eps}) && \text{ for } (\mathrm{SDD}) \\
 u_{R}^{\eps}  \in \mathcal{V}_{\eps} & \text{ for } (\mathrm{RDD}), && \quad w_{N}^{\eps} \in H^{1}(\Omega, \xi_{\eps}) && \text{ for } (\mathrm{NDDH}), \\
w_{D}^{\eps} \in \mathcal{W}_{\eps} & \text{ for } (\mathrm{DDDH}), && \quad \, &&
\end{alignat*}
such that for all $\varphi \in H^{1}(\Omega, \xi_{\eps})$, $\psi \in H^{1}(\Omega, \delta_{\eps})$, $\phi \in \mathcal{V}_{\eps}$, $\Phi \in \mathcal{W}_{\eps}$, 
\begin{align}
 a_{B}^{\eps}(u^{\eps}, \phi) + a_{S}^{\eps}(v^{\eps}, \psi) + K l_{S}^{\eps}(v^{\eps} - u^{\eps}, \psi - \phi) & = l_{B}^{\eps}(f^{Ea}, \phi) + \beta l_{S}^{\eps}(g^{Ec}, \psi), \label{weakform:CDD} \\
a_{S}^{\eps}(v_{S}^{\eps}, \psi) & = l_{S}^{\eps}(g^{Ec}, \psi), \label{weakform:SDD} \\
a_{B}^{\eps}(u_{R}^{\eps}, \phi) + \beta l_{S}^{\eps}(u_{R}^{\eps}, \phi) & = l_{B}^{\eps}(f^{Ea}, \phi) + \beta l_{S}^{\eps}(g^{Ec}, \phi), \label{weakform:RDD} \\
a_{B}^{\eps}(w_{D}^{\eps} + \tilde{g}^{Er}, \Phi) + \frac{1}{\eps} l_{S}^{\eps}(w_{D}^{\eps}, \Phi) & = l_{B}^{\eps}(f^{Ea}, \Phi), \label{weakform:DDDH} \\
a_{B}^{\eps}(w_{N}^{\eps}, \varphi) - l_{B}^{\eps}(\nabla \cdot (\mathcal{A}^{Er} \nabla h^{Er}) - a^{Ea} h^{Er}, \varphi) & = l_{B}^{\eps}(f^{Ea} , \varphi). \label{weakform:NDDH}
\end{align}
Moreover, the weak solutions satisfy
\begin{align}
\norm{u^{\eps}}_{1,\xi_{\eps}}^{2} + \norm{u^{\eps}}_{0,\delta_{\eps}}^{2}+ \norm{v^{\eps}}_{1,\delta_{\eps}}^{2} & \leq C(\norm{f^{Ea}}_{L^{2}(\Omega)}^{2} + \norm{g}_{L^{2}(\Gamma)}^{2}), \label{apriori:coupledbulk} \\
\norm{v_{S}^{\eps}}_{1,\delta_{\eps}}^{2} & \leq C\norm{g}_{L^{2}(\Gamma)}^{2}, \label{apriori:surface} \\
\norm{u_{R}^{\eps}}_{1,\xi_{\eps}}^{2} + \norm{u_{R}^{\eps}}_{0,\delta_{\eps}}^{2} & \leq C(\norm{f^{Ea}}_{L^{2}(\Omega)}^{2} + \norm{g}_{L^{2}(\Gamma)}^{2}), \label{apriori:Robin} \\
\norm{w_{D}^{\eps}}_{1,\xi_{\eps}}^{2} + \norm{w_{D}^{\eps}}_{0,\tfrac{1}{\eps}\delta_{\eps}}^{2} & \leq C(\norm{f^{Ea}}_{L^{2}(\Omega)}^{2} + \norm{\tilde{g}^{Er}}_{H^{1}(\Omega)}^{2}) \label{apriori:Dirichlet}, \\
\norm{w_{N}^{\eps}}_{1,\xi_{\eps}}^{2} & \leq C(\norm{f^{Ea}}_{L^{2}(\Omega)}^{2} + \norm{h^{Er}}_{H^{2}(\Omega)}^{2}), \label{apriori:Neumann}
\end{align}
where the constants $C$ are independent of $\eps$.
\end{thm}

We point out that, thanks to the fact that $g^{Ec}$ is the constant extension of $g$ in the normal direction, the estimates in (\ref{apriori:coupledbulk}), (\ref{apriori:surface}) and (\ref{apriori:Robin}) are independent of $\eps$.

\subsection{Convergence results}
Recall that the unique weak solutions to (CSI), (SSI), (RSI), (DSIH) and (NSIH) are denoted by $(u,v)$, $v_{S}$, $u_{R}$, $w_{D}$ and $w_{N}$, respectively.  Keeping the notation that the unique weak solutions to (CDD), (SDD), (RDD), (DDDH) and (NDDH) are denoted by $(u^{\eps}, v^{\eps})$, $v_{S}^{\eps}$, $u_{R}^{\eps}$, $w_{D}^{\eps}$, and $w_{N}^{\eps}$, our second main result in the weak convergence of diffuse domain approximations:

\begin{thm}[Weak convergence of diffuse domain approximation]\label{thm:weakconv}
Suppose Assumptions \ref{assump:Gamma}, \ref{assump:General}, \ref{assump:DSI}, \ref{assump:NSI}, \ref{assump:bulkdata}, \ref{assump:surfacedata}, \ref{assump:Xi}, and \ref{assump:delta} are satisfied.  In addition, for $(\mathrm{CDD})$, we assume $\theta_{3} \geq K$.  Then, as $\eps \to 0$,
\begin{align*}
u^{\eps} \mid_{\Omega^{*}}, \; u^{\eps}_{R} \mid_{\Omega^{*}}, \; w^{\eps}_{D} \mid_{\Omega^{*}}, \; w^{\eps}_{N} \mid_{\Omega^{*}} \text{  converge weakly to } u, \; u_{R}, \; w_{D}, \; w_{N} \text{ in } H^{1}(\Omega^{*}),
\end{align*}
respectively.  While, for any $\psi \in H^{1}(\Gamma)$ with extension $\psi^{Ec} \in H^{1}(\Omega)$ (as constructed in Corollary \ref{cor:constextH1} below), as $\eps \to 0$, 
\begin{alignat*}{3}
\int_{\Omega} \delta_{\eps} v^{\eps} \psi^{Ec} \dx & \to \int_{\Gamma} v \psi \dHaus, && \quad \int_{\Omega} \delta_{\eps} \nabla v^{\eps} \cdot \nabla \psi^{Ec} \dx & \to \int_{\Gamma} \surf v \cdot \surf \psi \dHaus, \\
\int_{\Omega} \delta_{\eps} v^{\eps}_{S} \psi^{Ec} \dx & \to \int_{\Gamma} v_{S} \psi \dHaus, && \quad \int_{\Omega} \delta_{\eps} \nabla v_{S}^{\eps} \cdot \nabla \psi^{Ec} \dx & \to \int_{\Gamma} \surf v_{S} \cdot \surf \psi \dHaus.
\end{alignat*}
\end{thm}

Let
\begin{align}\label{defn:CGamma}
\mathcal{C}_{\Gamma} :=  \left \{ \psi^{Ec} \in H^{1}(\Omega) : \psi^{Ec} \text{ constructed from } \psi \in H^{1}(\Gamma) \text{ as in Corollary } \ref{cor:constextH1} \right \}.
\end{align}
Then, one can formally interpret the above weak convergence of the surface quantities in the following way:   
\begin{align*}
\delta_{\eps} v^{\eps} \rightharpoonup \delta_{\Gamma} v, \quad \delta_{\eps} \nabla v^{\eps} \rightharpoonup \delta_{\Gamma} \surf v \text{ in } \mathcal{C}_{\Gamma}' \text{ as } \eps \to 0.
\end{align*}
The precise definition of the above notion of weak convergence for the surface quantities can be found in the proof of Lemma \ref{lem:H1deltabddconv}.  In fact, in Lemma \ref{lem:H1deltabddconv}, the weak convergence for the lower order term holds for all $H^{1}(\Omega)$ functions:  As $\eps \to 0$,
\begin{align*}
\int_{\Omega} \delta_{\eps} v^{\eps} \varphi \dx \to \int_{\Gamma} v \gamma_{0}(\varphi) \dHaus \text{ for all } \varphi \in H^{1}(\Omega).
\end{align*}

Thanks to the coercivity of the bilinear forms $a_{B}^{\eps}$ and $a_{S}^{\eps}$, we also obtain strong convergence:

\begin{thm}[Strong convergence of diffuse domain approximations]\label{thm:strongconv}
Suppose Assumptions \ref{assump:Gamma}, \ref{assump:General}, \ref{assump:DSI}, \ref{assump:NSI}, \ref{assump:bulkdata}, \ref{assump:surfacedata}, \ref{assump:Xi}, and \ref{assump:delta} are satisfied.  In addition, for $(\mathrm{CDD})$ we assume $\theta_{3} \geq K$.

Let $v^{Ec}, v_{S}^{Ec} \in H^{1}(\Omega)$ denote the constant extensions of $v, v_{S} \in H^{1}(\Gamma)$ in the normal direction, as constructed in Corollary \ref{cor:constextH1}.  Let $u^{Er}, u^{Er}_{R} \in H^{1}(\Omega)$ denote the extensions of $u, u_{R} \in H^{1}(\Omega^{*})$ by reflection.  Then, as $\eps \to 0$,
\begin{align*}
\norm{u^{\eps} \mid_{\Omega^{*}} - u}_{H^{1}(\Omega^{*})} + \norm{u^{\eps} - u^{Er}}_{0,\delta_{\eps}} + \norm{v^{\eps} - v^{Ec}}_{1,\delta_{\eps}} & \to 0, \\
\norm{v_{S}^{\eps} - v_{S}^{Ec}}_{1,\delta_{\eps}} & \to 0, \\
\norm{u^{\eps}_{R} \mid_{\Omega^{*}} - u_{R}}_{H^{1}(\Omega^{*})} + \norm{u^{\eps}_{R} - u_{R}^{Er}}_{0,\delta_{\eps}} & \to 0, \\
\norm{w^{\eps}_{N} \mid_{\Omega^{*}} - w_{N}}_{H^{1}(\Omega^{*})} & \to 0, \\
\norm{w^{\eps}_{D} \mid_{\Omega^{*}} - w_{D}}_{H^{1}(\Omega^{*})} & \to 0.
\end{align*}
Consequently, we also have the norm convergence: As $\eps \to 0$,
\begin{alignat*}{3}
\norm{u^{\eps}}_{0,\delta_{\eps}} &\to \norm{\gamma_{0}(u)}_{L^{2}(\Gamma)}, &&\quad \norm{v^{\eps}}_{1,\delta_{\eps}} &&\to \norm{v}_{H^{1}(\Gamma)},\\
\norm{v^{\eps}_{S}}_{1,\delta_{\eps}} &\to \norm{v_{S}}_{H^{1}(\Gamma)}, && \quad \norm{u^{\eps}_{R}}_{0,\delta_{\eps}} &&\to \norm{\gamma_{0}(u_{R})}_{L^{2}(\Gamma)}.
\end{alignat*}
\end{thm}

In the next section, we present some technical results that will be essential to prove the main results.

\section{Technical results}\label{sec:Technical}
\subsection{Change of variables in the tubular neigbourhood}\label{sec:TubNbd}
Choose $\eta > 0$ and the diffeomorphism $\Theta^{\eta} : \mathrm{Tub}^{\eta}(\Gamma) \to \Gamma \times (-\eta, \eta)$ defined in (\ref{diffeo}).  For $t \in (-\eta, \eta)$, let $\Gamma_{t}$ denote the level set $\{ x \in \Omega : d(x) = t \}$.  Then, by the co-area formula \cite[Theorem 5, p. 629]{book:Evans} or \cite[Theorem 2, p. 117]{book:EvansGariepy}, and that $\abs{\nabla d} = \abs{\nu} = 1$, we can write
\begin{align}\label{TubGammadecompose}
\int_{\mathrm{Tub}^{\eta}(\Gamma)} f(x) \abs{\nabla d(x)} \dx = \int_{\mathrm{Tub}^{\eta}(\Gamma)} f(x) \dx = \int_{-\eta}^{\eta} \int_{\Gamma_{t}} f \dHaus \dt.
\end{align}

We define the mapping $\rho_{t} : \Gamma \to \Gamma_{t}$ by
\begin{align}\label{maprhot}
\rho_{t}(p) = p + t \nu(p) \text{ for } p \in \Gamma.
\end{align}
This map is well-defined and is injective due to the diffeomorphism $\Theta^{\eta}$.  Then, by a change of variables, we obtain
\begin{align*}
\int_{\Gamma_{t}} f \dHaus = \int_{\Gamma} f(p + t \nu(p)) \abs{\det ((\nabla \rho_{t})^{T}(\nabla \rho_{t}))}^{\tfrac{1}{2}} \dHaus,
\end{align*}
where $\nabla \rho_{t}$ is the Jacobian matrix of $\rho_{t}$.  To identify $\det ((\nabla \rho_{t})^{T} (\nabla \rho_{t}))$ as a function of $t$ we use local coordinates.

Since $\Gamma$ is a compact hypersurface, we can always find a finite open cover of $\Gamma$ consisting of open sets $W_{i} \subset \R^{n}$, $1 \leq i \leq N$ such that $\Gamma \subset \bigcup_{i=1}^{N} W_{i}$.  For each $1 \leq i \leq N$, let $\alpha_{i}(s)$ denote a regular parameterisation of $W_{i} \cap \Gamma$ with parameter domain $\mathcal{S}_{i} \subset \R^{n-1}$, i.e, $\alpha_{i} : \mathcal{S}_{i} \to W_{i} \cap \Gamma$ is a local regular parameterisation of $\Gamma$ (the existence of such local regular parameterisations follows from the regularity of $\Gamma$).  Let
\begin{align*}
J_{i,0}(s) := (\pd_{s_{1}} \alpha_{i}(s), \dots, \pd_{s_{n-1}} \alpha_{i}(s), \nu(\alpha_{i}(s)) \in \R^{n \times n}.
\end{align*}
Since $\alpha_{i}$ is a regular parameterisation, the tangent vectors $\{ \pd_{s_{j}} \alpha_{i} \}_{1 \leq j \leq n-1}$ are linearly independent and hence the determinant $\det J_{i,0} \neq 0$.  Then, for any $f \in L^{1}_{loc}(\Gamma)$,
\begin{align}\label{integralf:transform}
\int_{W_{i} \cap \Gamma} f \dHaus = \int_{\mathcal{S}_{i}} f(\alpha_{i}(s)) \abs{\det J_{i,0}} \ds.
\end{align}
By the injectivity of $\rho_{t}$, $\Gamma_{t}$ is also a compact hypersurface with a finite open cover $\{\rho_{t}(W_{i} \cap \Gamma)\}_{i=1}^{N}$.  In addition, $\rho_{t} \circ \alpha_{i}$ is a local parameterisation of $\Gamma_{t}$.  Let
\begin{align*}
J_{i,\eta}(s,t) & := (\pd_{s_{1}} \alpha_{i}(s) + t \pd_{s_{1}} \nu(\alpha_{i}(s)), \dots, \pd_{s_{n-1}} \alpha_{i}(s) + t \pd_{s_{n-1}} \nu(\alpha_{i}(s)), \nu(\alpha_{i}(s)) \in \R^{n \times n} \\
& = J_{i,0}(s) + t (\pd_{s_{1}} \nu(\alpha_{i}(s)), \dots, \pd_{s_{n-1}} \nu(\alpha_{i}(s)), 0) =: J_{i,0}(s) + t B_{i}(s).
\end{align*}
A short calculation shows that
\begin{align*}
\det J_{i,\eta} &  = \det (J_{i,0} + t B_{i}) = (\det J_{i,0}) (\det (I + t J_{i,0}^{-1} B_{i})) = (\det J_{i,0}) t^{n} \det \left (t^{-1} I + J_{i,0}^{-1} B_{i} \right ) \\
& = (\det J_{i,0})t^{n} \left ( \frac{1}{t^{n}} - \frac{1}{t^{n-1}} \tr{-J_{i,0}^{-1} B_{i}} + \cdots + (-1)^{n} \det (-J_{i,0}^{-1}B_{i}) \right ) \\
& = (\det J_{i,0}) \left (1 + \tr{t J_{i,0}^{-1} B_{i}} + \cdots + (-1)^{2n} \det (t J_{i,0}^{-1} B_{i })  \right ),
\end{align*}
where $\tr{A}$ denotes the trace of a matrix $A$, and we used the well-known fact that the coefficients of the monic characteristic polynomial
\begin{align*}
\det (xI - A) = p_{A}(x) = x^{n} + a_{n-1} x^{n-1} + \cdots + a_{1}x + a_{0},
\end{align*}
are given by
\begin{align*}
a_{n-k} = (-1)^{k} \sum_{1 \leq j_{1} < \cdots < j_{k} \leq n} \lambda_{j_{1}} \lambda_{j_{2}} \dots \lambda_{j_{k}}, \quad k = 1, \dots, n-1, \quad a_{0} = (-1)^{n} \det (A),
\end{align*}
where $\{ \lambda_{j}\}_{j=1}^{n}$ are the eigenvalues of $A$ (see \cite{article:Brooks06}).  We define
\begin{align*}
C_{i,H}(s):= \tr{J_{i,0}^{-1}(s) B_{i}(s)}, \quad C_{i,R}(s,t) := t^{n} p_{-J_{i,0}^{-1}B_{i}}(1/t) - 1 - t C_{i,H}(s),
\end{align*}
so that
\begin{align*}
\det J_{i,\eta}(s,t) = (\det J_{i,0}(s)) (1 + t C_{i,H}(s) + C_{i,R}(s,t)).
\end{align*}

Consequently, we have
\begin{equation}\label{integralf:Gammat:transform}
\begin{aligned}
& \; \int_{\rho_{t}(W_{i} \cap \Gamma)} f \dHaus = \int_{\mathcal{S}_{i}} f(\alpha_{i}(s) + t \nu(\alpha_{i}(s))) \abs{\det J_{i,\eta}(s,t)} \ds \\
= & \;  \int_{\mathcal{S}_{i}} f(\alpha_{i}(s) + t \nu(\alpha_{i}(s))) \abs{\det J_{i,0}(s)} \abs{1 + t C_{i,H}(s) + C_{i,R}(s,t)} \ds.
\end{aligned}
\end{equation}

By Assumption $\ref{assump:Gamma}$, $\Gamma$ is a compact $C^{2}$ hypersurface, and so the eigenvalues of $J_{i,0}(s)$ and $J_{i,0}^{-1}B_{i}(s)$ are bounded  uniformly in $s \in \mathcal{S}_{i}$.  Hence, for each $1 \leq i \leq N$, there exists a constant $c_{i}$ such that 
\begin{align*}
\abs{\det J_{i,\eta}(s,t) - \det J_{i,0}(s)} = \abs{\det J_{i,0}(s)} \abs{t C_{i,H}(s) + C_{R}(s,t)} \leq c_{i} \abs{t}.
\end{align*}
Moreover, by the compactness of $\Gamma$, there are only a finite number of $c_{i}$, and so we can deduce that there exists a constant $\tilde{c}$ (that can be chosen independent of $\eta$) such that, for all $t \in (-\eta, \eta), s \in \mathcal{S}_{i}, 1 \leq i \leq N$, 
\begin{align*}
\max_{1 \leq i \leq N} \abs{\det J_{i,\eta}(s,t) - \det J_{i,0}(s)} \leq \tilde{c} \abs{t}.
\end{align*}

Let $\{ \mu_{i}\}_{i=1}^{N}$ be a partition of unity subordinate to the covering $\{W_{i} \cap \Gamma\}_{i=1}^{N}$ of $\Gamma$.  Consequently, by the diffeomorphism $\Theta^{\eta}$, we observe that $\{ \mu_{i} \circ \rho_{t}^{-1}\}_{i=1}^{N}$ is a partition of unity subordinate to the covering $\{ \rho_{i} (W_{i} \cap \Gamma)\}_{i=1}^{N}$ of $\Gamma_{t}$ for $t \in (-\eta, \eta)$.  We define
\begin{align*}
C_{H}(p) := \sum_{i=1}^{N} \mu_{i}(p) C_{i,H}(\alpha_{i}^{-1}(p)), \quad C_{R}(p,t) := \sum_{i=1}^{N} \mu_{i}(p) C_{i,R}(\alpha_{i}^{-1}(p),t), \text{ for } p \in \Gamma.
\end{align*}

From the above discussion, we observe that $C_{H}(p)$ and $C_{R}(p,t)$ are uniformly bounded in $p \in \Gamma$ and
\begin{align*}
\abs{t C_{H}(p)  + C_{R}(p,t)} \leq \tilde{c} \abs{t} \text{ for all } \abs{t} < \eta.
\end{align*}
Then, from (\ref{integralf:transform}), for any $f \in L^{1}(\Gamma)$ we have
\begin{align*}
\int_{\Gamma} f \dHaus = \sum_{i=1}^{N} \int_{W_{i} \cap \Gamma} \mu_{i} f \dHaus = \sum_{i=1}^{N} \int_{\mathcal{S}_{i}} (\mu_{i} f)(\alpha_{i}(s)) \abs{\det J_{i,0}(s)} \ds,
\end{align*}
and similarly from (\ref{integralf:Gammat:transform}), for any $f \in L^{1}(\mathrm{Tub}^{\eta}(\Gamma))$,
\begin{align*}
& \; \int_{\mathrm{Tub}^{\eta}(\Gamma)} f(x) \dx = \int_{-\eta}^{\eta} \int_{\Gamma_{t}} f \dHaus \dt = \int_{-\eta}^{\eta} \sum_{i=1}^{N} \int_{\rho_{t}(W_{i} \cap \Gamma)} (\mu_{i} \circ \rho_{t}^{-1}) f \dHaus \dt \\
= & \; \int_{-\eta}^{\eta} \sum_{i=1}^{N} \int_{\mathcal{S}_{i}} \mu_{i}(\alpha_{i}(s)) f(\alpha_{i}(s) + t \nu(\alpha_{i}(s))) \abs{\det J_{i,0}(s)} \abs{1 + t C_{i,H}(s) + C_{i,R}(s,t)} \ds \dt \\
= & \; \int_{-\eta}^{\eta} \sum_{i=1}^{N} \int_{W_{i} \cap \Gamma} \mu_{i}(p) f(p + t \nu(p)) \abs{ 1 + t C_{H}(p) + C_{R}(p,t)} \dHaus (p) \dt \\
= & \; \int_{-\eta}^{\eta} \int_{\Gamma} f(p + t \nu(p)) \abs{1 + t C_{H}(p) + C_{R}(p,t)} \dHaus (p) \dt.
\end{align*}

Hence, we can identify 
\begin{align*}
\abs{ \det ((\nabla \rho_{t})^{T} (\nabla \rho_{t}))}^{\frac{1}{2}} (p,t) = \abs{1 + t C_{H}(p) + C_{R}(p,t)}.
\end{align*}
We summarise our findings of this section in the following:
\begin{lemma}
Suppose Assumption \ref{assump:Gamma} is satisfied.  There exists $\tilde{c} > 0$ such that for any $\eta < (\tilde{c})^{-1}$, and any $f \in L^{1}(\mathrm{Tub}^{\eta}(\Gamma))$, we have
\begin{align}\label{TubNbd:intf:change}
\int_{\mathrm{Tub}^{\eta}(\Gamma)} f(x) \dx = \int_{-\eta}^{\eta} \int_{\Gamma} f(p + t \nu(p)) \abs{1 + t C_{H}(p) + C_{R}(p,t)} \dHaus \dt,
\end{align}
where
\begin{align}\label{CHCRbdd}
\abs{t C_{H}(p) + C_{R}(p,t)} \leq \tilde{c} \abs{t} \text{ for all } \abs{t} < \eta.
\end{align}
Consequently,
\begin{equation}
\label{TubNbd:intf:upplowbdd}
\begin{aligned}
\frac{1}{1+ \tilde{c} \eta} \int_{\mathrm{Tub}^{\eta}(\Gamma)} \abs{f(x)} \dx & \leq \int_{-\eta}^{\eta} \int_{\Gamma} \abs{f(p + t \nu(p))} \dHaus \dt \\
& \leq \frac{1}{1-\tilde{c} \eta} \int_{\mathrm{Tub}^{\eta}(\Gamma)} \abs{f(x)} \dx.
\end{aligned}
\end{equation}
\end{lemma}

\subsection{Coordinates in a scaled tubular neighbourhood}\label{sec:ScaledNbd}
In the subsequent convergence analysis, we will use a tubular neighbourhood whose width scales with $\eps^{k}$ for some $0 < k \leq 1$, i.e., we consider $X^{\eps} := \mathrm{Tub}^{\eps^{k} \eta}(\Gamma)$.  For this section, we take $X^{\eps} = \mathrm{Tub}^{\eps \eta}(\Gamma)$, i.e., $k = 1$, to derive some technical results.

Let $W \subset \R^{n}$ be one of the open sets in a finite open cover of $\Gamma$, with associated local regular parameterisation $\alpha : \mathcal{S} \to W \cap \Gamma$ and parameter domain $\mathcal{S} \subset \R^{n-1}$.  We define the metric tensor $\mathcal{G}_{0} = (g_{0,ij})_{1 \leq i,j \leq n-1}$ by
\begin{align}\label{metrictensor:Gamma}
g_{0,ij}(p) := \pd_{s_{i}}\alpha(s) \cdot \pd_{s_{j}} \alpha(s) \text{ for } 1 \leq i,j \leq n-1, \text{ if } p = \alpha(s) \in \Gamma \cap W,
\end{align}
with inverse metric tensor $\mathcal{G}_{0}^{-1} = (g_{0}^{ij})_{1 \leq i,j \leq n-1}$.  Then, for a scalar function $f: \Gamma \to \R$, the surface gradient of $f$ on $\Gamma$ at a point $p = \alpha(s) \in \Gamma \cap W$ for some $s \in \mathcal{S}$ is defined as
\begin{align}\label{surfGradient}
\surf f(\alpha(s)) = \sum_{j=1}^{n-1} \left ( \sum_{i=1}^{n-1} g_{0}^{ij}(\alpha(s)) \pd_{s_{i}} f(\alpha(s)) \right ) \pd_{s_{j}} \alpha(s).
\end{align}
One can show that the definition (\ref{surfGradient}) does not depend on the parameterisation.  There is also an alternate and equivalent definition of the surface gradient in terms of level sets, we refer the reader to \cite[\S 2.2]{article:DziukElliott13} for more details.

Recall the signed distance function $d$ defined in (\ref{defn:signeddist}).  We introduce the rescaled distance variable
\begin{align*}
z = \frac{d}{\eps}.
\end{align*}
For $z \in (-\eta, \eta)$, we define a parallel hypersurface at distance $\eps z$ away from $\Gamma$ as
\begin{align}\label{Gammaepsz}
\Gamma^{\eps z} := \{ p + \eps z \nu(p) : p \in \Gamma \}.
\end{align}
Let $p(x)$ denote the closest point operator of $x \in X^{\eps}$ as defined in (\ref{closestpoint}), such that
\begin{align}\label{scaledNbd:decomposition}
x = p(x) + \eps z \nu(p(x)) \text{ for some } z \in (-\eta, \eta).
\end{align}
Then, by the injectivity of the closest point operator, we have
\begin{align}\label{samenormalGamma}
\nu(y) = \nu(p(y)) \text{ for } y \in \Gamma^{\eps z}.
\end{align}

For any scalar function $f: X^{\eps} \to \R$, we define its representation $F_{\eps}(p,z)$ in the $(p,z)$ coordinate system by
\begin{align}\label{representation:pz}
F_{\eps}(p,z) := f(p + \eps z \nu(p)) \text{ for } p \in \Gamma, z \in (-\eta, \eta).
\end{align}

Locally, for $x \in X^{\eps} \cap W$, we have 
\begin{align}\label{Geps}
x = G_{\eps}(s,z) := \alpha(s) + \eps z \nu(\alpha(s)) \text{ for some } s \in \mathcal{S}, z \in (-\eta, \eta).
\end{align}
Then, we can define the local representation $\tilde{F}_{\eps}(s,z)$ of $f$ in the $(s,z)$ coordinate system by
\begin{align}\label{representation:sz}
\tilde{F}_{\eps}(s,z) := F_{\eps}(\alpha(s), z) = f(\alpha(s) + \eps z \nu(\alpha(s)))  \text{ for } s \in \mathcal{S}, z \in (-\eta, \eta).
\end{align}

Let $(s_{1}, \dots, s_{n-1}) \in \mathcal{S}$ and $s_{n} := z$.  Then, by (\ref{Geps}),
\begin{align}\label{derivativeGeps}
\pd_{s_{i}} G_{\eps}(s,z) = \pd_{s_{i}} \alpha(s) + \eps z \pd_{s_{i}} \nu(\alpha(s)) \text{ for } 1 \leq i \leq n-1, \quad \pd_{s_{n}} G_{\eps} = \eps \nu(\alpha(s)),
\end{align}
and $\{ \pd_{s_{i}} G_{\eps} \}_{i=1}^{n}$ is a basis of $\R^{n}$ locally around $\Gamma^{\eps z}$.  For $1 \leq i, j \leq n-1$, we define the metric tensor $\mathcal{G}_{\eps z} = (g_{\eps,ij})_{1 \leq i,j \leq n}$ in these new coordinates as
\begin{equation}\label{metrictensor:eps}
\begin{aligned}
g_{\eps,ij}(p,z) & = (\pd_{s_{i}} \alpha(s) + \eps z \pd_{s_{i}} \nu(\alpha(s))) \cdot (\pd_{s_{j}} \alpha(s) + \eps z \pd_{s_{j}} \nu(\alpha(s))), \\
g_{\eps,in}(p,z) & = g_{\eps,ni}(p,z) = (\pd_{s_{i}} \alpha(s) + \eps z \pd_{s_{i}} \nu(\alpha(s))) \cdot \eps \nu(\alpha(s)) = 0, \\
 g_{\eps, nn}(p,z) & = (\eps \nu(\alpha(s)) ) \cdot (\eps \nu(\alpha(s))) = \eps^{2}, 
\end{aligned}
\end{equation}
for $p = \alpha(s) \in \Gamma \cap W$, and we have used that $\pd_{s_{i}} \nu \cdot \nu = \tfrac{1}{2} \pd_{s_{i}} \abs{\nu}^{2} = 0$.

Let $\mathcal{G}^{-1}_{\eps z} = (g_{\eps}^{ij})_{1 \leq i,j \leq n}$ denote the inverse metric tensor.  Then, for any scalar function $f: X^{\eps} \to \R$, with the representation $\tilde{F}_{\eps}(s,z)$ in the $(s,z)$ coordinate system, we can express the surface gradient $\surfzloc \tilde{F}_{\eps}(s,z)$ on the parallel hypersurface $\Gamma^{\eps z}$ in local coordinates as
\begin{align}\label{surfGradientepsztildeFeps}
\surfzloc \tilde{F}_{\eps}(s,z) := \sum_{j=1}^{n-1} \left ( \sum_{i=1}^{n-1} g_{\eps}^{ij}(\alpha(s),z) \pd_{s_{i}} \tilde{F}_{\eps}(s,z) \right ) \pd_{s_{j}} G_{\eps}(s,z).
\end{align}
Similarly, we can define the surface gradient $\surfloc \tilde{F}_{\eps}(s,z)$ on the original hypersurface $\Gamma$ as
\begin{align}
\label{surfGradienttildeFeps}
\surfloc \tilde{F}_{\eps}(s,z) = \sum_{j=1}^{n-1} \left ( \sum_{i=1}^{n-1} g_{0}^{ij}(\alpha(s)) \pd_{s_{i}} \tilde{F}_{\eps}(s,z) \right ) \pd_{s_{j}} \alpha(s).
\end{align}

Using (\ref{representation:pz}) and (\ref{representation:sz}) to switch to the $(p,z)$ coordinate system, we can define the surface gradient $\surfz F_{\eps}(p,z)$ on the parallel hypersurface $\Gamma^{\eps z}$ via the formula
\begin{align}\label{surfGradientepszFeps}
\surfz F_{\eps}(p,z) := \surfzloc \tilde{F}_{\eps}(\alpha^{-1}(p),z) \text{ for } p = \alpha(s) \in \Gamma \cap W,
\end{align}
and the surface gradient $\surf F_{\eps}(p,z)$ on the original hypersurface $\Gamma$ as
\begin{align}\label{surfGradientFeps}
\surf F_{\eps}(p,z) := \surfloc \tilde{F}_{\eps}(\alpha^{-1}(p),z) \text{ for } p = \alpha(s) \in \Gamma \cap W.
\end{align}

It is convenient to define the remainder both in local and global representations:
\begin{equation}
\label{defn:nablaepsz}
\begin{aligned}
\tilde{\nabla}_{\eps z} \tilde{F}_{\eps}(s,z) & := \surfzloc \tilde{F}_{\eps}(s,z) - \surfloc \tilde{F}_{\eps}(s,z), \\
\nabla_{\eps z} F_{\eps}(p,z) & := \surfz F_{\eps}(p,z) - \surf F_{\eps}(p,z),
\end{aligned}
\end{equation}
with the relation
\begin{align*}
\nabla_{\eps z} F_{\eps}(p,z) = \tilde{\nabla}_{\eps z} \tilde{F}_{\eps}(\alpha^{-1}(p),z) \text{ for } p = \alpha(s) \in \Gamma \cap W.
\end{align*}

We now state a result that decomposes the Euclidean gradient $\nabla f$ into a surface component and a component in the normal direction:
\begin{lemma}\label{ScaledNbd:gradientdecomposition}
Suppose that Assumption \ref{assump:Gamma} is satisfied.  Let $f : X^{\eps} \to \R$ be a $C^{1}$ function with representation $F_{\eps}$ in the $(p,z)$ coordinate system and representation $\tilde{F}_{\eps}$ in the $(s,z)$ coordinate system, as defined in $(\ref{representation:pz})$ and $(\ref{representation:sz})$, respectively.  Then, for $p \in \Gamma$, $s \in \mathcal{S}$, $z \in (-\eta, \eta)$ such that $p = \alpha(s)$, and $x = \alpha(s) + \eps z \nu(\alpha(s)) = p + \eps z \nu(p) \in X^{\eps}$,
\begin{equation}\label{gradientf:decomposition}
\begin{aligned}
\nabla f(x) & = \frac{1}{\eps} \nu(p) \pd_{z} F_{\eps}(p,z) + \surfz F_{\eps}(p,z) \\
& = \frac{1}{\eps} \nu(\alpha(s)) \pd_{z} \tilde{F}_{\eps}(s,z) + \surfzloc \tilde{F}_{\eps}(s,z),
\end{aligned}
\end{equation}
where $\surfz(\cdot)$ and $\surfzloc (\cdot)$ are defined in $(\ref{surfGradientepszFeps})$ and $(\ref{surfGradientepsztildeFeps})$, respectively.

In addition, the remainders $\nabla_{\eps z} F_{\eps}$ and $\tilde{\nabla}_{\eps z} \tilde{F}_{\eps}$ defined in $(\ref{defn:nablaepsz})$ satisfy
\begin{align}
\nabla_{\eps z} F_{\eps}(p,z) \cdot \nu(p) = 0, & \quad \tilde{\nabla}_{\eps z} \tilde{F}_{\eps}(s,z) \cdot \nu(\alpha(s)) = 0, \label{gradientremainderdotnuzero} \\
\nabla_{\eps z} F_{\eps} (p,z) = \mathcal{O}(\eps), & \quad \tilde{\nabla}_{\eps z} \tilde{F}_{\eps} (s,z) = \mathcal{O}(\eps) \text{ as } \eps \to 0. \label{gradientremainderOrdereps}
\end{align}

\end{lemma}
\begin{proof}
For the first assertion, we follow the proof given in \cite[Appendix]{article:AbelsGarckeGrun11}.  The equivalent result in two dimensions can be found in \cite[Appendix B]{article:GarckeStinner06} and in \cite{article:ElliottStinner09}.

Let $f: X^{\eps} \to \R$ be a $C^{1}$ function.  Recall the metric tensor $\mathcal{G}_{\eps z} = (g_{\eps,ij})_{ 1\leq i,j \leq n}$ as defined in (\ref{metrictensor:eps}).  If we denote the submatrix $\tilde{\mathcal{G}}_{\eps z}$ and its inverse $\tilde{\mathcal{G}}_{\eps z}^{-1}$ by:
\begin{align*}
\tilde{\mathcal{G}}_{\eps z} = (g_{\eps,ij})_{1 \leq i,j \leq n-1}, \quad  \tilde{\mathcal{G}}_{\eps z}^{-1} = (g_{\eps}^{ij})_{1 \leq i, j \leq n-1}.
\end{align*}
Then, we observe that
\begin{align*}
\mathcal{G}_{\eps z} = \left ( \begin{array}{cc}
\tilde{\mathcal{G}}_{\eps z} & \bm{0} \\
\bm{0}^{T} & \eps^{2}
\end{array} \right ), \quad \mathcal{G}_{\eps z}^{-1} = \left ( \begin{array}{cc}
\tilde{\mathcal{G}}_{\eps z}^{-1} & \bm{0} \\
\bm{0}^{T} & \eps^{2}
\end{array} \right ),
\end{align*}
where $\bm{0} \in \R^{n-1}$ is the zero column vector of length $n-1$.  Moreover, for $f(x) = \tilde{F}_{\eps}(s(x),z(x))$, we have from (\ref{derivativeGeps}) and  (\ref{surfGradientepsztildeFeps}),
\begin{align*}
\nabla f(x) & = \sum_{j=1}^{n} \left ( \sum_{i=1}^{n} g_{\eps}^{ij} \pd_{s_{i}} \tilde{F}_{\eps} \right ) \pd_{s_{j}} G_{\eps} = \frac{1}{\eps^{2}} \pd_{z} \tilde{F}_{\eps} \pd_{z} G_{\eps} + \surfzloc \tilde{F}_{\eps} = \frac{1}{\eps} \pd_{z} \tilde{F}_{\eps} \nu(\alpha(s)) + \surfzloc \tilde{F}_{\eps}.
\end{align*}

For the assertion regarding the remainder, let
\begin{align*}
\mathcal{C}_{ij}(s) &:= \pd_{s_{i}} \nu(\alpha(s)) \cdot \pd_{s_{j}} \alpha(s) + \pd_{s_{i}} \alpha(s) \cdot \pd_{s_{j}} \nu(\alpha(s)), \\
\mathcal{D}_{ij}(s) &:= \pd_{s_{i}} \nu(\alpha(s)) \cdot \pd_{s_{j}} \nu(\alpha(s)),
\end{align*}
and
\begin{align*}
\mathcal{C} := (\mathcal{C}_{ij})_{1 \leq i,j \leq n-1}, \quad \mathcal{D} := (\mathcal{D}_{ij})_{1 \leq i,j \leq n-1},
\end{align*}
so that by (\ref{metrictensor:eps}) and (\ref{metrictensor:Gamma}),
\begin{align*}
g_{\eps,ij}(\alpha(s),z)  & = g_{0,ij}(\alpha(s)) + \eps z \mathcal{C}_{ij}(s) + (\eps z)^{2} \mathcal{D}_{ij}(s), \\
\tilde{\mathcal{G}}_{\eps z} & = \mathcal{G}_{0} + \eps z \mathcal{C} + (\eps z)^{2} \mathcal{D}.
\end{align*}
A calculation involving the ansatz 
\begin{align}\label{ansatz}
\tilde{\mathcal{G}}_{\eps z}^{-1} = (\mathcal{G}_{0} + \eps z \mathcal{C} + (\eps z)^{2} \mathcal{D} )^{-1} = \mathcal{G}_{0}^{-1} + \mathcal{E},
\end{align}
will yield that
\begin{align*}
\mathcal{E} = - (I + \mathcal{G}_{0}^{-1} (\eps z \mathcal{C} + (\eps z)^{2} \mathcal{D}))^{-1} (\mathcal{G}_{0}^{-1} (\eps z \mathcal{C} + (\eps z)^{2} \mathcal{D}) \mathcal{G}_{0}^{-1}),
\end{align*}
if $\tilde{\mathcal{G}}_{\eps z}$, $\mathcal{G}_{0}$ and $I +  \mathcal{G}_{0}^{-1} (\eps z \mathcal{C} + (\eps z)^{2} \mathcal{D})$ are invertible.  Here, $I$ denotes the identity matrix.

Since $\Gamma$ is a compact $C^{2}$ hypersurface, all entries in the matrices $\mathcal{G}_{0}$, $\mathcal{C}$, and $\mathcal{D}$ are bounded.  For a matrix $H$ and $\eps$ sufficiently small so that the absolute values of the eigenvalues of $\eps H$ are less than 1, we have,
\begin{align*}
(I + \eps H)^{-1} = I - \eps H + \eps^{2} H^{2} - \cdots.
\end{align*}
Hence, we can express
\begin{align*}
\tilde{\mathcal{G}}_{\eps z}^{-1} = \mathcal{G}_{0}^{-1} - \eps z \mathcal{G}_{0}^{-1} \mathcal{C} \mathcal{G}_{0}^{-1} + \mathcal{O}(\eps^{2}) \text{ as } \eps \to 0.
\end{align*}
Consequently, looking at (\ref{surfGradientepsztildeFeps}), and (\ref{surfGradienttildeFeps}), we see that for $s \in \mathcal{S}, z \in (-\eta, \eta)$,
\begin{equation}\label{gradientepsz:decomposition}
\begin{aligned}
& \; \surfzloc \tilde{F}_{\eps}(s,z) = \sum_{i,j=1}^{n-1} g_{\eps}^{ij}(\alpha(s),z) \pd_{s_{i}} \tilde{F}_{\eps}(s,z) \pd_{s_{j}} G_{\eps}(s,z) \\
= & \; \sum_{i,j=1}^{n-1} g_{0}^{ij}(\alpha(s)) \pd_{s_{i}} \tilde{F}_{\eps}(s,z) \pd_{s_{j}} \alpha(s) + \eps z \sum_{i,j=1}^{n-1}  g_{0}^{ij}(\alpha(s)) \pd_{s_{i}} \tilde{F}_{\eps}(s,z) \pd_{s_{j}} \nu(\alpha(s)) \\
 - & \; \eps z \sum_{i,j=1}^{n-1} (\mathcal{G}_{0}^{-1} \mathcal{C} \mathcal{G}_{0}^{-1})_{ij}(s) \pd_{s_{i}} \tilde{F}_{\eps}(s,z) \pd_{s_{j}} \alpha(s)  + \text{ h.o.t.} \\
= & \; \surfloc \tilde{F}_{\eps}(s,z) + \tilde{\nabla}_{\eps z} \tilde{F}_{\eps}(s,z) = \surfloc \tilde{F}_{\eps}(s,z) + \mathcal{O}(\eps) \text{ as } \eps \to 0,
\end{aligned}
\end{equation}
where h.o.t. denotes terms of higher order in $\eps$.  By definition and by (\ref{samenormalGamma}),
\begin{align}\label{surfGradient:remaining:nu}
\tilde{\nabla}_{\eps z} \tilde{F}_{\eps}(s,z) \cdot \nu(\alpha(s)) = \surfzloc \tilde{F}_{\eps}(s,z) \cdot \nu(\alpha(s)) - \surfloc \tilde{F}_{\eps}(s,z) \cdot \nu(\alpha(s)) = 0.
\end{align}

The analogous statements for the representation in the $(p,z)$ coordinate system can be obtain from (\ref{gradientepsz:decomposition}) and (\ref{surfGradient:remaining:nu}) by applying the relations (\ref{surfGradientepszFeps}), (\ref{surfGradientFeps}), and (\ref{defn:nablaepsz}).
\end{proof}

\begin{remark}
In two dimensions, for a local arclength parameterisation $\alpha = \alpha(s)$, $s \in \mathcal{S} \subset \R$, of $\Gamma$, we have that $\mathcal{G}_{0} = 1$, $\pd_{s} \nu = - \kappa \pd_{s}\alpha$, $\mathcal{C} = -2 \kappa$, and $\surfloc \tilde{F}_{\eps}(s,z) = \pd_{s} \tilde{F}_{\eps}(s,z) \pd_{s} \alpha(s)$, where $\kappa$ is the curvature of $\Gamma$.  Hence, we obtain from $(\ref{gradientf:decomposition})$ and $(\ref{gradientepsz:decomposition})$ that
\begin{align*}
\nabla f = \frac{1}{\eps} \pd_{z} \tilde{F}_{\eps} \nu + \pd_{s} \tilde{F}_{\eps} \pd_{s} \alpha + \kappa \eps z \pd_{s} \tilde{F}_{\eps} \pd_{s} \alpha + \mathcal{O}(\eps^{2}),
\end{align*}
which is consistent with \cite[Appendix B]{article:GarckeStinner06}.
\end{remark}  

\begin{remark}\label{rem:ScaledNbd:gradientdecomposition:W11}
For fixed $\eps > 0$ sufficiently small, such that $\pd X^{\eps} \cap \pd \Omega = \emptyset$, we see that $\pd X^{\eps}$ has the same regularity as $\Gamma$.  Using that $C^{1}(\overline{X^{\eps}})$ functions are dense in $H^{1}(X^{\eps})$, we can extend the assertions of Lemma \ref{ScaledNbd:gradientdecomposition} to $H^{1}(X^{\eps})$ functions, where $(\ref{gradientf:decomposition})$ hold for weak derivatives, and $\pd_{z}(\cdot)$ and $\surfz(\cdot)$ are to be understood in the weak sense, and $(\ref{gradientremainderdotnuzero})$ and $(\ref{gradientremainderOrdereps})$ hold  almost everywhere.
\end{remark}

Using (\ref{TubNbd:intf:change}), (\ref{gradientf:decomposition}), and a change of variables $t \mapsto \eps z$, we have for any $0 < k \leq 1$,
\begin{equation}\label{ScaledNbd:coordchange}
\begin{aligned}
& \; \int_{\mathrm{Tub}^{\eps^{k} \eta}(\Gamma)} \abs{f}^{2} + \abs{\nabla f}^{2} \dx \\
= & \; \int_{-\eps^{k} \eta}^{\eps^{k} \eta} \int_{\Gamma} (\abs{f}^{2} + \abs{\nabla f}^{2})(p + t \nu(p)) \abs{1 + t C_{H}(p) + C_{R}(p,t)} \dHaus \dt, \\
= & \; \int_{\frac{-\eta}{\eps^{1-k}}}^{\frac{\eta}{\eps^{1-k}}} \int_{\Gamma} \frac{1}{\eps} \abs{\pd_{z} F_{\eps}}^{2} (p,z) \abs{1 + \eps z C_{H}(p) + C_{R}(p,\eps z)} \dHaus \dz \\
+ & \; \int_{\frac{-\eta}{\eps^{1-k}}}^{\frac{\eta}{\eps^{1-k}}} \int_{\Gamma} \eps (\abs{F_{\eps}}^{2} + \abs{\surfz F_{\eps}}^{2}) (p,z) \abs{1 + \eps z C_{H}(p) + C_{R}(p,\eps z)} \dHaus \dz.
\end{aligned}
\end{equation}
In addition, using that for $\eps \in (0,1]$,
\begin{align*}
1 - \tilde{c} \eps^{k} \eta \geq 1 - \tilde{c} \eta, \quad 1 + \tilde{c} \eps^{k} \eta \leq 1 + \tilde{c} \eta,
\end{align*}
we find that an analogous statement to (\ref{TubNbd:intf:upplowbdd}) holds for $\mathrm{Tub}^{\eps^{k} \eta}(\Gamma)$ for any $0 < k \leq 1$:
\begin{equation}\label{ScaledNbd:intf:upplowbdd}
\begin{aligned}
\frac{1}{1+\tilde{c} \eta} \int_{\mathrm{Tub}^{\eps^{k} \eta}(\Gamma)} \abs{f} \dx & \leq \int_{-\eps^{k} \eta}^{\eps^{k} \eta} \int_{\Gamma} \abs{f(p + t \nu(p))} \dHaus \dt \\
& \leq \frac{1}{1-\tilde{c} \eta} \int_{\mathrm{Tub}^{\eps^{k} \eta}(\Gamma)} \abs{f} \dx.
\end{aligned}
\end{equation}

\subsection{On functions extended constantly along the normal direction}\label{sec:ConstExt}
Let $f \in H^{1}(\Gamma)$, and $f^{e}$ denote its constant extension off $\Gamma$ in the normal direction to $\mathrm{Tub}^{\eta}(\Gamma)$, as defined in (\ref{defn:constext}).  A short calculation shows the following relation between $\nabla f^{e}$ and $\surf f$ (see \cite[proof of Theorem 2.10]{article:DziukElliott13} for more details):
\begin{align}\label{ConstExt:relationgradients}
\nabla f^{e}(x) = (\bm{I} - d(x) \bm{H}(x)) \surf f(p(x)) \text{ for } x \in \mathrm{Tub}^{\eta}(\Gamma),
\end{align}
where, $\bm{H}$ denotes the Hessian of the signed distance function $d$, and $\bm{I}$ denotes the identity tensor.  Consequently 
\begin{align}\label{ConstExt:samegradients}
\nabla f^{e}(y) = \surf f(y) \text{ for } y \in \Gamma.
\end{align}

\begin{cor}\label{cor:constextH1}
Let Assumption \ref{assump:Gamma} be satisfied, and choose $\eta > 0$ with the diffeomorphism $\Theta^{\eta} : \mathrm{Tub}^{\eta}(\Gamma) \to \Gamma \times (-\eta, \eta)$ as defined in $(\ref{diffeo})$.  Let $1 \leq q \leq \infty$, $f \in W^{1,q}(\Gamma)$, and let $f^{e}$ denote its constant extension in the normal direction off $\Gamma$ to $\mathrm{Tub}^{\eta}(\Gamma)$, as defined in $(\ref{defn:constext})$.  Then, $f^{e} \in W^{1,q}(\mathrm{Tub}^{\eta}(\Gamma))$, and there exists a constant $C > 0$, independent of $f$, such that
\begin{align*}
\norm{f^{e}}_{L^{q}(\mathrm{Tub}^{\eta}(\Gamma))} \leq C \norm{f}_{L^{q}(\Gamma)}, \quad \norm{\nabla f^{e}}_{L^{q}(\mathrm{Tub}^{\eta}(\Gamma))} \leq C \norm{ \surf f}_{L^{q}(\Gamma)}.
\end{align*}
Moreover, let $f^{Ec}$ denote the extension of $f^{e}$ from $\mathrm{Tub}^{\eta}(\Gamma)$ to $\Omega$ by the method of reflection.  Then $f^{Ec} \in W^{1,q}(\Omega)$ and there exists a constant $C > 0$, independent of $f$, such that
\begin{align}\label{constextH1:bdd}
\norm{f^{Ec}}_{L^{q}(\Omega))} \leq C \norm{f}_{L^{q}(\Gamma)}, \quad \norm{\nabla f^{Ec}}_{L^{q}(\Omega))} \leq C \norm{ \surf f}_{L^{q}(\Gamma)}. 
\end{align}
\end{cor}

\begin{proof}
Since $\Gamma$ is a $C^{2}$ hypersurface, we see that
\begin{align}\label{Hessian}
\norm{\bm{H}}_{C^{0}(\mathrm{Tub}^{\eta}(\Gamma))} \leq \norm{d}_{C^{2}(\mathrm{Tub}^{\eta}(\Gamma))} < \infty.
\end{align}
Let $1 \leq q < \infty$, $f \in W^{1,q}(\Gamma)$ and let $f^{e}$ denote its constant extension  in the normal direction off $\Gamma$ to $\mathrm{Tub}^{\eta}(\Gamma)$, as defined in (\ref{defn:constext}).   Then, by (\ref{ConstExt:relationgradients}), (\ref{TubNbd:intf:change}), (\ref{CHCRbdd}), and (\ref{TubNbd:intf:upplowbdd}), we have
\begin{align*}
\int_{\mathrm{Tub}^{\eta}(\Gamma)} \abs{f^{e}(x)}^{q} \dx & = \int_{\mathrm{Tub}^{\eta}(\Gamma)} \abs{f(p(x))}^{q} \dx \\
& = \int_{-\eta}^{\eta} \int_{\Gamma} \abs{f(p)}^{q} \abs{1 + t C_{H}(p) + C_{R}(p,t)} \dHaus \dt \\
& \leq C(\tilde{c} \eta) \norm{f}_{L^{q}(\Gamma)}^{q}, \\
\int_{\mathrm{Tub}^{\eta}(\Gamma)} \abs{\nabla f^{e}(x)}^{q} \dx & = \int_{\mathrm{Tub}^{\eta}(\Gamma)} \abs{\surf f(p(x)) - d(x) \bm{H}(x) \surf f(p(x))}^{q} \dx \\
& \leq C(1 + \norm{d}_{C^{2}(\mathrm{Tub}^{\eta}(\Gamma))}) \int_{\mathrm{Tub}^{\eta}(\Gamma)} \abs{\surf f(p(x))}^{q} \dx \\
& \leq C(1 + \norm{d}_{C^{2}(\mathrm{Tub}^{\eta}(\Gamma))})  C(\tilde{c} \eta) \norm{\surf f}_{L^{q}(\Gamma)}^{q}.
\end{align*}
Thus, $f^{e} \in W^{1,q}(\mathrm{Tub}^{\eta}(\Gamma))$ with
\begin{align*}
\norm{f^{e}}_{W^{1,q}(\mathrm{Tub}^{\eta}(\Gamma))} \leq C(\tilde{c} \eta , \norm{d}_{C^{2}(\mathrm{Tub}^{\eta}(\Gamma))}) \norm{f}_{W^{1,q}(\Gamma)}.
\end{align*}
For the case $q = \infty$, thanks to (\ref{defn:constext}), we have
\begin{align*}
\norm{f^{e}}_{L^{\infty}(\mathrm{Tub}^{\eta}(\Gamma))} = \norm{f}_{L^{\infty}(\Gamma)},
\end{align*}
and by (\ref{ConstExt:relationgradients}), we have
\begin{align*}
\norm{\nabla f^{e}}_{L^{\infty}(\mathrm{Tub}^{\eta}(\Gamma))} \leq C(\tilde{c} \eta, \norm{d}_{C^{2}(\mathrm{Tub}^{\eta}(\Gamma))}) \norm{\surf f}_{L^{\infty}(\Gamma)}.
\end{align*}

By the extension theorem \cite[Theorem 1, p. 254]{book:Evans}, we can extend $f^{e}$ to $f^{Ec} \in W^{1,q}(\Omega)$ with 
\begin{align*}
\norm{f^{Ec}}_{L^{q}(\Omega)} & \leq C\norm{f^{e}}_{L^{q}(\mathrm{Tub}^{\eta}(\Gamma))} \leq C \norm{f}_{L^{q}(\Gamma)}, \\
\norm{\nabla f^{Ec}}_{L^{q}(\Omega)} & \leq C\norm{\nabla f^{e}}_{L^{q}(\mathrm{Tub}^{\eta}(\Gamma))} \leq C \norm{\surf f}_{L^{q}(\Gamma)},
\end{align*}
where $C$ is independent of $f$.
\end{proof}

The next lemma allows us to test with extensions of $H^{1}(\Gamma)$ functions as constructed in Corollary \ref{cor:constextH1} in the weak formulation of the diffuse domain approximations.

\begin{lemma}\label{lem:constextH1delta}
Suppose Assumption \ref{assump:Gamma} and \ref{assump:delta} are satisfied.  Let $f \in H^{1}(\Gamma)$ and let $f^{Ec}$ denote its extension to $\Omega$ as constructed in Corollary \ref{cor:constextH1}.  Then, for all $\eps > 0$, 
\begin{align*}
f^{Ec} \in H^{1}(\Omega, \delta_{\eps}),
\end{align*}
and there exists a constant $C > 0$, independent of $f$ and $\eps$, such that
\begin{align*}
\norm{f^{Ec}}_{0,\delta_{\eps}} \leq C \norm{f}_{L^{2}(\Gamma)}, \quad \norm{\nabla f^{Ec}}_{0,\delta_{\eps}} \leq C \norm{\surf f}_{L^{2}(\Gamma)}.
\end{align*}
Furthermore, let $b^{Ec}$ and $\mathcal{B}^{Ec}$ denote the extensions of the data $b$ and $\mathcal{B}$ as mentioned in Assumption \ref{assump:surfacedata} and Remark \ref{rem:extensionambiguity}.  Then, as $\eps \to 0$,
\begin{equation}
\label{convergence:constextbB}
\begin{aligned}
\int_{\Omega} \delta_{\eps} b^{Ec} \abs{f^{Ec}}^{2} \dx & \to \int_{\Gamma} b \abs{f}^{2} \dHaus, \\
\int_{\Omega} \delta_{\eps} \mathcal{B}^{Ec} \nabla f^{Ec} \cdot \nabla f^{Ec} \dx & \to \int_{\Gamma} \mathcal{B} \surf f \cdot \surf f \dHaus.
\end{aligned}
\end{equation}
Consequently, as $\eps \to 0$,
\begin{align}\label{convergence:constext}
\int_{\Omega} \delta_{\eps} \abs{f^{Ec}}^{2} \dx \to \int_{\Gamma} \abs{f}^{2} \dHaus, \quad \int_{\Omega} \delta_{\eps} \abs{\nabla f^{Ec}}^{2} \dx \to \int_{\Gamma} \abs{\surf f}^{2} \dHaus.
\end{align}
\end{lemma}

\begin{proof}
We note that (\ref{convergence:constext}) follows from (\ref{convergence:constextbB}) if we consider $b(p) \equiv 1$ with $b^{Ec}(x) = 1$, and $b_{ij}(p) = \delta_{ij}$ with $b_{ij}^{Ec}(x) = \delta_{ij}$, for $p \in \Gamma$, $x \in \Omega$, $1\leq i,j \leq n$ and $\delta_{ij}$ is the Kronecker delta.  

Choose $\eta > 0$ and the diffeomorphism $\Theta^{\eta}$ as defined in (\ref{diffeo}).  Thanks to (\ref{property:delta:ptwiseconv}) with $q = 1$, 
\begin{align}\label{deltaeps:Linfty:remainder:convergence}
\norm{\delta_{\eps}}_{L^{\infty}(\Omega \setminus \mathrm{Tub}^{\eta}(\Gamma))} \to 0 \text{ as } \eps \to 0,
\end{align}
and so, for $\eps \in (0,1]$, 
\begin{align}\label{deltaeps:Linfty:remainder:bdd}
\norm{\delta_{\eps}}_{L^{\infty}(\Omega \setminus \mathrm{Tub}^{\eta}(\Gamma))} \leq \sup_{\eps \in (0,1]} \norm{\delta_{\eps}}_{L^{\infty}(\Omega \setminus \mathrm{Tub}^{\eta}(\Gamma))} =: C_{\mathrm{sup}}.
\end{align}

Then, by (\ref{TubNbd:intf:change}) with a change of variables $t \mapsto \eps s$, (\ref{defn:constext}), (\ref{CHCRbdd}), (\ref{property:delta}), and (\ref{constextH1:bdd}) with $q = 2$,
\begin{align*}
& \; \int_{\Omega} \delta_{\eps} \abs{f^{Ec}}^{2} \dx \leq \int_{\mathrm{Tub}^{\eta}(\Gamma)} \delta_{\eps} \abs{f^{e}}^{2} \dx + \norm{\delta_{\eps}}_{L^{\infty}(\Omega \setminus \mathrm{Tub}^{\eta}(\Gamma))} \norm{f^{Ec}}_{L^{2}(\Omega \setminus \mathrm{Tub}^{\eta}(\Gamma))}^{2} \\
\leq & \; \int_{-\eta}^{\eta} \int_{\Gamma} \frac{1}{\eps} \delta \left (\frac{t}{\eps} \right ) \abs{f^{e}(p + t \nu(p))}^{2} \abs{1 + t C_{H}(p) + C_{R}(p,t)} \dHaus \dt + C_{\mathrm{sup}} \norm{f^{Ec}}_{L^{2}(\Omega)}^{2} \\
\leq & \; \int_{\frac{-\eta}{\eps}}^{\frac{\eta}{\eps}} \int_{\Gamma} \delta(s) \abs{f^{e}(p + \eps s \nu(p))}^{2} \abs{1 + \eps s C_{H}(p) + C_{R}(p,\eps s)} \dHaus \ds + C_{\mathrm{sup}} \norm{f^{Ec}}_{L^{2}(\Omega)}^{2} \\
\leq & \; \int_{\frac{-\eta}{\eps}}^{\frac{\eta}{\eps}} \int_{\Gamma} \delta(s) (1 + \tilde{c} \eps \abs{s}) \abs{f(p)}^{2} \dHaus \ds + C_{\mathrm{sup}} \norm{f^{Ec}}_{L^{2}(\Omega)}^{2} \\
\leq & \; (1 + \tilde{c} \eta) \norm{f}_{L^{2}(\Gamma)}^{2} + C_{\mathrm{sup}} \norm{f^{Ec}}_{L^{2}(\Omega)}^{2} \leq C \norm{f}_{L^{2}(\Gamma)}^{2}.
\end{align*}
Hence, $f^{Ec} \in L^{2}(\Omega, \delta_{\eps})$.  

Next, we have by the definition of $b^{e}$, (\ref{defn:constext}), (\ref{TubNbd:intf:change}) with a change of variables $t \mapsto \eps s$, and (\ref{property:delta}):
\begin{align*}
\int_{\mathrm{Tub}^{\eta}(\Gamma)} \delta_{\eps} b^{e} \abs{f^{e}}^{2} \dx & = \int_{\frac{-\eta}{\eps}}^{\frac{\eta}{\eps}} \int_{\Gamma} \delta(s) b(p) \abs{f(p)}^{2} \abs{1 + \eps s C_{H}(p) + C_{R}(p, \eps s)} \dHaus \ds, \\
\int_{\Gamma} b \abs{f}^{2} \dHaus & = \int_{\R} \int_{\Gamma} \delta(s) b(p) \abs{f(p)}^{2} \dHaus \ds.
\end{align*}
Together with (\ref{CHCRbdd}), and (\ref{property:delta:bound}), we deduce that
\begin{align*}
& \; \abs{ \int_{\mathrm{Tub}^{\eta}(\Gamma)} \delta_{\eps} b^{e} \abs{f^{e}}^{2} \dx - \int_{\Gamma} b \abs{f}^{2} \dHaus } \\
\leq & \; \abs{\int_{\R \setminus (\frac{-\eta}{\eps}, \frac{\eta}{\eps})} \delta(s) \ds} \norm{b}_{L^{\infty}(\Gamma)} \norm{f}_{L^{2}(\Gamma)}^{2} \\
+ &\; \abs{\int_{\frac{-\eta}{\eps}}^{\frac{\eta}{\eps}} \int_{\Gamma} \delta(s) b(p) \abs{f(p)}^{2} \abs{\eps s C_{H}(p) + C_{R}(p,\eps s)} \dHaus \ds} \\
\leq & \; \left ( \int_{\R \setminus (\frac{-\eta}{\eps}, \frac{\eta}{\eps})} \delta(s) \ds + \eps \tilde{c} C_{\delta, \mathrm{int}} \right ) \norm{b}_{L^{\infty}(\Gamma)} \norm{f}_{L^{2}(\Gamma)}^{2} \to 0 \text{ as } \eps \to 0.
\end{align*}
Then, the first statement of (\ref{convergence:constextbB}) follows from the above and
\begin{align*}
\int_{\Omega \setminus \mathrm{Tub}^{\eta}(\Gamma)} \delta_{\eps} b^{Ec} \abs{f^{Ec}}^{2} \dx \leq \norm{\delta_{\eps}}_{L^{\infty}(\Omega \setminus \mathrm{Tub}^{\eta}(\Gamma))} \norm{b^{Ec}}_{L^{\infty}(\Omega)}  \norm{f^{Ec}}_{L^{2}(\Omega)}^{2} \to 0, \text{ as } \eps \to 0,
\end{align*}
where we have used (\ref{deltaeps:Linfty:remainder:convergence}).

We point out that the proof for the assertions for the gradients follow similarly by using (\ref{ConstExt:relationgradients}), i.e., 
\begin{align*}
\norm{\nabla f^{Ec}}_{0,\delta_{\eps}}^{2} & \leq C_{\mathrm{sup}} \norm{\nabla f^{Ec}}_{L^{2}(\Omega \setminus \mathrm{Tub}^{\eta}(\Gamma))}^{2} \\
& + C(\norm{d}_{C^{2}(\mathrm{Tub}^{\eta}(\Gamma))})  \int_{\frac{-\eta}{\eps}}^{\frac{\eta}{\eps}} \int_{\Gamma} \delta(s) (1 + \tilde{c} \eps \abs{s}) \abs{\surf f(p)}^{2} \dHaus \ds \\
& \leq C_{\mathrm{sup}} \norm{\nabla f^{Ec}}_{L^{2}(\Omega)}^{2} + C(d, \eta, \tilde{c}) \norm{\surf f}_{L^{2}(\Gamma)}^{2} \leq C \norm{\surf f}_{L^{2}(\Gamma)}^{2},
\end{align*}
which then implies that $f^{Ec} \in H^{1}(\Omega, \delta_{\eps})$.  Similarly, using (\ref{TubNbd:intf:change}), (\ref{ConstExt:relationgradients}) and the definition of $\mathcal{B}^{e}$, and a change of variables $t \mapsto \eps s$,
\begin{align*}
& \; \int_{\mathrm{Tub}^{\eta}(\Gamma)} \delta_{\eps} \mathcal{B}^{e}(x) \nabla f^{Ec}(x) \cdot \nabla f^{Ec}(x) \dx \\
= & \; \int_{-\eta}^{\eta} \int_{\Gamma} \frac{1}{\eps} \delta \left ( \frac{t}{\eps} \right ) \mathcal{B}(p) [(\bm{I} - t \bm{H}) \surf f(p)] \cdot [(\bm{I} - t \bm{H}) \surf f(p)] \\
& \; \times \abs{1 + t C_{H}(p) + C_{R}(p,t)} \dHaus \dt \\
= & \; \int_{\frac{-\eta}{\eps}}^{\frac{\eta}{\eps}} \int_{\Gamma} \delta(s)\mathcal{B}(p) [(\bm{I} - \eps s \bm{H}) \surf f(p)] \cdot [(\bm{I} - \eps s \bm{H}) \surf f(p)] \\
& \; \times  \abs{1 + \eps s C_{H}(p) + C_{R}(p,\eps s)} \dHaus \ds \\
=: & \; \int_{\frac{-\eta}{\eps}}^{\frac{\eta}{\eps}} \int_{\Gamma} \delta(s) \mathcal{B}(p) \surf f(p) \cdot \surf f(p) \abs{1 + \eps s C_{H}(p) + C_{R}(p,\eps s)} \dHaus \ds + \mathcal{R},
\end{align*}
where, by (\ref{Hessian}), (\ref{CHCRbdd}), and (\ref{property:delta:bound}),
\begin{align*}
\abs{\mathcal{R}} \leq \eps C(\norm{d}_{C^{2}(\mathrm{Tub}^{\eta}(\Gamma))}, \tilde{c} \eta, C_{\delta, \mathrm{int}}) \norm{\mathcal{B}}_{L^{\infty}(\Gamma)} \norm{\surf f}_{L^{2}(\Gamma)}^{2} \to 0 \text{ as } \eps \to 0.
\end{align*}
Thus, setting $L := \int_{\frac{-\eta}{\eps}}^{\frac{\eta}{\eps}} \int_{\Gamma} \delta(s) \mathcal{B}(p) \surf f(p) \cdot \surf f(p) \dHaus \ds$ allow us to compute
\begin{align*}
& \; \abs{\int_{\mathrm{Tub}^{\eta}(\Gamma)} \delta_{\eps} \mathcal{B}^{Ec} \nabla f^{Ec} \cdot \nabla f^{Ec} \dx - L + L -  \int_{\Gamma} \mathcal{B} \surf f \cdot \surf f \dHaus  } \\
\leq & \; \abs{\int_{\R \setminus (\frac{-\eta}{\eps}, \frac{\eta}{\eps})} \delta(s) \int_{\Gamma} \mathcal{B}(p) \surf f(p) \cdot \surf f(p) \dHaus \ds } \\
+ & \; \abs{\int_{\frac{-\eta}{\eps}}^{\frac{\eta}{\eps}} \int_{\Gamma} \delta(s) \mathcal{B}(p) \surf f(p) \cdot \surf f(p) \abs{ \eps s C_{H}(p) + C_{R}(p,\eps s)} \dHaus ds} + \abs{\mathcal{R}} \\
\leq & \; \left ( \int_{\R \setminus (\frac{-\eta}{\eps}, \frac{\eta}{\eps})} \delta(s) \ds + \eps \tilde{c} C_{\delta, \mathrm{int}} \right ) \norm{\mathcal{B}}_{L^{\infty}(\Gamma)} \norm{\surf f}_{L^{2}(\Gamma)}^{2} + \abs{\mathcal{R}} \to 0 \text{ as } \eps \to 0.
\end{align*}
Then, the second statement of (\ref{convergence:constextbB}) follows from the above and
\begin{align*}
\int_{\Omega \setminus \mathrm{Tub}^{\eta}(\Gamma)} \delta_{\eps} \mathcal{B}^{Ec} \nabla f^{Ec} \cdot \nabla f^{Ec} \dx \leq \norm{\delta_{\eps}}_{L^{\infty}(\Omega \setminus \mathrm{Tub}^{\eta}(\Gamma))} \norm{\mathcal{B}^{Ec}}_{L^{\infty}(\Omega)} \norm{\nabla f^{Ec}}_{L^{2}(\Omega)}^{2} \to 0,
\end{align*}
as $\eps \to 0$.
\end{proof}

\begin{remark}
For the double-obstacle regularisation, we use the fact that $\delta_{\eps} = 0$ on $\Omega \setminus \mathrm{Tub}^{\eta}(\Gamma)$ for $\eta \geq \eps \frac{\pi}{2}$ to deduce the same results.
\end{remark}

\subsection{On the regularised indicator functions}\label{sec:RegularisedFunct}

Due to the boundedness of $\xi_{\eps}$, (\ref{property:Xi}), and Lebesgue's dominated convergence theorem, we have
\begin{lemma}\label{lem:XiDCT}
Suppose that Assumption \ref{assump:Xi} is satisfied, then for any $g \in L^{1}(\Omega)$,
\begin{align*}
\int_{\Omega^{*}} \xi_{\eps} g \dx \to \int_{\Omega^{*}} g \mid_{\Omega^{*}} \dx, \quad \int_{\Omega^{*}} (1-\xi_{\eps}) g \dx \to 0, \quad \int_{\Omega \setminus \Omega^{*}} \xi_{\eps} g \dx \to 0 \text{ as } \eps \to 0.
\end{align*}
\end{lemma}

\begin{lemma}\label{lem:H1impliesL2delta}
Suppose that Assumptions \ref{assump:Gamma} and \ref{assump:delta} are satisfied.  Then, there exists a constant $C > 0$ such that, for all $f \in W^{1,1}(\Omega)$ and all $\eps \in (0,1]$,
\begin{align*}
\int_{\Omega} \delta_{\eps} \abs{f} \dx \leq C \norm{f}_{W^{1,1}(\Omega)}.
\end{align*}
In particular, there exists a constant $C > 0$ such that, for all $f \in H^{1}(\Omega)$ and all $\eps \in (0,1]$,
\begin{align*}
\norm{f}_{0,\delta_{\eps}} \leq C \norm{f}_{H^{1}(\Omega)}.
\end{align*}
\end{lemma}
\begin{proof}
Choose $\eta > 0$ and the diffeomorphism $\Theta^{\eta}$ as defined in (\ref{diffeo}).  Let $f \in W^{1,1}(\Omega)$.  Then, by (\ref{deltaeps:Linfty:remainder:bdd}), 
\begin{align}\label{deltafsquare:remainder}
\int_{\Omega \setminus \mathrm{Tub}^{\eta}(\Gamma)} \delta_{\eps} \abs{f} \dx \leq C_{\mathrm{sup}} \norm{f}_{L^{1}(\Omega \setminus \mathrm{Tub}^{\eta}(\Gamma))} \leq C_{\mathrm{sup}} \norm{f}_{W^{1,1}(\Omega)}.
\end{align}
Using the diffeomorphism $\Theta^{\eta}$, $f \mid_{\mathrm{Tub}^{\eta}(\Gamma)} \in W^{1,1}(\Gamma \times (-\eta, \eta))$.  By absolute continuity on lines for $W^{1,1}$ functions (see \cite[Theorem 2.1.4]{book:Ziemer} or \cite[Theorem 1, \S 1.1.3]{book:Mazja}) there exists a version of $f$, which we denote by the same symbol, such that for a.e. $p \in \Gamma$, it is absolutely continuous as a function of $t \in (-\eta, \eta)$.  With absolute continuity with respect to $t$, we have for $t \in (-\eta, \eta)$,
\begin{align*}
f(p + t \nu(p)) = f(p) + \int_{0}^{t} \frac{\dd}{\dzeta} f (p + \zeta \nu(p)) \dzeta.
\end{align*}
Then, by (\ref{TubNbd:intf:change}), a change of coordinates $t = \eps s$, (\ref{property:delta}), and (\ref{TubNbd:intf:upplowbdd}), 
\begin{align*}
& \; \int_{\mathrm{Tub}^{\eta}(\Gamma)} \delta_{\eps} \abs{f} \dx = \int_{-\eta}^{\eta} \int_{\Gamma} \frac{1}{\eps} \delta \left ( \frac{t}{\eps} \right ) \abs{f(p + t \nu(p))} \abs{1 + t C_{H}(p) + C_{R}(p, t)} \dHaus \dt \\
= & \; \int_{-\eta}^{\eta} \int_{\Gamma} \frac{1}{\eps} \delta \left ( \frac{t}{\eps} \right ) \left [ \abs{f(p)} + \int_{0}^{t} \abs{\frac{\dd}{\dzeta} f(p,\zeta)}  \dzeta \right ] \abs{1 + t C_{H}(p) + C_{R}(p,t)} \dHaus \dt \\
\leq & \; \left ( \norm{\gamma_{0}(f)}_{L^{1}(\Gamma)}^{2} + \int_{\Gamma} \int_{-\eta}^{\eta} \abs{\nabla f(p + \zeta \nu(p))} \dzeta \dHaus \right ) \int_{\frac{-\eta}{\eps}}^{\frac{\eta}{\eps}} \delta(s) (1 + \tilde{c} \eps \abs{s}) \ds \\
\leq & \; (1 + \eps \tilde{c} C_{\delta, \mathrm{int}} ) \left ( \norm{\gamma_{0}(f)}_{L^{1}(\Gamma)} + C(\tilde{c} \eta) \norm{\nabla f}_{L^{1}(\mathrm{Tub}^{\eta}(\Gamma))} \right ) \\
\leq & \; (1 + \tilde{c} C_{\delta, \mathrm{int}} ) \left ( \norm{\gamma_{0}(f)}_{L^{1}(\Gamma)} + C(\tilde{c} \eta) \norm{f}_{W^{1,1}(\mathrm{Tub}^{\eta}(\Gamma))} \right ).
\end{align*}
Using the trace theorem \cite[Theorem 1, p. 272]{book:Evans} and (\ref{deltafsquare:remainder}), we have
\begin{align*}
\int_{\Omega} \delta_{\eps} \abs{f} \dx \leq ((1 + \tilde{c} C_{\delta, \mathrm{int}}) (C_{\mathrm{tr}} + C(\tilde{c} \eta)) + C_{\mathrm{sup}}) \norm{f}_{W^{1,1}(\Omega)},
\end{align*}
where $C_{\mathrm{tr}}$ is the constant from the trace theorem.  

Note that the second assertion follows due to 
\begin{align*}
\norm{\nabla (\abs{f}^{2})}_{L^{1}(\Omega)} \leq 2 \norm{f}_{L^{2}(\Omega)}\norm{\nabla f}_{L^{2}(\Omega)} \leq 2 \norm{f}_{H^{1}(\Omega)}^{2}.
\end{align*}
\end{proof}
\begin{remark}
For the double-obstacle regularisation, we directly obtain
\begin{align*}
\int_{\mathrm{Tub}^{\eps \frac{\pi}{2}}(\Gamma)} \delta_{\eps} \abs{f} \dx \leq (1 + \tilde{c} C_{\delta, \mathrm{int}})(C_{\mathrm{tr}} + C(\tilde{c} \eta)) \norm{f}_{W^{1,1}(\Omega)}.
\end{align*}
\end{remark}

\begin{lemma}\label{lem:convergetodirac}
Suppose that Assumptions \ref{assump:Gamma} and \ref{assump:delta} are satisfied.  For $f \in W^{1,1}(\Omega)$, it holds that
\begin{align}\label{convergetodirac}
\int_{\Omega} \delta_{\eps} f \dx \to \int_{\Gamma} \gamma_{0}(f) \dHaus \text{ as } \eps \to 0.
\end{align} 
Moreover, for $\eta > 0$ sufficiently small, if $f \in W^{1,q}(\Omega)$, $1 \leq q < \infty$, or if $f = C^{1}(\overline{\Omega})$ and $q = \infty$, then there exists a constant $C > 0$, independent of $f$ and $\eps$, such that
\begin{align}\label{convergetodirac:power}
\abs{\int_{\mathrm{Tub}^{\eta}(\Gamma)} \delta_{\eps} f \dx - \int_{\Gamma} \gamma_{0}(f) \dHaus} \leq C \eps^{1 - \frac{1}{q}} \norm{f}_{W^{1,q}(\Omega)}.
\end{align}
Furthermore, if $f \in W^{1,q}(\Omega)$, $2 \leq q < \infty$ with $\gamma_{0}(f) = 0$, or if $f \in C^{1}(\overline{\Omega})$ with $f \mid_{\Gamma} = 0$ and $q = \infty$, we have
\begin{align}\label{convergetodirac:zerotrace}
\int_{\mathrm{Tub}^{\eta}(\Gamma)} \delta_{\eps} \abs{f}^{2} \dx \leq C \eps^{2 - \frac{2}{q}} \norm{f}_{W^{1,q}(\Omega)}^{2}.
\end{align}
\end{lemma}

\begin{proof}
By the trace theorem and Lemma \ref{lem:H1impliesL2delta}, the integrals in (\ref{convergetodirac}) are well-defined.  Choose $0 < \eta < (\tilde{c})^{-1}$ and the diffeomorphism $\Theta^{\eta}$ as defined in (\ref{diffeo}), where $\tilde{c}$ is the constant in (\ref{CHCRbdd}).  

By (\ref{property:delta:ptwiseconv}) with $q = 2$, we see that
\begin{align*}
\lim_{\eps \to 0} \frac{1}{\eps^{2}} \delta \left ( \frac{\eta}{\eps} \right ) = 0.
\end{align*}
Let $C_{\mathrm{sup},\eta} := \sup_{\eps \in (0,1]} \frac{1}{\eps^{2}} \delta \left ( \frac{\eta}{\eps} \right ) < \infty$.  Then, for all $\eps \in (0,1]$, we have
\begin{align*}
\delta \left ( \frac{\eta}{\eps} \right ) \leq C_{\mathrm{sup},\eta} \eps^{2},
\end{align*}
and so, from the monotonicity of $\delta(\abs{s})$, the symmetry of $\delta(s)$ about $s = 0$, and (\ref{property:delta:bound}), 
\begin{equation}\label{int:delta:ordereps}
\begin{aligned}
\int_{\R \setminus (\frac{-\eta}{\eps}, \frac{\eta}{\eps})} \delta(s) \ds & = 2\int_{\frac{\eta}{\eps}}^{\infty} \delta(s) \ds \leq 2\sqrt{\delta \left ( \frac{\eta}{\eps} \right )} \left ( \int_{\frac{\eta}{\eps}}^{\infty} \sqrt{\delta(s)} \ds \right ) \\
& \leq 2\sqrt{C_{\mathrm{sup},\eta}} C_{\delta, \mathrm{int}} \eps \leq C \eps,
\end{aligned}
\end{equation}
for some constant $C > 0$.  

Let $q < \infty$ and $f \in W^{1,q}(\Omega)$.  Then, by the diffeomorphism $\Theta^{\eta}$, absolute continuity on lines, and H\"{o}lder's inequality, we have
\begin{equation}\label{Holder}
\begin{aligned}
& \; \abs{f(p + \eps s \nu(p)) - f(p)} = \abs{\int_{0}^{s} \frac{\dd}{\dzeta} f(p + \eps \zeta \nu(p)) \dzeta} \\
= & \; \eps \abs{\int_{0}^{s} \nabla f(p+ \eps \zeta \nu(p)) \cdot \nu(p) \dzeta} \leq \eps \abs{s}^{\frac{q-1}{q}} \left ( \int_{0}^{s} \abs{\nabla f(p + \eps \zeta \nu(p))}^{q} \dzeta \right )^{\frac{1}{q}}.
\end{aligned}
\end{equation}
Then, by (\ref{Holder}),  (\ref{TubNbd:intf:change}), a change of variables $t \mapsto \eps s$, H\"{o}lder's inequality, and (\ref{CHCRbdd}),
\begin{equation}\label{difference:Holder}
\begin{aligned}
& \; \abs{\int_{\mathrm{Tub}^{\eta}(\Gamma)} \delta_{\eps} f \dx - \int_{\Gamma} \gamma_{0}(f) \dHaus} \\ 
\leq & \; \int_{\R} \int_{\Gamma} \delta(s) \abs{\chi_{\left (\frac{-\eta}{\eps},\frac{\eta}{\eps} \right )}(s) f(p + \eps s \nu(p)) \abs{1 + \eps s C_{H}(p) + C_{R}(p,\eps s)} - f(p)} \dHaus \ds \\
\leq & \; \int_{\R} \int_{\Gamma} \delta(s) \abs{f(p)} \abs{1 - \chi_{\left (\frac{-\eta}{\eps},\frac{\eta}{\eps} \right )}(s) \abs{1 + \eps s C_{H}(p) + C_{R}(p,\eps s)}} \dHaus \ds \\
+ & \; \int_{\R} \int_{\Gamma} \delta(s) \chi_{(\frac{-\eta}{\eps}, \frac{\eta}{\eps})}(s) \abs{f (p + \eps s \nu(p)) - f(p)} \abs{1 + \eps s C_{H}(p) + C_{R}(p, \eps s)} \dHaus \ds \\
\leq & \; \norm{\gamma_{0}(f)}_{L^{1}(\Gamma)} \int_{\R} \delta(s) \left ( \chi_{\R \setminus \left ( \frac{-\eta}{\eps}, \frac{\eta}{\eps} \right )} (s) + \eps \tilde{c} \abs{s} \right ) \ds \\
+ & \; \eps (1+ \tilde{c} \eta) \int_{\R} \int_{\Gamma} \chi_{\left (\frac{-\eta}{\eps},\frac{\eta}{\eps} \right )}(s) \delta(s) \abs{s}^{\frac{q-1}{q}} \left ( \int_{0}^{s} \abs{\nabla f(p + \eps \zeta \nu(p))}^{q} \dzeta \right )^{\frac{1}{q}} \dHaus \ds \\
\leq & \; \norm{\gamma_{0}(f)}_{L^{1}(\Gamma)} \int_{\R} \delta(s) \left ( \chi_{\R \setminus \left ( \frac{-\eta}{\eps}, \frac{\eta}{\eps} \right )} (s) + \eps \tilde{c} \abs{s} \right ) \ds \\
+ & \; \eps (1 + \tilde{c} \eta) \abs{\Gamma}^{\frac{q-1}{q}}  \int_{\R} \delta(s) (1 + \abs{s}) \left ( \int_{\Gamma} \int_{0}^{\frac{\eta}{\eps}} \abs{\nabla f(p + \eps \zeta \nu(p))}^{q} \dzeta \dHaus \right )^{\frac{1}{q}} \ds.
\end{aligned}
\end{equation}
Here, we have also used that $\abs{s}^{1 - \frac{1}{r}} \leq 1 + \abs{s}$ for all $s \in \R$ and $1 \leq r \leq \infty$.  By a change of variables $\eps \zeta \mapsto t$, and (\ref{TubNbd:intf:upplowbdd}), we observe that
\begin{equation}\label{negativepower:eps}
\begin{aligned}
& \; \left ( \int_{\Gamma} \int_{0}^{\frac{\eta}{\eps}} \abs{\nabla f(p + \eps \zeta \nu(p))}^{q} \dzeta \dHaus \right )^{\frac{1}{q}} \\
= & \; \eps^{-\frac{1}{q}} \left ( \int_{\Gamma} \int_{0}^{\eta} \abs{\nabla f(p + t \nu(p))}^{q} \dzeta \dHaus \right)^{\frac{1}{q}} \leq C(\tilde{c} \eta) \eps^{-\frac{1}{q}} \norm{\nabla f}_{L^{q}(\mathrm{Tub}^{\eta}(\Gamma))},
\end{aligned}
\end{equation}
Together with (\ref{int:delta:ordereps}), and (\ref{property:delta:bound}), we see from (\ref{difference:Holder}) that
\begin{align*}
\abs{\int_{\mathrm{Tub}^{\eta}(\Gamma)} \delta_{\eps} f \dx  - \int_{\Gamma} \gamma_{0}(f) \dHaus} & \leq \eps C \norm{\gamma_{0}(f)}_{L^{1}(\Gamma)} + \eps^{1-\frac{1}{q}} C \norm{\nabla f}_{L^{q}(\mathrm{Tub}^{\eta}(\Gamma))} \\
& \leq C \eps^{1-\frac{1}{q}} \norm{f}_{W^{1,q}(\Omega)},
\end{align*}
which is (\ref{convergetodirac:power}) for $q < \infty$. 

For the case $q = \infty$, let $f \in C^{1}(\overline{\Omega})$.  Then, by the fundamental theorem of calculus, we have
\begin{equation}\label{C1:difference:fundamentalthmCalculus}
\begin{aligned}
\abs{f (p + \eps s \nu(p)) - f(p)} & = \eps \abs{\int_{0}^{s} \nabla f(p + \eps \zeta \nu(p)) \cdot \nu(p) \dzeta} \\
& \leq \eps \abs{s} \sup_{\zeta \in [0,s]} \abs{\nabla f(p + \eps \zeta \nu(p))} \leq \eps \abs{s} \norm{\nabla f}_{C^{0}(\mathrm{Tub}^{\eta}(\Gamma))},
\end{aligned}
\end{equation}
and hence, by a similar calculation to (\ref{difference:Holder}) yields,
\begin{align*}
& \; \abs{\int_{\mathrm{Tub}^{\eta}(\Gamma)} \delta_{\eps} f \dx - \int_{\Gamma} \gamma_{0}(f) \dHaus } \\
\leq & \; \eps C \norm{\gamma_{0}(f)}_{L^{1}(\Gamma)} +  \eps (1+\tilde{c} \eta) \abs{\Gamma} \norm{\nabla f}_{C^{0}(\mathrm{Tub}^{\eta}(\Gamma))} \int_{\R} \delta(s) \abs{s} \ds \\
\leq & \; C \eps \norm{f}_{C^{1}(\Omega)}.
\end{align*}

Next, by (\ref{deltaeps:Linfty:remainder:convergence}), we see that
\begin{align*}
\int_{\Omega \setminus \mathrm{Tub}^{\eta}(\Gamma)} \delta_{\eps} \abs{f} \dx \leq \norm{\delta_{\eps}}_{L^{\infty}(\Omega \setminus \mathrm{Tub}^{\eta}(\Gamma))} \norm{f}_{L^{1}(\Omega)} \to 0 \text { as } \eps \to 0.
\end{align*} 
Together with (\ref{convergetodirac:power}) for $q = \infty$, we have that (\ref{convergetodirac}) holds true for all $C^{1}(\overline{\Omega})$ functions.  Let $\zeta > 0$ be arbitrary, then by the density of smooth functions, for any $f \in W^{1,1}(\Omega)$, there exists $g \in C^{1}(\overline{\Omega})$ such that
\begin{align*}
\norm{f - g}_{W^{1,1}(\Omega)} < \zeta.
\end{align*}
Then, by (\ref{convergetodirac}) for $C^{1}(\overline{\Omega})$ functions, there exists $\eps_{0}(\zeta)$ such that for $\eps < \eps_{0}(\zeta)$, 
\begin{align*}
\abs{\int_{\Omega} \delta_{\eps} g \dx - \int_{\Gamma} \gamma_{0}(g) \dHaus} < \zeta.
\end{align*}
By Lemma \ref{lem:H1impliesL2delta}, and the trace theorem, we find that for $\eps < \eps_{0}(\zeta)$,
\begin{equation}\label{C1density}
\begin{aligned}
& \; \abs{\int_{\Omega} \delta_{\eps} f \dx - \int_{\Gamma} \gamma_{0}(f) \dHaus }\\
\leq & \; \abs{\int_{\Omega} \delta_{\eps}(f-g) \dx - \int_{\Gamma} \gamma_{0}(f-g) \dHaus} + \abs{\int_{\Omega} \delta_{\eps} g \dx - \int_{\Gamma} \gamma_{0}(g) \dHaus} \\
\leq & \; C \norm{f-g}_{W^{1,1}(\Omega)} + \norm{f-g}_{L^{1}(\Gamma)} + \zeta \\
\leq & \; (C+ C_{\mathrm{tr}}) \norm{f-g}_{W^{1,1}(\Omega)}  +  \zeta \leq (C + C_{\mathrm{tr}} + 1) \zeta.
\end{aligned}
\end{equation}
By the arbitrariness of $\zeta$, we conclude that (\ref{convergetodirac}) holds for any $f \in W^{1,1}(\Omega)$.

For assertion (\ref{convergetodirac:zerotrace}), we note that for $f \in C^{1}(\overline{\Omega})$ with $f \mid_{\Gamma} = 0$, a similar calculation to (\ref{Holder}) yields for $2 \leq q < \infty$,
\begin{align*}
\abs{f(p + \eps s \nu(p))}^{2} & \leq \eps^{2} \abs{\int_{0}^{s} \abs{\nabla f(p+\eps \zeta \nu(p))} \dzeta}^{2} \\
& \leq \eps^{2} \abs{s}^{2 - \frac{2}{q}} \left ( \int_{0}^{s} \abs{\nabla f(p + \eps \zeta \nu(p))}^{q} \dzeta \right)^{\frac{2}{q}} .
\end{align*}
Then, a similar calculation to (\ref{difference:Holder}), and using (\ref{negativepower:eps}) yield
\begin{align*}
& \; \abs{\int_{\mathrm{Tub}^{\eta}(\Gamma)} \delta_{\eps} \abs{f}^{2} \dx } \\
\leq & \; \eps^{2} \int_{\frac{-\eta}{\eps}}^{\frac{\eta}{\eps}} \int_{\Gamma} \abs{1 + \eps s C_{H}(p) + C_{R}(p,\eps s)} \delta(s) \abs{s}^{2- \frac{2}{q}} \left ( \int_{0}^{s} \abs{\nabla f(p + \eps \zeta \nu(p))}^{q} \dzeta \right)^{\frac{2}{q}} \dHaus \ds \\
\leq & \; \eps^{2} (1 + \tilde{c} \eta) \abs{\Gamma}^{1-\frac{2}{q}} \left ( \int_{\Gamma} \int_{0}^{\frac{\eta}{\eps}} \abs{\nabla f(p + \eps \zeta \nu(p))}^{q} \dzeta \dHaus \right )^{\frac{2}{q}} \left ( \int_{\R} \delta(s) (1 + \abs{s} + \abs{s}^{2}) \ds \right ) \\
\leq & \; C \eps^{2 - \frac{2}{q}} \norm{f}_{W^{1,q}(\Omega)}^{2}.
\end{align*}
Using the density of $C^{1}(\overline{\Omega})$ in $W^{1,q}(\Omega)$, and arguing similarly as in (\ref{C1density}), we see that assertion (\ref{convergetodirac:zerotrace}) holds true for any $f \in W^{1,q}(\Omega)$, $2 \leq q < \infty$ with $\gamma_{0}(f) = 0$.

While, for the case $q = \infty$, let $f \in C^{1}(\overline{\Omega})$ with $f \mid_{\Gamma} = 0$.  We have 
\begin{align*}
\abs{f(p+\eps s \nu(p)) - f(p)}^{2} = \abs{f(p+\eps s \nu(p))}^{2} \leq \eps^{2} \abs{s}^{2} \norm{\nabla f}_{C^{0}(\mathrm{Tub}^{\eta}(\Gamma))}^{2},
\end{align*}
and so
\begin{align}\label{convergetodirac:zerotrace:C1}
\abs{\int_{\mathrm{Tub}^{\eta}(\Gamma)} \delta_{\eps} \abs{f}^{2} \dx} \leq \eps^{2} \norm{f}_{C^{1}(\mathrm{Tub}^{\eta}(\Gamma))}^{2} (1 + \tilde{c} \eta) \int_{\R} \delta(s) \abs{s}^{2} \ds \leq C \eps^{2} \norm{f}_{C^{1}(\Omega)}^{2}.
\end{align}

\end{proof} 

\begin{remark}
The analogous assertions for the double-obstacle regularisation with $\eta = \eps \frac{\pi}{2}$ follows along a similar argument.  For this case, we point out that $(\ref{int:delta:ordereps})$ is not needed.
\end{remark}

%

We define
\begin{align*}
H^{1}_{\Gamma,0}(\Omega) := \{ f \in H^{1}(\Omega) : \gamma_{0}(f) = 0 \}.
\end{align*}
This is a closed subspace of $H^{1}(\Omega)$ under $\norm{\cdot}_{H^{1}(\Omega)}$ and hence is a Hilbert space itself.

\begin{lemma}\label{lem:H1zeroimpliesL2deltaeps-1}
Suppose that Assumptions \ref{assump:Gamma} and \ref{assump:delta} are satisfied.  There exists a constant $C > 0$ such that for all $f \in H^{1}_{\Gamma,0}(\Omega)$ and all $\eps \in (0,1]$,
\begin{align*}
\norm{f}_{0,\frac{1}{\eps} \delta_{\eps}} \leq C \norm{f}_{H^{1}(\Omega)}.
\end{align*}
Moreover,
\begin{align}\label{H1zeroL2deltaeps-1conv}
\int_{\Omega} \frac{1}{\eps} \delta_{\eps} \abs{f}^{2} \dx \to 0 \text{ as } \eps \to 0.
\end{align}
\end{lemma}

\begin{proof}
Choose $0 < \eta < (\tilde{c})^{-1}$ an the diffeomorphism $\Theta^{\eta}$ as defined in (\ref{diffeo}).  By (\ref{property:delta:ptwiseconv}) with $q = 2$, we have
\begin{align}\label{remainder:delta:eps-1}
\norm{\tfrac{1}{\eps} \delta_{\eps}}_{L^{\infty}(\Omega \setminus \mathrm{Tub}^{\eta}(\Gamma))} \to 0 \text{ as } \eps \to 0.
\end{align}
Hence, for $\eps \in (0,1]$,
\begin{align*}
\norm{\tfrac{1}{\eps} \delta_{\eps}}_{L^{\infty}(\Omega \setminus \mathrm{Tub}^{\eta}(\Gamma))} \leq \sup_{\eps \in (0,1]} \norm{\tfrac{1}{\eps} \delta_{\eps}}_{L^{\infty}(\Omega \setminus \mathrm{Tub}^{\eta}(\Gamma))}  =: C_{\mathrm{sup},-1}.
\end{align*}
Let $f \in H^{1}_{\Gamma,0}(\Omega)$.  By (\ref{convergetodirac:zerotrace}) with $q = 2$, we have 
\begin{align*}
\int_{\Omega} \frac{1}{\eps} \delta_{\eps} \abs{f}^{2} \dx & \leq \int_{\mathrm{Tub}^{\eta}(\Gamma)} \frac{1}{\eps} \delta_{\eps} \abs{f}^{2} \dx + \norm{\tfrac{1}{\eps} \delta_{\eps}}_{L^{\infty}(\Omega \setminus \mathrm{Tub}^{\eta}(\Gamma))} \norm{f}_{L^{2}(\Omega \setminus \mathrm{Tub}^{\eta}(\Gamma))}^{2} \\
& \leq C \norm{f}_{H^{1}(\Omega)}^{2} + C_{\mathrm{sup},-1} \norm{f}_{L^{2}(\Omega)}^{2} \leq C \norm{f}_{H^{1}(\Omega)}^{2},
\end{align*}
where $C$ is independent of $f$ and $\eps$.  Next, take $g \in C^{1}(\overline{\Omega})$ with $g \mid_{\Gamma} = 0$.  Then, from (\ref{convergetodirac:zerotrace:C1}) and (\ref{remainder:delta:eps-1}) we have
\begin{align*}
\int_{\Omega} \frac{1}{\eps} \delta_{\eps} \abs{g}^{2} \dx \leq C \eps \norm{g}_{C^{1}(\Omega)}^{2} + \norm{\tfrac{1}{\eps} \delta_{\eps}}_{L^{\infty}(\Omega \setminus \mathrm{Tub}^{\eta}(\Gamma))} \norm{g}_{L^{2}(\Omega \setminus \mathrm{Tub}^{\eta}(\Gamma))}^{2}  \to 0 \text{ as } \eps \to 0.
\end{align*}
Using a density argument similar to (\ref{C1density}), we see that assertion (\ref{H1zeroL2deltaeps-1conv}) holds true for any $f \in H^{1}_{\Gamma,0}(\Omega)$.
\end{proof}

As a consequence of (\ref{property:Xidelta}), we have
\begin{align}\label{XiDelta}
\frac{1}{\eps} \delta \left ( \frac{d(x)}{\eps} \right ) \leq \frac{1}{\eps} \frac{1}{C_{\xi}} \xi \left ( \frac{d(x)}{\eps} \right ),
\end{align}
and thus, for all $\eps \in (0,1]$, and any $f \in L^{2}(\Omega, \xi_{\eps})$, we have
\begin{align}\label{XiDeltaRelation}
\int_{\Omega} \delta_{\eps} \abs{f}^{2} \dx \leq \frac{1}{\eps} \frac{1}{C_{\xi}} \int_{\Omega} \xi_{\eps} \abs{f}^{2} \dx.
\end{align}

We now introduce the following weighted Sobolev space:
\begin{defn}
\begin{align*}
L^{2}(\Gamma \times \R, \delta) := \left \{ f : \Gamma \times \R \to \R \text{ measurable s.t. } \int_{\R} \int_{\Gamma} \delta(z) \abs{f(p,z)}^{2} \dHaus \dz < \infty \right \},
\end{align*}
with the inner product and induced norm:
\begin{align*}
\inner{f}{g}_{L^{2}(\Gamma \times \R, \delta)} := \int_{\R} \int_{\Gamma} \delta(z) f(p,z) g(p,z) \dHaus \dz, \quad \norm{f}_{L^{2}(\Gamma \times \R, \delta)}^{2} := \inner{f}{f}_{L^{2}(\Gamma \times \R, \delta)}^{2},
\end{align*}
along with the identification,
\begin{align*}
f = g \Leftrightarrow f(p,z) = g(p,z) \text{ for a.e. } p \in \Gamma \text{ and a.e. } z \in \{ t \in \R : \delta(t) > 0 \}.
\end{align*}
\end{defn}
In the next two lemma, we give compactness results for $\mathcal{V}_{\eps}$ and $H^{1}(\Omega, \delta_{\eps})$.

\begin{lemma}\label{lem:Vepscompactness}
Suppose that Assumptions \ref{assump:Gamma}, \ref{assump:Xi} and \ref{assump:delta} are satisfied.  Let $\{ u^{\eps} \}_{\eps \in (0,1]} \subset \mathcal{V}_{\eps}$ and $\{ w^{\eps} \}_{\eps \in (0,1]} \subset H^{1}(\Omega, \xi_{\eps})$ denote two bounded sequences, i.e., there exist constants $C > 0$, independent of $\eps$, such that
\begin{align*}
\norm{u^{\eps}}_{1,\xi_{\eps}}^{2} + \norm{u^{\eps}}_{0,\delta_{\eps}}^{2} \leq C, \quad \norm{w^{\eps}}_{1,\xi_{\eps}}^{2} \leq C.
\end{align*}
Then, there exist $\tilde{u}, \tilde{w} \in H^{1}(\Omega^{*})$ such that $u^{\eps} \mid_{\Omega^{*}}$ converges weakly to $\tilde{u}$ in $H^{1}(\Omega^{*})$,  $w^{\eps} \mid_{\Omega^{*}}$ converges weakly to $\tilde{w}$ in $H^{1}(\Omega^{*})$, along subsequences, and for any $\varphi \in H^{1}(\Omega)$, 
\begin{align}\label{convergencetotrace}
\int_{\Omega} \delta_{\eps} u^{\eps} \varphi \dx \to \int_{\Gamma} \gamma_{0}(\tilde{u}) \gamma_{0}(\varphi) \dHaus \text{ as } \eps \to 0.
\end{align}
\end{lemma}
\begin{proof}
By Assumption \ref{assump:Xi}, $\xi_{\eps} \geq \frac{1}{2}$ in $\overline{\Omega^{*}}$, and so, 
\begin{align}\label{H1Xibddconv}
\norm{u^{\eps} \mid_{\Omega^{*}}}_{H^{1}(\Omega^{*})} \leq 2 \int_{\Omega^{*}} \xi_{\eps} \abs{u^{\eps}}^{2} \dx \leq 2 \int_{\Omega} \xi_{\eps} \abs{u^{\eps}}^{2} \dx \leq 2C \text{ for all } \eps \in (0,1].
\end{align}
Hence, $\{ u^{\eps} \mid_{\Omega^{*}} \}_{\eps \in (0,1]} \subset H^{1}(\Omega^{*})$ is a bounded sequence.  Then, by the reflexive weak compactness theorem \cite[Theorem 3, p. 639]{book:Evans}, there exists a function $\tilde{u} \in H^{1}(\Omega^{*})$ such that
\begin{align*}
u^{\eps} \mid_{\Omega^{*}} \rightharpoonup \tilde{u} \text{ in } H^{1}(\Omega^{*}) \text{ as } \eps \to 0,
\end{align*}
along a subsequence.  The proof for the weak convergence of $w^{\eps} \mid_{\Omega^{*}}$ to $\tilde{w}$ in $H^{1}(\Omega^{*})$ follows along the same lines.

Choose $0 < \eta < (\tilde{c})^{-1}$ and the diffeomorphism $\Theta^{\eta}$ as defined in (\ref{diffeo}), where $\tilde{c}$ is the constant in (\ref{CHCRbdd}).  We consider the scaled tubular neighbourhood $X^{\eps} = \mathrm{Tub}^{\eps^{k} \eta}(\Gamma)$ for $0 < k < 1$.  Furthermore, choosing $q = \frac{1}{1-k} > 1$ in (\ref{property:delta:ptwiseconv}), and using (\ref{property:delta}) and a rescaling, we obtain
\begin{align*}
\delta_{\eps}(x) \leq \frac{1}{\eps} \delta \left ( \frac{\eta}{\eps^{1-k}} \right ) = \frac{1}{\tilde{\eps}^{\frac{1}{1-k}}} \delta \left ( \frac{\eta}{\tilde{\eps}} \right ) \to 0 \text{ as } \eps \to 0, 
\end{align*}
for all $x \in \Omega \setminus X^{\eps}$.  Here, we have used that if $x \in \Omega \setminus X^{\eps}$, then $\abs{d(x)} \geq \eps^{k} \eta$.  Thus, we deduce that
\begin{align}\label{delta:remainder:Xeps}
\norm{\delta_{\eps}}_{L^{\infty}(\Omega \setminus X^{\eps})} \to 0 \text{ as } \eps \to 0.
\end{align}
Next, let $\varphi \in H^{1}(\Omega)$.  Then, by the Cauchy--Schwarz inequality and the uniform boundedness of $\{ u^{\eps} \}_{\eps \in (0,1]}$ in $\mathcal{V}_{\eps}$, we have
\begin{equation}\label{int:uepsvarphi:remainder}
\begin{aligned}
\abs{\int_{\Omega \setminus X^{\eps}} \delta_{\eps} u^{\eps} \varphi \dx } & \leq \norm{u^{\eps}}_{L^{2}(\Omega \setminus X^{\eps}, \delta_{\eps})} \norm{\varphi}_{L^{2}(\Omega \setminus X^{\eps}, \delta_{\eps})} \\
& \leq C \norm{\delta_{\eps}}_{L^{\infty}(\Omega \setminus X^{\eps})}^{\frac{1}{2}} \norm{\varphi}_{L^{2}(\Omega)} \to 0 \text{ as } \eps \to 0.
\end{aligned}
\end{equation}
Hence, it suffices to look at the integral over $X^{\eps}$.  Let $U^{\eps}(p,z)$ and $\Phi_{\eps}(p,z)$ denote the representation of $u^{\eps}$ and $\varphi$ in the $(p,z)$ coordinate system, respectively.  For convenience, let us use the notation
\begin{align}
\eta_{\eps,k} := \frac{\eta}{\eps^{1-k}}.
\end{align}
Then, as $u^{\eps}$ is bounded uniformly in $H^{1}(\Omega, \xi_{\eps})$, by (\ref{XiDelta}), (\ref{ScaledNbd:coordchange}), and (\ref{CHCRbdd}), we have
\begin{equation}\label{pdzUeps:bdd}
\begin{aligned}
C & \geq \norm{u^{\eps}}_{1,\xi_{\eps}}^{2} \geq \int_{X^{\eps}} \xi_{\eps} \abs{\nabla u^{\eps}}^{2} \dx \geq \int_{X^{\eps}} \eps C_{\xi} \delta_{\eps} \abs{\nabla u^{\eps}}^{2} \dx \\
& \geq \int_{-\eta_{\eps,k}}^{\eta_{\eps,k}} \int_{\Gamma} \eps C_{\xi} \delta(z) \frac{1}{\eps^{2}} \abs{\pd_{z} U^{\eps}}^{2} \abs{1 + \eps z C_{H}(p) + C_{R}(p,\eps z)} \dHaus \dz \\
& \geq C_{\xi} (1-\tilde{c} \eta) \int_{\R} \int_{\Gamma} \chi_{(-\eta_{\eps,k}, \eta_{\eps,k})}(z) \delta(z) \frac{1}{\eps} \abs{\pd_{z}U^{\eps}}^{2}(p,z) \dHaus \dz.
\end{aligned}
\end{equation}
Multiplying by $\eps$ on both sides of the inequality allows us to deduce that
\begin{align}\label{pdzUeps:convergezero}
\chi_{(-\eta_{\eps,k}, \eta_{\eps,k})} (z) \pd_{z} U^{\eps}(p,z) \to 0 \text{ in } L^{2}(\Gamma \times \R, \delta) \text{ as } \eps \to 0.
\end{align}
By (\ref{property:delta}), $\delta(0) \geq \eps \delta_{\eps}(x)$ for all $x \in \Omega$.  As $\varphi \in H^{1}(\Omega)$, using (\ref{ScaledNbd:coordchange}), and (\ref{property:delta}) we have
\begin{equation}\label{pdzPhieps:bdd}
\begin{aligned}
\norm{\varphi}_{H^{1}(\Omega)}^{2} & \geq \int_{X^{\eps}} \abs{\nabla \varphi}^{2} \dx \geq \frac{1}{\delta(0)} \int_{X^{\eps}} \eps \delta_{\eps} \abs{\nabla \varphi}^{2} \dx \\
& \geq \frac{1-\tilde{c} \eta}{\delta(0)} \int_{\R} \int_{\Gamma} \chi_{(-\eta_{\eps,k}, \eta_{\eps,k})}(z) \delta(z) \frac{1}{\eps} \abs{\pd_{z}\Phi_{\eps}}^{2}(p,z)  \dHaus \dz,
\end{aligned}
\end{equation}
and thus,
\begin{align}\label{pdzPhieps:convergezero}
\chi_{(-\eta_{\eps,k}, \eta_{\eps,k})}(z) \pd_{z}\Phi_{\eps}(p,z) \to 0 \text{ in } L^{2}(\Gamma \times \R, \delta) \text{ as } \eps \to 0.
\end{align}
Meanwhile, by Lemma \ref{lem:H1impliesL2delta}, we see that
\begin{equation}\label{Phieps:bdd}
\begin{aligned}
\norm{\varphi}_{H^{1}(\Omega)}^{2} & \geq C \norm{\varphi}_{0,\delta_{\eps}}^{2} \geq C \int_{X^{\eps}} \delta_{\eps} \abs{\varphi}^{2} \dx \\
& \geq C(1-\tilde{c} \eta) \int_{\R} \int_{\Gamma} \chi_{(-\eta_{\eps,k}, \eta_{\eps,k})}(z) \delta(z) \abs{\Phi_{\eps}}^{2}(p,z) \dHaus \dz.
\end{aligned}
\end{equation}
Similarly, the uniform boundedness of $u^{\eps}$ in $L^{2}(\Omega, \delta_{\eps})$ gives
\begin{equation}\label{Ueps:bdd}
\begin{aligned}
C& \geq \norm{u^{\eps}}_{0,\delta_{\eps}}^{2} \geq \int_{X^{\eps}} \delta_{\eps} \abs{u^{\eps}}^{2} \dx \\
& \geq (1-\tilde{c} \eta) \int_{\R} \int_{\Gamma} \chi_{(-\eta_{\eps,k}, \eta_{\eps,k})}(z) \delta(z) \abs{U^{\eps}}^{2} (p,z) \dHaus \dz.
\end{aligned}
\end{equation}
Hence, by the reflexive weak compactness theorem, there exists a function $\overline{u} \in L^{2}(\Gamma \times \R, \delta)$ such that
\begin{align}\label{Ueps:converge}
\chi_{(-\eta_{\eps,k}, \eta_{\eps,k})}(z) U^{\eps}(p,z) \rightharpoonup \overline{u}(p,z) \text{ in } L^{2}(\Gamma \times \R, \delta) \text{ as } \eps \to 0,
\end{align}
along a subsequence.  By (\ref{pdzUeps:convergezero}) we can deduce that
\begin{align}\label{pdzoverlineu:zero}
\pd_{z} \overline{u} = 0 \text{ on } J := \{ z \in \R : \delta(z) > 0 \},
\end{align}
and so $\overline{u} = \overline{u}(p)$ in $J$.  Indeed, for any $\Psi(p,z)$ that is smooth and compactly supported in $J$, we have from (\ref{pdzUeps:convergezero}),
\begin{align}\label{pdzUepsPsi:convergezero}
\int_{-\eta_{\eps,k}}^{\eta_{\eps,k}} \int_{\Gamma} \delta(z) (\pd_{z}U^{\eps} \Psi)(p,z) \dHaus \dz \to 0 \text{ as } \eps \to 0.
\end{align}
Then, by (\ref{property:delta:bound}), and the smoothness of $\Psi$, we have
\begin{equation*}
\begin{aligned}
& \; \int_{(-\eta_{\eps,k}, \eta_{\eps,k}) \cap J} \int_{\Gamma} \frac{\abs{\delta'(z)}^{2}}{\delta(z)} \abs{\Psi}^{2} (p,z) \dHaus dz\\
 \leq & \; C(\tilde{c} \eta, \abs{\Gamma}) \norm{\Psi}_{L^{\infty}_{loc}(\Gamma \times J)} \int_{(-\eta_{\eps,k}, \eta_{\eps,k}) \cap J} \frac{\abs{\delta'(z)}^{2}}{\delta(z)} \dx < \infty.
\end{aligned}
\end{equation*}
This implies that
\begin{align*}
\frac{\delta'(z)}{\delta(z)} \Psi(p,z) \in L^{2}(\Gamma \times J, \delta),
\end{align*}
and so by (\ref{Ueps:converge}), as $\eps \to 0$,
\begin{align*}
& \; \int_{-\eta_{\eps,k}}^{\eta_{\eps,k}} \int_{\Gamma} \chi_{J}(z) \delta(z) \left ( U^{\eps}(p,z) \frac{\delta'(z)}{\delta(z)} \Psi(p,z) \right ) \dHaus \dz \\
\to & \; \int_{J} \int_{\Gamma} \delta'(z) (\overline{u} \Psi)(p,z) \dHaus dz.
\end{align*}
Thus, we see that, as $\eps \to 0$,
\begin{align*}
& \; \int_{-\eta_{\eps,k}}^{\eta_{\eps,k}} \int_{\Gamma} \delta(z) (\pd_{z}U^{\eps} \Psi)(p,z) \dHaus dz \\
= & \; - \int_{-\eta_{\eps,k}}^{\eta_{\eps,k}} \int_{\Gamma} \chi_{J}(z) \delta'(z) (U^{\eps} \Psi)(p,z) + \chi_{J}(z) \delta(z) (\pd_{z}\Psi U^{\eps}) (p,z) \dHaus \dz \\
\to & \; - \int_{J} \int_{\Gamma} \delta'(z) (\overline{u} \Psi)(p,z) + \delta(z) (\overline{u} \pd_{z}\Psi)(p,z) \dHaus dz \\
= & \; \int_{J} \int_{\Gamma} \delta(z) (\pd_{z} \overline{u} \Psi)(p,z) \dHaus \dz.
\end{align*}
But by (\ref{pdzUepsPsi:convergezero}), the left hand side also converge to zero as $\eps \to 0$.  Hence, for arbitrary $\Psi$ that is smooth and compactly supported in $J$, we have
\begin{align}\label{pdzoverlineu:zero:inJ}
\int_{J} \int_{\Gamma} \delta(z) (\pd_{z}\bar{u} \Psi)(p,z) \dHaus \dz = 0,
\end{align}
which implies that $\pd_{z} \overline{u} = 0$ a.e. on $J$.

To finish the proof, we will show that for all $\varphi \in H^{1}(\Omega)$,
\begin{align}\label{productconvergence}
\int_{X^{\eps}} \delta_{\eps} u^{\eps} \varphi \dx \to \int_{\Gamma} \overline{u} \gamma_{0}(\varphi) \dHaus \text{ as } \eps \to 0,
\end{align}
and then the identification
\begin{align}\label{identification}
\overline{u} = \gamma_{0}(\tilde{u}) \text{ a.e. on } \Gamma.
\end{align}

First, we note that by (\ref{CHCRbdd}) and the definition of $\Phi_{\eps}(p,z)$ (see (\ref{representation:pz})), we have for a.e. $(p,z) \in \Gamma \times \R$, 
\begin{align}\label{aeconvergencePhi}
\chi_{(-\eta_{\eps,k}, \eta_{\eps,k})}(z) \Phi_{\eps}(p,z) \abs{1 + \eps z C_{H}(p) + C_{R}(p,\eps z)}^{\frac{1}{2}} \to \gamma_{0}(\varphi)(p) \text{ as } \eps \to 0.
\end{align}
By Lemma \ref{lem:H1impliesL2delta}, we have the uniform boundedness of the norm:
\begin{equation}\label{normbddPhi}
\begin{aligned}
& \; \int_{-\eta_{\eps,k}}^{\eta_{\eps,k}} \int_{\Gamma} \delta(z) \abs{\Phi_{\eps}}^{2} \abs{1 + \eps z C_{H}(p) + C_{R}(p,\eps z)} \dHaus \dz \\
= & \; \int_{X^{\eps}} \delta_{\eps} \abs{\varphi}^{2} \dx \leq C \norm{\varphi}_{H^{1}(\Omega)}^{2}.
\end{aligned}
\end{equation}
By (\ref{convergetodirac}), we have
\begin{align*}
\int_{\Omega} \delta_{\eps} \abs{\varphi^{2}} \dx \to \int_{\Gamma} \abs{\gamma_{0}(\varphi)}^{2} \dHaus \text{ as } \eps \to 0.
\end{align*}
Furthermore, from (\ref{delta:remainder:Xeps}) we see that
\begin{align*}
\int_{\Omega \setminus X^{\eps}} \delta_{\eps} \abs{\varphi}^{2} \dx \leq \norm{\delta_{\eps}}_{L^{\infty}(\Omega \setminus X^{\eps})} \norm{\varphi}_{L^{2}(\Omega)}^{2} \to 0 \text{ as } \eps \to 0.
\end{align*}
Hence, we deduce that
\begin{align}\label{convergetotrace:Xeps}
\int_{X^{\eps}} \delta_{\eps} \abs{\varphi}^{2} \dx \to \int_{\Gamma} \abs{\gamma_{0}(\varphi)}^{2} \dHaus \text{ as } \eps \to 0.
\end{align}
In particular, we obtain the norm convergence:
\begin{equation}\label{normconvergencePhi}
\begin{aligned}
& \; \int_{-\eta_{\eps,k}}^{\eta_{\eps,k}} \int_{\Gamma} \delta(z) \abs{\Phi_{\eps}}^{2}(p,z) \abs{1 + \eps z C_{H}(p) + C_{R}(p,\eps z)} \dHaus \dz \\
= & \; \int_{X^{\eps}} \delta_{\eps} \abs{\varphi}^{2} \dx \to \int_{\Gamma} \abs{\gamma_{0}(\varphi)}^{2} \dHaus = \int_{\R} \int_{\Gamma} \delta(z) \abs{\gamma_{0}(\varphi)}^{2}(p) \dHaus \dz,
\end{aligned}
\end{equation}
as $\eps \to 0$.  

Almost everywhere convergence (\ref{aeconvergencePhi}) and uniform boundedness of the norm (\ref{normbddPhi}) imply weak convergence in $L^{2}(\Gamma \times \R, \delta)$ \cite[Proposition 4.7.12, p. 282]{book:Bogachev}.  Together with the norm convergence (\ref{normconvergencePhi}) yields strong convergence in $L^{2}(\Gamma \times \R, \delta)$ \cite[Corollary 4.7.16 p. 285]{book:Bogachev}.  I.e., as $\eps \to 0$, 
\begin{equation}\label{strongconvergencePhi}
\begin{aligned}
& \; \chi_{(-\eta_{\eps,k}, \eta_{\eps,k})}(z) \Phi_{\eps}(p,z) \abs{1 + \eps z C_{H}(p) + C_{R}(p,\eps z)}^{\frac{1}{2}} \\
 \to & \; \gamma_{0}(\varphi)(p) \text{ in } L^{2}(\Gamma \times \R, \delta).
\end{aligned}
\end{equation}

Recall that $\eta_{\eps,k} = \frac{\eta}{\eps^{1-k}}$.  Then, by (\ref{CHCRbdd}), 
\begin{align}\label{CHCRsup}
\sup_{(p,z) \in \Gamma \times (-\eta_{\eps,k}, \eta_{\eps,k})} \abs{\eps z C_{H}(p) + C_{R}(p,\eps z)} \leq \sup_{(p,z) \in \Gamma \times (-\eta_{\eps,k}, \eta_{\eps,k})} \tilde{c} \eps \abs{z} \leq \tilde{c} \eps^{k} \eta,
\end{align}
and hence, using the identity
\begin{align*}
\sqrt{a} - \sqrt{b} = \frac{a-b}{\sqrt{a} + \sqrt{b}},
\end{align*}
we obtain, as $\eps \to 0$,
\begin{equation}\label{1CHCRstrongconverge}
\begin{aligned}
& \; \esssup_{(p,z) \in \Gamma \times (-\eta_{\eps,k}, \eta_{\eps,k})} \abs{ \abs{1 + \eps z C_{H}(p) + C_{R}(p,\eps z)}^{\frac{1}{2}} - 1} \\
\leq & \;  \frac{\tilde{c} \eps^{k} \eta}{\sqrt{1 + \tilde{c} \eps^{k} \eta} + 1}  \leq C \eps^{k} \eta \to 0.
\end{aligned}
\end{equation}
By the weak-strong product convergence, we have
\begin{equation}\label{productweakstrong}
\begin{aligned}
& \; \int_{X^{\eps}} \delta_{\eps} u^{\eps} \varphi \dx \\
 = & \; \int_{\R} \int_{\Gamma} \delta(z) \chi_{(-\eta_{\eps,k}, \eta_{\eps,k})}(z) (U^{\eps} \Phi_{\eps})(p,z) \abs{1 + \eps z C_{H}(p) + C_{R}(p,\eps z)} \dHaus \dz \\
\to & \; \int_{\R} \delta(z) \int_{\Gamma} \overline{u}(p) \gamma_{0}(\varphi)(p) \dHaus \dz = \int_{\Gamma} \overline{u} \gamma_{0}(\varphi) \dHaus \text{ as } \eps \to 0.
\end{aligned}
\end{equation}

For the identification (\ref{identification}), we use that $\eps < 1 \Rightarrow \eta_{\eps,k} > \eta$ to see that the weak convergence of $U^{\eps}$ to $\overline{u}$ also holds in the restricted space $L^{2}(\Gamma \times (-\eta, \eta), \delta)$, which is equivalent to restricting to $\mathrm{Tub}^{\eps \eta}(\Gamma)$.  Indeed, for $\Psi \in L^{2}(\Gamma \times \R, \delta)$, we have
\begin{equation}\label{weakconv:Ueps:restricted}
\begin{aligned}
& \; \abs{\int_{-\eta}^{\eta} \int_{\Gamma} \delta(z) (U^{\eps} - \overline{u}) \Psi \dHaus \dz} \\
= & \; \abs{\int_{\R} \int_{\Gamma} \delta(z) (U^{\eps} - \overline{u}) (\chi_{(-\eta, \eta)}(z) \Psi) \dHaus \dz} \to 0,
\end{aligned}
\end{equation}
as $\eps \to 0$.  Let 
\begin{align*}
C_{\eta} := \int_{-\eta}^{\eta} \delta(z) \dz.
\end{align*}
Then, by (\ref{CHCRbdd}), we find that
\begin{align*}
& \; \int_{-\eta}^{\eta} \int_{\Gamma} \delta(z) \abs{\gamma_{0}(\varphi)(p)}^{2} \abs{\abs{1 + \eps z C_{H}(p) + C_{R}(p,\eps z)} - 1} \dHaus \dz \\
\leq & \; \tilde{c} \eta \eps \norm{\gamma_{0}(\varphi)}_{L^{2}(\Gamma)}^{2} \int_{-\eta}^{\eta} \delta(z) \abs{z} \dz \leq C(\tilde{c}, \eta, C_{\delta, \mathrm{int}}) \norm{\gamma_{0}(\varphi)}_{L^{2}(\Gamma)}^{2} \eps \to 0 \text{ as } \eps \to 0,
\end{align*}
and thus we deduce the following strong convergence:
\begin{equation}\label{restricted:strongconvergencePhi}
\begin{aligned}
& \; \chi_{(-\eta, \eta)}(z) \gamma_{0}(\varphi)(p) \abs{1 + \eps z C_{H}(p) + C_{R}(p,\eps z)}^{\frac{1}{2}} \\
\to & \; \gamma_{0}(\varphi)(p) \chi_{(-\eta, \eta)}(z) \text{ in } L^{2}(\Gamma \times \R, \delta), \text{ as } \eps \to 0.
\end{aligned}
\end{equation}

By the Cauchy--Schwarz inequality, (\ref{CHCRbdd}), (\ref{property:delta:bound}) and Lemma \ref{lem:H1impliesL2delta}, 
\begin{align*}
& \; \int_{-\eta}^{\eta} \int_{\Gamma} \delta(z) \abs{\gamma_{0}(\tilde{u})(p) \gamma_{0}(\varphi)(p)} \abs{1 + \eps z C{H}(p) + C_{R}(p,\eps z)} \dHaus \dz \\
\leq & \;  \int_{-\eta}^{\eta} \int_{\Gamma} \delta(z) \abs{\gamma_{0}(\tilde{u})(p)}\abs{\gamma_{0}(\varphi)(p)} (1 + \tilde{c} \eps \abs{z}) \dHaus \dz  \\
\leq & \; (1 + \tilde{c} \eps C_{\delta, \mathrm{int}}) \norm{\gamma_{0}(\tilde{u})}_{L^{2}(\Gamma)}\norm{\gamma_{0}(\varphi)}_{L^{2}(\Gamma)} < \infty.
\end{align*}

Recalling that, by definition (see (\ref{representation:pz})), $\Phi_{\eps}(p,0) = \gamma_{0}(\varphi)$, we can now compute (writing $C_{H} = C_{H}(p)$ and $C_{R} = C_{R}(p,\eps z)$ in the following):
\begin{equation}\label{identification:difference}
\begin{aligned}
& \; \abs{\int_{-\eta}^{\eta} \int_{\Gamma} \delta(z) ((U^{\eps}\Phi_{\eps})(p,z) - (\gamma_{0}(\tilde{u})\gamma_{0}(\varphi))(p)) \abs{1 + \eps z C_{H} + C_{R}} \dHaus \dz} \\
\leq & \; \abs{\int_{-\eta}^{\eta} \int_{\Gamma} \delta(z) (U^{\eps}(p,0) - \gamma_{0}(\tilde{u})(p))\gamma_{0}(\varphi)(p) \abs{1 + \eps z C_{H} + C_{R}} \dHaus \dz} \\
+ & \; \abs{\int_{-\eta}^{\eta} \int_{\Gamma} \delta(z) \left ( \int_{0}^{z} \frac{\dd}{\dzeta} (U^{\eps} \Phi_{\eps})(p,\zeta) \dzeta \right ) \abs{1 + \eps z C_{H} + C_{R}} \dHaus \dz}  \\
\leq & \; (1 + \tilde{c} C_{\delta, \mathrm{int}}) \int_{\Gamma} \abs{\gamma_{0}(u^{\eps}) - \gamma_{0}(\tilde{u})} \abs{\gamma_{0}(\varphi)} \dHaus \\
+ & \; (1 + \tilde{c} \eta) \int_{-\eta}^{\eta} \delta(z) \int_{\Gamma} \int_{-\eta}^{\eta} \abs{\frac{\dd}{\dzeta} (U^{\eps} \Phi_{\eps})(p,\zeta)} \dzeta \dHaus \dz.
\end{aligned}
\end{equation}

The first term on the right hand side converges to zero by the strong convergence of $\gamma_{0}(u^{\eps})$ to $\gamma_{0}(\tilde{u})$ in $L^{2}(\Gamma)$ (see Theorem \ref{thm:compacttrace}).  For the second term, using the monotonicity of $\delta(\abs{s})$, and the definition of $C_{\eta}$,
\begin{equation}\label{deltadenominator}
\begin{aligned}
& \; \int_{-\eta}^{\eta} \delta(z) \int_{\Gamma} \int_{-\eta}^{\eta} \frac{\delta(\zeta)}{\delta(\zeta)}\abs{\frac{\dd}{\dzeta} (U^{\eps} \Phi_{\eps})(p,\zeta)} \dzeta \dHaus \dz \\
\leq & \; \frac{C_{\eta}}{\delta(\eta)} \int_{\Gamma} \int_{-\eta}^{\eta} \delta(\zeta) (\abs{\pd_{\zeta} U^{\eps}(p,\zeta) \Phi_{\eps}(p,\zeta)} + \abs{U^{\eps}(p,\zeta) \pd_{\zeta}\Phi_{\eps}(p,\zeta)}) \dzeta \dHaus \\
\leq & \; C\norm{\pd_{z}U^{\eps}}_{L^{2}(\Gamma \times (-\eta, \eta), \delta)} \norm{\Phi_{\eps}}_{L^{2}(\Gamma \times (-\eta, \eta), \delta)} \\
+ & \; C\norm{U^{\eps}}_{L^{2}(\Gamma \times (-\eta, \eta), \delta)} \norm{\pd_{z}\Phi_{\eps}}_{L^{2}(\Gamma \times (-\eta, \eta), \delta)} \to 0,
\end{aligned}
\end{equation}
as $\eps \to 0$, where we used (\ref{pdzUeps:convergezero}), (\ref{pdzPhieps:convergezero}), (\ref{Phieps:bdd}), and (\ref{Ueps:bdd}) to show that 
\begin{align*}
& \; \norm{\pd_{z}U^{\eps}}_{L^{2}(\Gamma \times (-\eta, \eta), \delta)} \norm{\Phi_{\eps}}_{L^{2}(\Gamma \times (-\eta, \eta), \delta)} \\
 + & \; \norm{U^{\eps}}_{L^{2}(\Gamma \times (-\eta, \eta), \delta)} \norm{\pd_{z}\Phi_{\eps}}_{L^{2}(\Gamma \times (-\eta, \eta), \delta)} \\
\leq & \; C \norm{\varphi}_{H^{1}(\Omega)} \norm{\chi_{(-\eta_{\eps,k}, \eta_{\eps,k})}(z) \pd_{z}U^{\eps}}_{L^{2}(\Gamma \times \R, \delta)}\\
 + & \; C\norm{u^{\eps}}_{0, \delta_{\eps}} \norm{\chi_{(-\eta_{\eps,k}, \eta_{\eps,k})}(z) \pd_{z} \Phi_{\eps}}_{L^{2}(\Gamma \times \R, \delta)} \to 0,
\end{align*}
as $\eps \to 0$.  As a consequence of (\ref{identification:difference}) and (\ref{deltadenominator}), we have
\begin{equation}\label{differencelimitzero}
\begin{aligned}
& \; \int_{-\eta}^{\eta} \int_{\Gamma} \delta(z)((U^{\eps}\Phi_{\eps})(p,z) - (\gamma_{0}(\tilde{u})\gamma_{0}(\varphi))(p)) \abs{1 + \eps z C_{H} + C_{R}} \dHaus \dz \\
\to & \;0 \text{ as } \eps \to 0.
\end{aligned}
\end{equation}
But, by (\ref{CHCRbdd}) and (\ref{restricted:strongconvergencePhi}),
\begin{align*}
& \; \int_{-\eta}^{\eta} \int_{\Gamma} \delta(z) \gamma_{0}(\tilde{u})(p) \gamma_{0}(\varphi)(p) \abs{1 + \eps z C_{H} + C_{R}} \dHaus \dz \\
\to & \; \int_{-\eta}^{\eta} \int_{\Gamma} \delta(z) \gamma_{0}(\tilde{u})(p) \gamma_{0}(\varphi)(p) \dHaus \dz = C_{\eta} \int_{\Gamma} \gamma_{0}(\tilde{u}) \gamma_{0}(\varphi) \dHaus \text{ as } \eps \to 0.
\end{align*}
Together with the weak convergence of $U^{\eps}$ to $\overline{u}$ in $L^{2}(\Gamma \times (-\eta, \eta), \delta)$ (see (\ref{weakconv:Ueps:restricted})), we have
\begin{align*}
& \; \int_{-\eta}^{\eta} \int_{\Gamma} \delta(z) ((U^{\eps} \Phi_{\eps})(p,z) - (\gamma_{0}(\tilde{u}) \gamma_{0}(\varphi))(p)) \abs{1 + \eps z C_{H} + C_{R}} \dHaus \dz \\
\to & \; \int_{-\eta}^{\eta} \int_{\Gamma} \delta(z) (\overline{u}(p) - \gamma_{0}(\tilde{u})(p)) \gamma_{0}(\varphi)(p) \dHaus \dz = C_{\eta} \int_{\Gamma} (\overline{u} - \gamma_{0}(\tilde{u})) \gamma_{0}(\varphi) \dHaus,
\end{align*}
as $\eps \to 0$.  Hence, by (\ref{differencelimitzero}), we have
\begin{align*}
0 = C_{\eta} \int_{\Gamma}  (\overline{u} - \gamma_{0}(\tilde{u})) \gamma_{0}(\varphi) \dHaus \text{ for all } \varphi \in H^{1}(\Omega),
\end{align*}
which implies that $\overline{u} = \gamma_{0}(\tilde{u})$ a.e. on $\Gamma$.
\end{proof}

\begin{remark}
For the identification $(\ref{identification})$, we restricted to the tubular neighbourhood $\mathrm{Tub}^{\eps \eta}(\Gamma)$, so that in the $(p,z)$ coordinate system, we have $z \in (-\eta, \eta)$.  This is needed in $(\ref{deltadenominator})$ where we estimate the fraction $\frac{1}{\delta(\zeta)}$ from above.   Otherwise, we cannot deduce that the right hand side of $(\ref{identification:difference})$ converges to zero as $\eps \to 0$.
\end{remark}

\begin{remark}
For the double-obstacle regularisation, we choose $X^{\eps} = \mathrm{Tub}^{\eps \frac{\pi}{2}}(\Gamma)$ and we will have $U^{\eps}$ converging weakly to $\overline{u}(p)$ in $L^{2}(\Gamma \times (-\frac{\pi}{2},\frac{\pi}{2}), \delta)$.  The estimate $(\ref{CHCRsup})$ becomes
\begin{align*}
\sup_{p \in \Gamma, z \in (-\frac{\pi}{2}, \frac{\pi}{2})} \abs{ \eps z C_{H}(p) + C_{R}(p, \eps z)} \leq \tilde{c} \eps \frac{\pi}{2}.
\end{align*}
A similar argument with the above elements will show $(\ref{productconvergence})$, and we restrict to $\mathrm{Tub}^{\eps \frac{\pi}{4}}(\Gamma)$ in order to show $(\ref{identification})$.
\end{remark}

\begin{cor}\label{cor:CompactnessWeps}
Suppose that Assumptions \ref{assump:Gamma}, \ref{assump:Xi} and \ref{assump:delta} are satisfied.  Let $\{ w^{\eps} \}_{\eps \in (0,1]} \subset \mathcal{W}_{\eps}$ be a bounded sequence, i.e., there exist a constant $C > 0$, independent of $\eps$, such that
\begin{align*}
\norm{w^{\eps}}_{1,\xi_{\eps}}^{2} + \norm{w^{\eps}}_{0,\frac{1}{\eps} \delta_{\eps}}^{2} = \norm{w^{\eps}}_{1,\xi_{\eps}}^{2} + \frac{1}{\eps} \norm{w^{\eps}}_{0,\delta_{\eps}}^{2} \leq C.
\end{align*}
Then, there exists $\tilde{w} \in H^{1}_{0}(\Omega^{*})$ such that $w^{\eps} \mid_{\Omega^{*}}$ converges weakly to $\tilde{w}$ in $H^{1}(\Omega^{*})$.
\end{cor}

\begin{proof}
For $\eps \in (0,1]$, we see that $\norm{w^{\eps}}_{0,\delta_{\eps}}^{2} \leq C \eps \leq C$, and thus, by Lemma \ref{lem:Vepscompactness}, there exists a function $\tilde{w} \in H^{1}(\Omega^{*})$ such that, for all $\varphi \in H^{1}(\Omega)$,
\begin{align}
\notag w^{\eps} \mid_{\Omega^{*}} & \rightharpoonup \tilde{w} \text{ in } H^{1}(\Omega^{*}) \text{ as } \eps \to 0, \\
\label{Weps:integrallimit} \int_{\Omega} \delta_{\eps} w^{\eps} \varphi \dx & \to \int_{\Gamma} \gamma_{0}(\tilde{w}) \gamma_{0}(\varphi) \dHaus \text{ as } \eps \to 0,
\end{align}
along a subsequence.  

On the other hand, we have by Lemma \ref{lem:H1impliesL2delta}, Cauchy--Schwarz inequality and the uniform boundedness of $\{ w^{\eps}\}_{\eps \in (0,1]}$ in $\mathcal{W}_{\eps}$:
\begin{align*}
\abs{\int_{\Omega} \delta_{\eps} w^{\eps} \varphi \dx} & = \abs{\int_{\Omega} \frac{\sqrt{\eps}}{\sqrt{\eps}} \delta_{\eps} w^{\eps} \varphi \dx}  \leq \left ( \int_{\Omega} \frac{1}{\eps} \delta_{\eps} \abs{w^{\eps}} \dx \right )^{\frac{1}{2}} \left ( \int_{\Omega} \eps \delta_{\eps} \abs{\varphi}^{2} \dx \right )^{\frac{1}{2}} \\
& \leq \norm{w^{\eps}}_{0,\frac{1}{\eps} \delta_{\eps}} \sqrt{\eps} \norm{\varphi}_{0,\delta_{\eps}} \leq C \sqrt{\eps} \norm{\varphi}_{H^{1}(\Omega)} \to 0 \text{ as } \eps \to 0.
\end{align*}

Upon comparing with (\ref{Weps:integrallimit}), we see that
\begin{align*}
\int_{\Gamma} \gamma_{0}(\tilde{w}) \gamma_{0}(\varphi) \dHaus = 0 \text{ for all } \varphi \in H^{1}(\Omega),
\end{align*}
and so, $\gamma_{0}(\tilde{w}) = 0$ a.e. on $\Gamma$.
\end{proof}

\begin{lemma} \label{lem:H1deltabddconv}
Suppose that $\Gamma$ is a $C^{3}$ compact hypersurface and that Assumptions \ref{assump:Gamma} and \ref{assump:delta} are satisfied.  Let $\{ v^{\eps}\}_{\eps \in (0,1]} \subset H^{1}(\Omega, \delta_{\eps})$ be a bounded sequence, i.e., there exists a constant $C > 0$, independent of $\eps$, such that, $\norm{v^{\eps}}_{1,\delta_{\eps}}^{2} \leq C$.  For $1 \leq i,j \leq n$, let $g^{Ec} \in L^{2}(\Omega)$, $b^{Ec} \in L^{\infty}(\Omega)$ and $b_{ij}^{Ec} \in L^{\infty}(\Omega)$ denote the extensions of the data $g \in L^{2}(\Gamma)$, $b \in L^{\infty}(\Gamma)$ and $b_{ij} \in L^{\infty}(\Gamma)$ as mentioned in Assumption \ref{assump:surfacedata} and Remark \ref{rem:extensionambiguity}.

Then, there exists $\bar{v} \in H^{1}(\Gamma)$ such that, 

\begin{align}\label{H1deltabdd:vepsg}
\int_{\Omega} \delta_{\eps} v^{\eps} g^{Ec} \dx \to \int_{\Gamma} \overline{v} g \dHaus,
\end{align}
and for any $\varphi \in H^{1}(\Omega)$ and $\psi \in H^{1}(\Gamma)$ with $\psi^{Ec} \in H^{1}(\Omega)$ as constructed in Corollary \ref{cor:constextH1},
\begin{equation}\label{H1deltabdd:convstatement}
\begin{aligned}
\int_{\Omega} \delta_{\eps} b^{Ec} v^{\eps} \varphi \dx & \to \int_{\Gamma} b \overline{v} \gamma_{0}(\varphi) \dHaus, \\
\int_{\Omega} \delta_{\eps} \mathcal{B}^{Ec} \nabla v^{\eps} \cdot \nabla \psi^{Ec} \dx & \to \int_{\Gamma} \mathcal{B} \surf \overline{v} \cdot \surf \psi \dHaus,
\end{aligned}
\end{equation}
as $\eps \to 0$.
\end{lemma}

\begin{proof}
Choose $0 < \eta < (\tilde{c})^{-1}$ and the diffeomorphism $\Theta^{\eta}$ as defined in (\ref{diffeo}).  Let  $X^{\eps} = \mathrm{Tub}^{\eps^{k} \eta}(\Gamma)$ for $0 < k < 1$.  By (\ref{delta:remainder:Xeps}), and the uniform boundedness of $\{ v^{\eps} \}_{ \eps \in (0,1]}$ in $H^{1}(\Omega, \delta_{\eps})$, 
\begin{align*}
\abs{\int_{\Omega \setminus X^{\eps}} \delta_{\eps} v^{\eps} g^{Ec} \dx} \leq \norm{v^{\eps}}_{0,\delta_{\eps}} \norm{\delta_{\eps}}_{L^{\infty}(\Omega \setminus X^{\eps})}^{\frac{1}{2}} \norm{g^{Ec}}_{L^{2}(\Omega)}  \to 0 \text{ as } \eps \to 0,
\end{align*}
and
\begin{align*}
\abs{\int_{\Omega \setminus X^{\eps}} \delta_{\eps} b^{Ec} v^{\eps} \varphi \dx} \leq \norm{b^{Ec}}_{L^{\infty}(\Omega)} \norm{v^{\eps}}_{0,\delta_{\eps}} \norm{\delta_{\eps}}_{L^{\infty}(\Omega \setminus X^{\eps})}^{\frac{1}{2}} \norm{\varphi}_{L^{2}(\Omega)} \to 0 \text{ as } \eps \to 0,
\end{align*}
for any $\varphi \in H^{1}(\Omega)$.  Similarly, by (\ref{delta:remainder:Xeps}), and the uniform boundedness of $\{ v^{\eps} \}_{ \eps \in (0,1]}$ in $H^{1}(\Omega, \delta_{\eps})$, for any $\psi^{Ec} \in H^{1}(\Omega)$ constructed from $\psi \in H^{1}(\Gamma)$ as in Corollary \ref{cor:constextH1}, 
\begin{align*}
& \; \abs{\int_{\Omega \setminus X^{\eps}} \delta_{\eps} \mathcal{B}^{Ec} \nabla v^{\eps} \cdot \nabla \psi^{Ec} \dx} \\
\leq & \; \norm{\mathcal{B}^{Ec}}_{L^{\infty}(\Omega)} \norm{\nabla v^{\eps}}_{0,\delta_{\eps}} \norm{\delta_{\eps}}_{L^{\infty}(\Omega \setminus X^{\eps})}^{\frac{1}{2}} \norm{ \nabla \psi^{Ec}}_{L^{2}(\Omega)} \to 0 \text{ as } \eps \to 0.
\end{align*} 
Hence, it is sufficient to restrict our attention to $X^{\eps}$.  Then, invoking the $(p,z)$ coordinate system, and using (\ref{ScaledNbd:coordchange}), and (\ref{ScaledNbd:intf:upplowbdd}),
\begin{equation}\label{vepsbddH1delta}
\begin{aligned}
C & \geq \norm{v^{\eps}}_{1,\delta_{\eps}}^{2} \geq \int_{X^{\eps}} \delta_{\eps} (\abs{v^{\eps}}^{2} + \abs{\nabla v^{\eps}}^{2}) \dx \\
& \geq (1-\tilde{c} \eta) \int_{-\eta_{\eps,k}}^{\eta_{\eps,k}} \int_{\Gamma} \delta(z) \left ( \abs{V^{\eps}}^{2} + \frac{1}{\eps^{2}} \abs{\pd_{z} V^{\eps}}^{2} + \abs{\surfz V^{\eps}}^{2} \right )(p,z) \dHaus \dz.
\end{aligned}
\end{equation}
Hence, there exists a function $\overline{v} \in L^{2}(\Gamma \times \R, \delta)$ and a vector-valued function $\bm{Q} \in (L^{2}(\Gamma \times \R, \delta))^{n}$ such that, as $\eps \to 0$,
\begin{alignat*}{3}
\chi_{(-\eta_{\eps,k}, \eta_{\eps,k})}(z) \pd_{z} V^{\eps}(p,z) & \to 0 && \text{ in } L^{2}(\Gamma \times \R, \delta), \\
\chi_{(-\eta_{\eps,k}, \eta_{\eps, k})}(z) V^{\eps}(p,z) & \rightharpoonup \overline{v}(p,z) && \text{ in } L^{2}(\Gamma \times \R, \delta), \\
\chi_{(-\eta_{\eps,k}, \eta_{\eps,k})}(z) \surfz V^{\eps}(p,z) & \rightharpoonup \bm{Q}(p,z) && \text{ in } (L^{2}(\Gamma \times \R, \delta))^{n}.
\end{alignat*}
Moreover, we can deduce, via a similar argument to the derivation of (\ref{pdzoverlineu:zero:inJ}), that $\pd_{z} \overline{v} = 0$ and $\overline{v} = \overline{v}(p)$ a.e. on $J$.

We claim that $\bm{Q} = \surf \overline{v}(p)$.  As $\Gamma$ is a $C^{3}$ compact hypersurface, and so for a finite open cover $\{ W_{i} \cap \Gamma \}_{i=1}^{N}, W_{i} \subset \R^{n}$ of $\Gamma$, we have local regular parameterisations $\alpha_{i} : \mathcal{S}_{i} \to W_{i} \cap \Gamma$.  Let $\{ \mu_{i} \}_{i=1}^{N}$ denote a partition of unity subordinate to $\{ W_{i} \}_{i=1}^{N}$.  Take $Y \in C^{\infty}_{c}(\Gamma \times \R)$, and for any $1 \leq r \leq n$, let $\tilde{Q}_{r}(s,z)$, $\tilde{V}^{\eps}(s,z)$, $\tilde{Y}(s,z)$, $\surfzloc \tilde{V}^{\eps}(s,z)$ denote the representation of $Q_{r}$, $V^{\eps}$, $Y$, $\surfz V^{\eps}$ in the $(s,z)$ coordinate system.  Then, 
\begin{align*}
& \; \int_{\R} \int_{\Gamma} \delta(z) (Q_{r} Y)(p,z) \dHaus \dz \\
= & \; \sum_{i=1}^{N} \int_{\R} \int_{\mathcal{S}_{i}} \mu_{i}(\alpha_{i}(s)) \delta(z) (\tilde{Q}_{r} \tilde{Y})(s,z) \abs{\det J_{i,0}(s)} \ds \dz.
\end{align*}

Moreover, since $\Gamma$ is $C^{3}$, the normal $\nu$ and the components $g_{0}^{ij}$ of the inverse of the metric tensor $\mathcal{G}_{0}$ on $\Gamma$ both belong to the class $C^{2}$.  From (\ref{metrictensor:eps}), we infer that the components $g_{\eps}^{ij}$ of the inverse of the metric tensor $\mathcal{G}_{\eps z}$ are $C^{1}$ functions.  So, for any $1 \leq r \leq n$, we have by (\ref{Geps}) and integration by parts,
\begin{equation}\label{surfgradientIBP}
\begin{aligned}
& \; \int_{\mathcal{S}_{i}} \mu_{i}(\alpha_{i}(s)) \delta(z) (\surfzloc \tilde{V}^{\eps})_{r}(s,z) \tilde{Y}(s,z) \abs{\det J_{i,0}(s)} \ds \\
= & \; \int_{\mathcal{S}_{i}} \mu_{i}(\alpha_{i}(s)) \delta(z) \sum_{j=1}^{n-1} g_{\eps}^{jr}(s,z) \pd_{s_{j}} \tilde{V}^{\eps}(s,z) \pd_{s_{r}} G_{\eps}(s,z) \tilde{Y}(s,z) \abs{\det J_{i,0}(s)} \ds \\
= & \; - \int_{\mathcal{S}_{i}} \delta(z) \sum_{j=1}^{n-1} \pd_{s_{j}} \left ( g_{\eps}^{jr}(s,z) \mu_{i}(\alpha_{i}(s)) \tilde{Y}(s,z) \abs{\det J_{i,0}(s)} \pd_{s_{r}} G_{\eps}(s,z) \right ) \tilde{V}^{\eps}(s,z) \ds
\end{aligned}
\end{equation}
Note that, by the regularity of $\Gamma$ and $Y$, 
\begin{align*}
\pd_{s_{j}} \left ( g_{\eps}^{jr} \mu_{i} \tilde{Y} \abs{\det J_{i,0}} \pd_{s_{r}} G_{\eps}  \right ) \in L^{2}(\mathcal{S}_{i} \times \R, \delta).
\end{align*}
Hence, integrating with respect to $z$, summing from $i = 1$ to $N$ and passing to the limit as $\eps \to 0$ in (\ref{surfgradientIBP}), we obtain
\begin{align*}
& \; \sum_{i=1}^{N} \int_{-\eta_{\eps,k}}^{\eta_{\eps,k}} \int_{\mathcal{S}_{i}} \mu_{i}(\alpha_{i}(s)) \delta(z) ((\surfzloc \tilde{V}^{\eps})_{r} \tilde{Y})(s,z) \abs{\det J_{i,0}(s)} \ds \dz \\
\to & \; - \sum_{i=1}^{N} \int_{\R} \int_{\mathcal{S}_{i}} \delta(z) \sum_{j=1}^{n-1} \pd_{s_{j}} \left ( g_{0}^{jr}(s) \mu_{i}(\alpha_{i}(s)) \tilde{Y}(s,z) \abs{\det J_{i,0}(s)} \pd_{s_{r}} \alpha_{i} \right ) \overline{v}(s) \ds \dz \\
= & \; \sum_{i=1}^{N} \int_{\R} \int_{\mathcal{S}_{i}} \mu_{i}(\alpha_{i}(s)) \delta(z) \sum_{j=1}^{n-1} g_{0}^{jr}(s) \pd_{s_{j}} \overline{v}(s) \pd_{s_{r}} \alpha_{i}(s) \tilde{Y}(s,z) \abs{\det J_{i,0}(s)} \ds \dz \\
= & \; \int_{\R} \int_{\Gamma} \delta(z) (\surf \overline{v})_{r}(p) Y(p,z) \dHaus \dz.
\end{align*}
Meanwhile, by the weak convergence of $\chi_{(-\eta_{\eps,k}, \eta_{\eps,k})}(z) (\surfzloc \tilde{V}^{\eps})_{r}$ to $\tilde{Q}_{r}$, we have
\begin{align*}
& \; \sum_{i=1}^{N} \int_{\eta_{\eps,k}}^{\eta_{\eps,k}} \int_{\mathcal{S}_{i}} \mu_{i}(\alpha_{i}(s)) ((\surfzloc \tilde{V}^{\eps})_{r} \tilde{Y})(s,z) \abs{\det J_{i,0}(s)} \ds \dz \\
\to & \; \sum_{i=1}^{N} \int_{\R} \int_{\mathcal{S}_{i}} \mu_{i}(\alpha_{i}(s)) \delta(z) (\tilde{Q}_{r} \tilde{Y})(s,z) \abs{\det J_{i,0}(s)} \ds \dz \text{ as } \eps \to 0.
\end{align*}
Transforming to the $(p,z)$ coordinate system and equating the limits leads to
\begin{align*}
\int_{\R} \int_{\Gamma} \delta(z) (Q_{r} Y)(p,z) \dHaus \dz = \int_{\R} \int_{\Gamma}(\surf \overline{v})_{r}(p) Y(p,z) \dHaus \dz,
\end{align*}
and since $Y$ is arbitrary, we have that $Q_{r} = (\surf \overline{v})_{r}$.

Let $g^{Ec}$, $b^{Ec}$ and $\mathcal{B}^{Ec}$ denote the extensions of the data $g$, $b$ and $\mathcal{B}$ as mentioned in Assumption \ref{assump:surfacedata} and Remark \ref{rem:extensionambiguity}.  We note that by the definition of the extensions:
\begin{align*}
g^{Ec}(x) = g^{e}(x) = g(p(x)), \quad b^{Ec}(x) = b^{e}(x) = b(p(x)), \quad \mathcal{B}^{Ec}(x) = \mathcal{B}^{e}(x) = \mathcal{B}(p(x)),
\end{align*}
for any $x \in X^{\eps}$.  To show (\ref{H1deltabdd:vepsg}) it suffices to show the strong convergence in $L^{2}(\Gamma \times \R, \delta)$:
\begin{align}\label{gstrongconv}
\chi_{(-\eta_{\eps,k}, \eta_{\eps,k})}(z) g(p) \abs{1 + \eps z C_{H}(p) + C_{R}(p, \eps z)}^{\frac{1}{2}} \to g(p) \text{ as } \eps \to 0.
\end{align}
We use the method in the proof of Lemma \ref{lem:Vepscompactness}.  Note that (\ref{gstrongconv}) holds pointwise for a.e. $p \in \Gamma$ and $z \in \R$.  Moreover, by Corollary \ref{cor:constextH1}, we have uniform boundedness of the norm:
\begin{align*}
\int_{X^{\eps}} \delta_{\eps} \abs{g^{e}}^{2} \dx \leq C \norm{g}_{L^{2}(\Gamma)}.
\end{align*}
Furthermore, we may appeal to $(\ref{convergence:constext})_{1}$ to show norm convergence, since the proof also works for the extension of a function in $L^{2}(\Gamma)$.  Then, arguing as in the proof of Lemma \ref{lem:Vepscompactness} yields the required strong convergence.  Using the weak convergence $\chi_{(-\eta_{\eps,k}, \eta_{\eps,k})}(z) V^{\eps}$ to $\overline{v}$ in $L^{2}(\Gamma \times \R, \delta)$, we obtain by the weak-strong product convergence:
\begin{align*}
& \; \int_{X^{\eps}} \delta_{\eps} v^{\eps} g^{e} \dx \\
 = & \; \int_{\R} \int_{\Gamma} \delta(z) \chi_{(-\eta_{\eps,k}, \eta_{\eps,k})}(z) V^{\eps}(p,z) g(p) \abs{1 + \eps z C_{H}(p) + C_{R}(p, \eps z)} \dHaus \dz \\
\to & \; \int_{\R} \int_{\Gamma} \delta(z) \overline{v}(p) g(p) \dHaus \dz  = \int_{\Gamma}  \overline{v} g \dHaus \text{ as } \eps \to 0.
\end{align*}

We note that the proof of $(\ref{H1deltabdd:convstatement})_{1}$ is similar to that of (\ref{productconvergence}).  More precisely, we replace $u^{\eps}$ with $v^{\eps}$, and $\overline{u}$ with $\overline{v}$.  Using that $b^{Ec}(x) = b(p(x))$ for all $x \in X^{\eps}$, and (\ref{1CHCRstrongconverge}), we can deduce the strong convergence 
\begin{align*}
& \; b^{Ec}(p + \eps z \nu(p)) \abs{1 + \eps z C_{H}(p) + C_{R}(p, \eps z)}^{\frac{1}{2}} \\
= & \; b(p)\abs{1 + \eps z C_{H}(p) + C_{R}(p, \eps z)}^{\frac{1}{2}}  \to b(p) \text{ in } L^{\infty}(\Gamma \times \R, \delta).
\end{align*}
Together with (\ref{strongconvergencePhi}) and the weak convergence of $\chi_{(-\eta_{\eps,k}, \eta_{\eps,k})}(z) V^{\eps}$ to $\overline{v}$ in $L^{2}(\Gamma \times \R, \delta)$, we obtain by the weak-strong product convergence:
\begin{align*}
& \; \int_{X^{\eps}} \delta_{\eps} b^{Ec} v^{\eps} \varphi \dx \\
= & \; \int_{\R} \int_{\Gamma} \delta(z) \chi_{(-\eta_{\eps,k}, \eta_{\eps,k})}(z) (V^{\eps} \Phi_{\eps})(p,z) b(p) \abs{1 + \eps z C_{H}(p) + C_{R}(p, \eps z)} \dHaus \dz \\
\to & \; \int_{\R} \int_{\Gamma} \delta(z) \overline{v}(p) \gamma_{0}(\varphi)(p) b(p) \dHaus \dz  = \int_{\Gamma} b \overline{v} \gamma_{0}(\varphi) \dHaus \text{ as } \eps \to 0.
\end{align*}

Next, recall that for a given $\psi \in H^{1}(\Gamma)$ with extension $\psi^{Ec} \in H^{1}(\Omega)$ as constructed in Corollary \ref{cor:constextH1}, within the scaled tubular neighbourhood $X^{\eps}$, $\psi^{Ec}$ is simply the constant extension $\psi^{e}$ in the normal direction off $\Gamma$, as defined in (\ref{defn:constext}).  The same is true for the extension of the data $\mathcal{B}$, i.e., $\mathcal{B}^{Ec}(x) = \mathcal{B}(p(x))$ for all $x \in X^{\eps}$.  By Corollary \ref{cor:constextH1}, we have that $\psi^{e} \in H^{1}(X^{\eps})$ and by (\ref{ConstExt:relationgradients}) we have $\nabla \psi^{e}(x) \cdot \nu(p(x)) = 0$.  Moreover, let $\Psi_{\eps}$ denote the representation of $\psi$ in the $(p,z)$ coordinate system.  Then, by (\ref{gradientf:decomposition}) and (\ref{gradientremainderdotnuzero}), we see that, for $x = p + \eps z \nu(p) \in X^{\eps}$,
\begin{align}\label{pdzpsiezero}
0 = \nabla \psi^{e}(x) \cdot \nu(p(x)) = \frac{1}{\eps} \pd_{z} \Psi_{\eps}(p,z) + \surfz \Psi_{\eps}(p,z) \cdot \nu(p) = \frac{1}{\eps} \pd_{z} \Psi_{\eps}(p,z),
\end{align}
i.e. $\pd_{z} \Psi_{\eps} = 0$.  Together with (\ref{gradientremainderdotnuzero}), we compute that 
\begin{equation}\label{transformBnablavepsnablapsi}
\begin{aligned}
& \; \int_{X^{\eps}} \delta_{\eps} \mathcal{B}^{e} \nabla v^{\eps} \cdot \nabla \psi^{e} \dx \\
= & \; \int_{-\eta_{\eps,k}}^{\eta_{\eps,k}} \int_{\Gamma} \delta(z) \surfz V^{\eps} \cdot \mathcal{B}(p)^{T} \surfz \Psi_{\eps} \abs{1 + \eps z C_{H}(p) + C_{R}(p, \eps z)} \dHaus \dz,
\end{aligned}
\end{equation}
where $\mathcal{B}^{T}$ denotes the transpose of $\mathcal{B}$.

With the weak convergence $\chi_{(-\eta_{\eps,k}, \eta_{\eps,k})}(z) \surfz V^{\eps}(p,z)$ to $\surf \overline{v}(p)$ in $L^{2}(\Gamma \times \R, \delta)$, in order to show 
\begin{align*}
\int_{X^{\eps}} \delta_{\eps} \mathcal{B}^{e} \nabla v^{\eps} \cdot \nabla \psi^{e} \dx \to \int_{\Gamma} \mathcal{B} \surf \overline{v} \cdot \surf \psi \dHaus \text{ as } \eps \to 0,
\end{align*}
it is sufficient to show the following strong convergence result: As $\eps \to 0$,
\begin{equation}\label{BsurfpsiStrong}
\begin{aligned}
& \; \chi_{(-\eta_{\eps,k}, \eta_{\eps,k})}(z) \mathcal{B}^{T}(p) \surfz \Psi_{\eps}(p,z) \abs{1 + \eps z C_{H}(p) + C_{R}(p,\eps z)}^{\frac{1}{2}} \\ \to & \; \mathcal{B}^{T}(p) \surf \psi(p) \text{ in } L^{2}(\Gamma \times \R, \delta).
\end{aligned}
\end{equation}

By (\ref{representation:pz}), (\ref{CHCRbdd}), and (\ref{gradientremainderOrdereps}), we have for almost every $p \in \Gamma$, $z \in \R$, as $\eps \to 0$,
\begin{equation}\label{Bsurfpsipointwise}
\begin{aligned}
& \; \chi_{(-\eta_{\eps,k}, \eta_{\eps,k})}(z) \mathcal{B}^{T}(p) \surfz \Psi_{\eps}(p,z) \abs{1 + \eps z C_{H}(p) + C_{R}(p, \eps z)}^{\frac{1}{2}}  \\
\to & \; \mathcal{B}^{T}(p) \surf \psi(p).
\end{aligned}
\end{equation}

By (\ref{pdzpsiezero}), (\ref{gradientf:decomposition}) and a slight modification to the proof of Corollary \ref{cor:constextH1}, we have uniform boundedness of the norm:
\begin{equation}\label{Bsurfpsiuniformnormbdd}
\begin{aligned}
& \; \int_{\R} \int_{\Gamma} \delta(z) \chi_{(-\eta_{\eps,k}, \eta_{\eps,k})}(z) \abs{\mathcal{B}^{T}\surfz \Psi_{\eps}}^{2} \abs{1 + \eps z C_{H}(p) + C_{R}(p, \eps z)} \dHaus \dz \\
= & \; \int_{X^{\eps}} \delta_{\eps} \abs{(\mathcal{B}^{T})^{e} \nabla \psi^{e}}^{2} \dx \leq C \norm{\mathcal{B}}_{L^{\infty}(\Gamma)} \norm{\surf \psi}_{L^{2}(\Gamma)}^{2}.
\end{aligned}
\end{equation}
Lastly, by a slight modification of the proof of $(\ref{convergence:constextbB})_{2}$ and using that
\begin{align*}
\int_{\Omega \setminus X^{\eps}} \delta_{\eps} \abs{(\mathcal{B}^{Ec})^{T} \nabla \psi^{Ec}}^{2} \dx \leq \norm{\delta_{\eps}}_{L^{\infty}(\Omega \setminus X^{\eps})} \norm{\mathcal{B}^{Ec}}_{L^{\infty}(\Omega)} \norm{\nabla \psi^{Ec}}_{L^{2}(\Omega)}^{2} \to 0 \text{ as } \eps \to 0,
\end{align*}
we can deduce 
\begin{align*}
\int_{X^{\eps}} \delta_{\eps} \abs{\mathcal{B}^{T} \nabla \psi^{e}}^{2} \dx \to \int_{\Gamma} \abs{\mathcal{B}^{T} \surf \psi}^{2} \dHaus \text{ as } \eps \to 0,
\end{align*}
and hence we obtain the norm convergence: As $\eps \to 0$,
\begin{equation}\label{Bsurfpsinormconvergence}
\begin{aligned}
& \; \int_{-\eta_{\eps,k}}^{\eta_{\eps,k}} \int_{\Gamma} \delta(z)  \abs{\mathcal{B}^{T}(p) \surfz \Psi_{\eps}(p,z)}^{2} \abs{1 + \eps z C_{H}(p) + C_{R}(p, \eps z)} \dHaus \dz  \\
= & \; \int_{X^{\eps}} \delta_{\eps} \abs{\mathcal{B}^{T} \nabla \psi^{e}}^{2} \dx \to \int_{\Gamma} \abs{\mathcal{B}^{T} \surf \psi}^{2} \dHaus \\
= & \; \int_{\R} \int_{\Gamma} \delta(z)   \abs{\mathcal{B}^{T}(p) (\surf \psi)(p)}^{2} \dHaus \dz.
\end{aligned}
\end{equation}
Then, arguing as in the proof of (\ref{strongconvergencePhi}), we use (\ref{Bsurfpsipointwise}), (\ref{Bsurfpsiuniformnormbdd}) and (\ref{Bsurfpsinormconvergence}) to deduce the required strong convergence (\ref{BsurfpsiStrong}).
\end{proof}

\begin{cor}\label{cor:deltaL2GammaH1}
Suppose that Assumptions \ref{assump:Gamma} and \ref{assump:delta} are satisfied.   Let $g^{Ec} \in L^{2}(\Omega)$ denote the extensions of the data $g \in L^{2}(\Gamma)$ as mentioned in Assumption \ref{assump:surfacedata} and Remark \ref{rem:extensionambiguity}.  Let $\varphi \in H^{1}(\Omega)$.  Then, as $\eps \to 0$,
\begin{align}\label{deltaL2GammaH1}
\int_{\Omega} \delta_{\eps} g^{Ec} \varphi \dx \to \int_{\Gamma} g \gamma_{0}(\varphi) \dHaus.
\end{align}
\end{cor}
\begin{proof}
Choose $0 < \eta < (\tilde{c})^{-1}$ and the diffeomorphism $\Theta^{\eta}$ as defined in (\ref{diffeo}).  Let  $X^{\eps} = \mathrm{Tub}^{\eps^{k} \eta}(\Gamma)$ for $0 < k < 1$.  By (\ref{delta:remainder:Xeps}),
\begin{align*}
\abs{\int_{\Omega \setminus X^{\eps}} \delta_{\eps} \varphi g^{Ec} \dx} \leq  \norm{\delta_{\eps}}_{L^{\infty}(\Omega \setminus X^{\eps})} \norm{g^{Ec}}_{L^{2}(\Omega)} \norm{\varphi}_{L^{2}(\Omega)}  \to 0 \text{ as } \eps \to 0.
\end{align*}
Hence, it is sufficient to restrict our attention to $X^{\eps}$.  Invoking the $(p,z)$ coordinate system, we note that by definition $g^{Ec}(x) = g^{e}(x) = g(p(x))$ for all $x \in X^{\eps}$.  Moreover,  (\ref{aeconvergencePhi}), (\ref{normbddPhi}), (\ref{normconvergencePhi}), and (\ref{strongconvergencePhi}) are still valid.  Hence, by (\ref{strongconvergencePhi}) and (\ref{1CHCRstrongconverge}), we have
\begin{align*}
& \; \int_{X^{\eps}} \delta_{\eps} g \varphi \dx = \int_{-\eta_{\eps,k}}^{\eta_{\eps,k}} \int_{\Gamma} \delta(z) g(p) \Phi_{\eps}(p,z) \abs{1 + \eps z C_{H}(p) + C_{R}(p, \eps z)} \dHaus \dz \\
\to & \; \int_{\R} \int_{\Gamma} \delta(z) g(p) \gamma_{0}(\varphi)(p) \dHaus \dz = \int_{\Gamma} g \gamma_{0}(\varphi) \dHaus \text{ as } \eps \to 0.
\end{align*}
\end{proof}

\section{Proof of the main results}\label{sec:Mainproof}

\subsection{Proof of Lemma \ref{lem:DirichletfromRobin}}\label{sec:proof:DirichletfromRobin}
The weak formulation of (RSIH) is: Find $w \in H^{1}(\Omega^{*})$ such that for all $\varphi \in H^{1}(\Omega^{*})$,
\begin{equation}\label{weakform:RSIH}
\begin{aligned}
& \; \int_{\Omega^{*}} \mathcal{A} \nabla w \cdot \nabla \varphi + a w \varphi \dx + \int_{\Gamma} \beta \gamma_{0}(w) \gamma_{0}(\varphi) \dHaus \\
 = & \; \int_{\Omega^{*}} f \varphi - \mathcal{A} \nabla \tilde{g} \cdot \nabla \varphi - a \tilde{g} \varphi \dx.
\end{aligned}
\end{equation}
By the Lax--Milgram theorem, for each $\beta > 0$, there exists a unique weak solution $w^{\beta} \in H^{1}(\Omega^{*})$ to (RSIH).  Moreover, $w^{\beta}$ satisfies the following estimate, with the constant $C_{\mathcal{A}, a}$ independent of $\beta$:
\begin{align*}
\min(\theta_{0}, \theta_{2}) \norm{w^{\beta}}_{H^{1}(\Omega^{*})}^{2} + 2 \beta \norm{\gamma_{0}(w^{\beta})}_{L^{2}(\Gamma)}^{2} \leq \frac{\norm{f}_{L^{2}(\Omega^{*})}^{2} + C_{\mathcal{A},a} \norm{\tilde{g}}_{H^{1}(\Omega^{*})}^{2}}{\min(\theta_{0}, \theta_{2})}.
\end{align*}
From this, we observe that $\{w^{\beta}\}_{\beta}$ is bounded uniformly in $H^{1}(\Omega^{*})$.  Hence there exists a subsequence $(\beta_{i})_{i \in \N} \to \infty$ such that
\begin{align*}
w^{\beta_{i}} \rightharpoonup w^{0} \in H^{1}(\Omega^{*}) \text{ as } \beta_{i} \to \infty.
\end{align*}
Moreover, from the a priori estimates on $\{w^{\beta}\}_{\beta}$, we see that
\begin{align*}
\norm{\gamma_{0}(w^{\beta})}_{L^{2}(\Gamma)} \leq \frac{C(\theta_{0}, \theta_{2}, f, \mathcal{A}, a, \tilde{g})}{\sqrt{\beta}} \to 0 \text{ as } \beta \to \infty.
\end{align*}
Compactness of the trace operator (Theorem \ref{thm:compacttrace}) implies that $\gamma_{0}(w^{\beta})$ converges strongly to $\gamma_{0}(w^{0})$, and subsequently the norms also converge.  Thus,
\begin{align*}
0 = \lim_{\beta_{i} \to \infty} \norm{\gamma_{0}(w^{\beta_{i}})}_{L^{2}(\Gamma)} = \norm{\gamma_{0}(w^{0})}_{L^{2}(\Gamma)},
\end{align*}
and so $\gamma_{0}(w^{0}) = 0$ a.e. on $\Gamma$.  Since $\Gamma$ is $C^{1}$, by the characterisation of zero trace Sobolev functions \cite[Theorem 2, p. 259]{book:Evans} we have $w^{0} \in H^{1}_{0}(\Omega^{*})$.  Testing with $\varphi \in H^{1}_{0}(\Omega^{*})$ in (\ref{weakform:RSIH}) and passing to the limit $\beta \to \infty$, we see that $w^{0}$ satisfies
\begin{align*}
\int_{\Omega^{*}} \mathcal{A} \nabla w^{0} \cdot \nabla \varphi + a w^{0} \varphi \dx = \int_{\Omega^{*}} f \varphi - \mathcal{A} \nabla \tilde{g} \cdot \nabla \varphi - a \tilde{g} \varphi \dx,
\end{align*}
for all $\varphi \in H^{1}_{0}(\Omega^{*})$.  Hence $w^{\beta}$ converges (along a subsequence) to the solution of (DSIH) weakly in $H^{1}(\Omega^{*})$ as $\beta \to \infty$.  Since (DSIH) is well-posed by the Lax--Milgram theorem, the aforementioned convergence applies to the whole sequence.

\subsection{Well-posedness for (CSI)}\label{sec:proof:CSI}
We consider the product Hilbert space and associated inner product
\begin{align*}
 \mathcal{X} & := H^{1}(\Omega^{(1)}) \times H^{1}(\Gamma), \\
\inner{(u_{1},v_{1})}{(u_{2},v_{2})}_{\mathcal{X}} & := \int_{\Omega^{(1)}} u_{1} u_{2} + \nabla u_{1} \cdot \nabla u_{2} \dx + \int_{\Gamma} v_{1} v_{2} + \surf v_{1} \cdot \surf v_{2} \dHaus. 
\end{align*}
Let
\begin{align*}
a_{CSI}((u,v),(\varphi, \psi)) & := a_{B}(u,\varphi) + a_{S}(v, \psi) + K l_{S}(v - \gamma_{0}(u), \psi - \gamma_{0}(\varphi)), \\
l_{CSI}((\varphi, \psi)) & := l_{B}(f,\varphi) + \beta l_{S}(g,\psi),
\end{align*}
where $a_{B}(\cdot, \cdot)$, $a_{S}(\cdot, \cdot)$, $l_{B}(\cdot, \cdot)$ and $l_{S}(\cdot, \cdot)$ are as defined in (\ref{bulkbilinear}) and (\ref{surfacebilinear}).  The weak formulation for (CSI) is: Find $(u,v) \in \mathcal{X}$ such that for all $(\varphi, \psi) \in \mathcal{X}$,
\begin{align*}
a_{CSI}((u,v), (\varphi, \psi)) = l_{CSI}((\varphi, \psi)).
\end{align*}

By the root mean square inequality, i.e.,
\begin{align}\label{rootmeansq}
a+b \leq \sqrt{2} \sqrt{a^{2} + b^{2}},
\end{align}
one can show that
\begin{equation}\label{l:CSI:bound}
\begin{aligned}
 \abs{l_{CSI}((\varphi, \psi))} & \leq \norm{f}_{L^{2}(\Omega^{*})} \norm{\varphi}_{L^{2}(\Omega^{*})} + \norm{g}_{L^{2}(\Gamma)} \norm{\psi}_{L^{2}(\Gamma)} \\
& \leq (\norm{f}_{L^{2}(\Omega^{*})} + \norm{g}_{L^{2}(\Gamma)})(\norm{\varphi}_{L^{2}(\Omega^{*})} + \norm{\psi}_{L^{2}(\Gamma)}) \\
& \leq \sqrt{2} (\norm{f}_{L^{2}(\Omega^{*})} + \norm{g}_{L^{2}(\Gamma)}) \norm{(\varphi, \psi)}_{\mathcal{X}},
\end{aligned}
\end{equation}
and
\begin{align*}
& \; \int_{\Gamma} \abs{(v - \gamma_{0}(u))(\psi - \gamma_{0}(\varphi))} \dHaus \\
 \leq & \; \norm{v - \gamma_{0}(u)}_{L^{2}(\Gamma)} \norm{\psi - \gamma_{0}(\varphi)}_{L^{2}(\Gamma)} \\
\leq & \; (\norm{v}_{L^{2}(\Gamma)} + C_{\mathrm{tr}} \norm{u}_{H^{1}(\Omega^{*})})(\norm{\psi}_{L^{2}(\Gamma)} + C_{\mathrm{tr}} \norm{\varphi}_{H^{1}(\Omega^{*})}) \\
\leq & \; 2 (1 + C_{\mathrm{tr}}^{2}) \norm{(u,v)}_{\mathcal{X}} \norm{(\varphi, \psi)}_{\mathcal{X}},
\end{align*}
and so
\begin{align*}
\abs{a_{CSI}((u,v), (\varphi, \psi))} \leq (C_{\mathcal{A}, \mathcal{B}, a, b} + 2K(1 + C_{\mathrm{tr}}^{2})) \norm{(u,v)}_{\mathcal{X}} \norm{(\varphi, \psi)}_{\mathcal{X}}.
\end{align*}
Moreover, 
\begin{align*}
a_{CSI}((u,v),(u,v)) & \geq \min(\theta_{0}, \theta_{2}) \norm{u}_{H^{1}(\Omega^{*})}^{2} + \min(\theta_{1}, \theta_{3}) \norm{v}_{H^{1}(\Gamma)}^{2} + K \norm{v - \gamma_{0}(u)}_{L^{2}(\Gamma)}^{2} \\
& \geq (\min_{i} \theta_{i}) \norm{(u,v)}_{\mathcal{X}}^{2}.
\end{align*}
By the Lax--Milgram theorem, there exists a unique weak solution $(u,v) \in \mathcal{X}$ to (CSI) such that
\begin{align*}
\norm{(u,v)}_{\mathcal{X}} \leq C(\theta_{i}) (\norm{f}_{L^{2}(\Omega^{*})} + \norm{g}_{L^{2}(\Gamma)}).
\end{align*}

\subsection{Well-posedness for (CDD)}\label{sec:proof:CDD}
We consider the product Hilbert space and associated inner product
\begin{align*}
 \mathcal{X}_{\eps} &:= \mathcal{V}_{\eps} \times H^{1}(\Omega, \delta_{\eps}), \\
\inner{(u_{1},v_{1})}{(u_{2},v_{2})}_{\mathcal{X}_{\eps}} & := \int_{\Omega} (\xi_{\eps} + \delta_{\eps})u_{1} u_{2} + \xi_{\eps} \nabla u_{1} \cdot \nabla u_{2} + \delta_{\eps} v_{1} v_{2} + \delta_{\eps} \nabla v_{1} \cdot \nabla v_{2} \dx.
\end{align*}
Let
\begin{align}
a_{CDD}((u,v),(\varphi, \psi)) & := a_{B}^{\eps}(u,\varphi) + a_{S}^{\eps}(v, \psi) + K l_{S}^{\eps}(v - u, \psi - \varphi), \label{defn:aCDD} \\
l_{CDD}((\varphi, \psi)) & := l_{B}^{\eps}(f^{Ea},\varphi) + \beta l_{S}^{\eps}(g^{Ec},\psi), \label{defn:lCDD}
\end{align}
where $a_{B}^{\eps}(\cdot, \cdot)$, $a_{S}^{\eps}(\cdot, \cdot)$, $l_{B}^{\eps}(\cdot, \cdot)$ and $l_{S}^{\eps}(\cdot, \cdot)$ are as defined in (\ref{bulkbilinear:eps}) and (\ref{surfacebilinear:eps}).  The weak formulation for (CDD) is: Find $(u^{\eps}, v^{\eps}) \in \mathcal{X}_{\eps}$ such that for all $(\varphi, \psi) \in \mathcal{X}_{\eps}$,
\begin{align*}
a_{CDD}((u^{\eps},v^{\eps}),(\varphi, \psi)) = l_{CDD}((\varphi, \psi)).
\end{align*}
A similar calculation to (\ref{l:CSI:bound}) involving the root mean square inequality (\ref{rootmeansq}) applied to $l_{CDD}$ yields
\begin{align*}
\abs{l_{CDD}((\varphi, \psi))} \leq \sqrt{2} \left ( \norm{f^{Ea}}_{0,\xi_{\eps}} + \norm{g^{Ec}}_{0,\delta_{\eps}} \right ) \norm{(\varphi, \psi)}_{\mathcal{X}_{\eps}}.
\end{align*}
Similarly, we have
\begin{align*}
\abs{a_{CDD}((u,v),(\varphi, \psi))} \leq (C(\mathcal{A}^{Ea}, a^{Ea}, \mathcal{B}^{Ec}, b^{Ec}) + K) \norm{(u^{\eps},v^{\eps})}_{\mathcal{X}_{\eps}} \norm{(\varphi, \psi)}_{\mathcal{X}_{\eps}}.
\end{align*}
By Young's inequality with constant $\mu \in (1,2)$, we have
\begin{align*}
\int_{\Omega} \delta_{\eps} \abs{v^{\eps} - u^{\eps}}^{2} \dx & \geq \int_{\Omega} \delta_{\eps}(\abs{v^{\eps}}^{2} - 2 \abs{v^{\eps}}\abs{u^{\eps}} + \abs{u^{\eps}}^{2}) \dx \\
& \geq (1-\mu) \norm{v^{\eps}}_{0,\delta_{\eps}}^{2} + (1-\mu^{-1})\norm{u^{\eps}}_{0,\delta_{\eps}}^{2}.
\end{align*}
Then, by the assumption specifically for (CDD), we have $\theta_{3} \geq K$, and so
\begin{equation}\label{CDDcoercivity}
\begin{aligned}
a_{CDD}((u^{\eps},v^{\eps}),(u^{\eps},v^{\eps})) & \geq C(\theta_{0},\theta_{2}) \norm{u^{\eps}}_{1,\xi_{\eps}}^{2} + \theta_{1} \norm{\nabla v^{\eps}}_{0,\delta_{\eps}}^{2} \\
& + (\theta_{3} + K (1- \mu)) \norm{v^{\eps}}_{0,\delta_{\eps}}^{2}  + K(1-\mu^{-1}) \norm{u^{\eps}}_{0,\delta_{\eps}}^{2} \\
& \geq C(\theta_{i}, K, \mu) \norm{(u^{\eps},v^{\eps})}_{\mathcal{X}_{\eps}}^{2}.
\end{aligned}
\end{equation}

Hence, by the Lax--Milgram theorem, for every $\eps > 0$ there exists a unique pair of functions $(u^{\eps},v^{\eps}) \in \mathcal{X}_{\eps}$ that is a weak solution to (CDD) and satisfies
\begin{align*}
\norm{(u^{\eps},v^{\eps})}_{\mathcal{X}_{\eps}} \leq C \left ( \norm{f^{Ea}}_{0,\xi_{\eps}} + \norm{g^{Ec}}_{0,\delta_{\eps}} \right ),
\end{align*}
where the constant $C$ is independent of $\eps$.  Next, by the assumption that $\xi_{\eps} \leq 1$ and Lemma \ref{lem:constextH1delta}, there exists a constant $C$, independent of $\eps$ such that
\begin{align*}
\norm{(u^{\eps},v^{\eps})}_{\mathcal{X}_{\eps}} \leq C \left ( \norm{f^{Ea}}_{0,\xi_{\eps}} + \norm{g^{Ec}}_{0,\delta_{\eps}} \right ) \leq C \left ( \norm{f^{Ea}}_{L^{2}(\Omega)} + \norm{g}_{L^{2}(\Gamma)} \right ),
\end{align*}
which is (\ref{apriori:coupledbulk}).

By a similar argument involving the Lax--Milgram theorem, one can show the well-posedness of (SDD) in $H^{1}(\Omega, \delta_{\eps})$, (RDD) in $\mathcal{V}_{\eps}$, (DDDH) in $\mathcal{W}_{\eps}$ and (NDDH) in $H^{1}(\Omega, \xi_{\eps})$.  The independence of $\eps$ for the constants in (\ref{apriori:surface}), (\ref{apriori:Robin}), (\ref{apriori:Dirichlet}) and (\ref{apriori:Neumann}) follow again from Lemma \ref{lem:constextH1delta} and the property that $\xi_{\eps} \leq 1$.

\subsection{Weak convergence}
Let $\varphi \in H^{1}(\Omega^{*})$ and $\psi \in H^{1}(\Gamma)$ be arbitrary.  Let $\varphi^{Er} \in H^{1}(\Omega)$ denote the extension of $\varphi$ to $\Omega$ by the extension theorem \cite[Theorem 1, p. 254]{book:Evans}, and let $\psi^{Ec} \in H^{1}(\Omega)$ denote the extension of $\psi$ to $\Omega$ as outlined in Corollary \ref{cor:constextH1}.

By $\xi_{\eps} \leq 1$, Lemma \ref{lem:H1impliesL2delta}, and Lemma \ref{lem:constextH1delta}, we see that $\varphi^{Er} \in \mathcal{V}_{\eps}$ and $\psi^{Ec} \in H^{1}(\Omega, \delta_{\eps})$ for all $\eps \in (0,1]$.  Moreover, there exists a constant $C > 0$, independent of $\eps$, such that, 
\begin{align*}
\norm{\varphi^{Er}}_{\mathcal{V}_{\eps}} \leq C \norm{\varphi}_{H^{1}(\Omega^{*})}, \quad \norm{\psi^{Ec}}_{1,\delta_{\eps}} \leq C \norm{\psi}_{H^{1}(\Gamma)}.
\end{align*}
Thus, we may test with $\varphi^{Er}$ and $\psi^{Ec}$ in the weak formulation for (CDD).

For $\eps \in (0,1]$, let $(u^{\eps}, v^{\eps}) \in \mathcal{X}_{\eps}$ denote the unique weak solution to (CDD).  Then, they satisfy
\begin{equation}\label{CDDweakconv}
\begin{aligned}
& \; \int_{\Omega} \xi_{\eps} \mathcal{A}^{Ea} \nabla u^{\eps} \cdot \nabla \varphi^{Er} + \xi_{\eps} a^{Ea} u^{\eps} \varphi^{Er} + \delta_{\eps} \mathcal{B}^{Ec} \nabla v^{\eps} \cdot \nabla \psi^{Ec} + \delta_{\eps} b^{Ec} v^{\eps} \psi^{Ec} \dx \\
+ & \; \int_{\Omega} K \delta_{\eps} (v^{\eps} - u^{\eps})(\psi^{Ec} - \varphi^{Er}) \dx - \int_{\Omega} \xi_{\eps} f^{Ea} \varphi^{Er} + \delta_{\eps} \beta g^{Ec} \psi^{Ec} \dx  = 0.
\end{aligned}
\end{equation}
We analyse the bulk and surface terms separately.  From (\ref{apriori:coupledbulk}), we have
\begin{align}\label{apriori:coupledbulk:bulkterms}
\norm{u^{\eps}}_{1,\xi_{\eps}}^{2} + \norm{u^{\eps}}_{0,\delta_{\eps}}^{2} \leq \norm{(u^{\eps}, v^{\eps})}_{\mathcal{X}_{\eps}}^{2} \leq C(\norm{f^{Ea}}_{L^{2}(\Omega)}^{2} + \norm{g}_{L^{2}(\Gamma)}^{2}) =: C_{f,g},
\end{align}
where $C_{f,g}$ is independent of $\eps$.  Then, by Lemma \ref{lem:Vepscompactness} and Rellich--Kondrachov compactness theorem, there exists $\tilde{u} \in H^{1}(\Omega^{*})$ such that, along a subsequence
\begin{align}
u^{\eps} \mid_{\Omega^{*}} & \rightharpoonup \tilde{u}  \text{ in } H^{1}(\Omega^{*}) \text{ as } \eps \to 0, \label{uepstildeuH1} \\
u^{\eps} \mid_{\Omega^{*}} & \rightarrow \tilde{u}  \text{ in } L^{2}(\Omega^{*}) \text{ as } \eps \to 0, \label{uepstildeuL2} \\
\int_{\Omega} \delta_{\eps} u^{\eps} (\psi^{Ec} - \varphi^{Er}) \dx & \to \int_{\Gamma} \gamma_{0}(\tilde{u}) (\psi - \gamma_{0}(\varphi) ) \dHaus \text{ as } \eps \to 0 \label{uepstildeutrace}.
\end{align}
By H\"{o}lder's inequality, $f^{Ea} \varphi^{Er} \in L^{1}(\Omega)$, and so by  Lemma \ref{lem:XiDCT}, we see that
\begin{align}\label{bulkRHS}
\int_{\Omega} \xi_{\eps} f^{Ea} \varphi^{Er} \dx = \int_{\Omega^{*}} \xi_{\eps} f \varphi \dx + \int_{\Omega \setminus \Omega^{*}} \xi_{\eps} f^{Ea} \varphi^{Er} \dx \to \int_{\Omega^{*} } f \varphi \dx \text{ as } \eps \to 0.
\end{align}
By Cauchy--Schwarz inequality, Lemma \ref{lem:XiDCT}, (\ref{apriori:coupledbulk:bulkterms}), and the fact that $(1- \xi_{\eps}) \leq \frac{1}{2} \leq \xi_{\eps}$ in $\Omega^{*}$, we have
\begin{equation}\label{bulklowerorder:intermediatecal}
\begin{aligned}
& \; \int_{\Omega^{*}} (1-\xi_{\eps}) \abs{u^{\eps}} \abs{\varphi} \dx  \leq \left ( \int_{\Omega^{*}} \tfrac{1}{2} \abs{u^{\eps}}^{2} \dx \right )^{\frac{1}{2}} \left ( \int_{\Omega^{*}} (1-\xi_{\eps}) \abs{\varphi}^{2} \dx \right )^{\frac{1}{2}} \\
\leq & \; \norm{u^{\eps}}_{0,\xi_{\eps}} \left ( \int_{\Omega^{*}} (1-\xi_{\eps}) \abs{\varphi}^{2} \dx \right )^{\frac{1}{2}}  \leq C_{f,g} \left ( \int_{\Omega^{*}} (1-\xi_{\eps}) \abs{\varphi}^{2} \dx \right )^{\frac{1}{2}} \to 0 \text{ as } \eps \to 0.
\end{aligned}
\end{equation}
Thus, together with (\ref{uepstildeuL2}), we have
\begin{align*}
& \; \abs{\int_{\Omega} \xi_{\eps} a^{Ea} u^{\eps} \varphi^{Er} \dx - \int_{\Omega^{*}} a \tilde{u} \varphi \dx}\\
 \leq & \;  \int_{\Omega^{*}} \abs{a} \abs{u^{\eps} - \tilde{u}} \abs{\varphi} \dx + \int_{\Omega^{*}} \abs{a} (1 - \xi_{\eps}) \abs{u^{\eps}} \abs{\varphi} \dx + \int_{\Omega \setminus \Omega^{*}} \xi_{\eps} \abs{a^{Ea}} \abs{u^{\eps}} \abs{\varphi^{Er}} \dx \\
 \leq & \; C_{a} \left ( \norm{u^{\eps} - \tilde{u}}_{L^{2}(\Omega^{*})} \norm{\varphi}_{L^{2}(\Omega^{*})} + \int_{\Omega^{*}} (1 - \xi_{\eps}) \abs{u^{\eps}} \abs{\varphi} \dx + C_{f,g} \norm{\varphi^{Er}}_{L^{2}(\Omega \setminus \Omega^{*}, \xi_{\eps})} \right )  \to 0,
\end{align*}
as $\eps \to 0$.  This implies that 
\begin{align}\label{bulkloworder}
\int_{\Omega} \xi_{\eps} a^{Ea} u^{\eps} \varphi^{Er} \dx \to \int_{\Omega^{*}} a \tilde{u} \varphi \dx.
\end{align}
A similar calculation to (\ref{bulklowerorder:intermediatecal}) gives
\begin{align*}
\int_{\Omega^{*}} (1-\xi_{\eps}) \abs{\nabla u^{\eps}} \abs{\nabla \varphi} \dx \leq C_{f,g} \left ( \int_{\Omega^{*}} (1-\xi_{\eps}) \abs{\nabla \varphi}^{2} \dx \right )^{\frac{1}{2}} \to 0 \text{ as } \eps \to 0.
\end{align*}
Meanwhile, by Cauchy--Schwarz inequality, (\ref{apriori:coupledbulk:bulkterms}), and Lemma \ref{lem:XiDCT}, 
\begin{align*}
\abs{\int_{\Omega \setminus \Omega^{*}} \xi_{\eps} \nabla u^{\eps} \cdot (\mathcal{A}^{Ea})^{T} \nabla \varphi^{Er} \dx} & \leq \norm{\nabla u^{\eps}}_{L^{2}(\Omega \setminus \Omega^{*}, \xi_{\eps})} C(\mathcal{A}^{Ea}) \norm{\nabla \varphi^{Er}}_{L^{2}(\Omega \setminus \Omega^{*}, \xi_{\eps})} \\
& \leq C_{f,g} C(\mathcal{A}^{Ea}) \norm{\nabla \varphi^{Er}}_{L^{2}(\Omega \setminus \Omega^{*}, \xi_{\eps})} \to 0 \text{ as } \eps \to 0.
\end{align*}
Then, together with the weak convergence of $\nabla u^{\eps} \mid_{\Omega^{*}}$ to $\nabla \tilde{u}$ in $L^{2}(\Omega^{*})$, we obtain
\begin{equation}\label{bulkgradient}
\begin{aligned}
& \abs{\int_{\Omega} \xi_{\eps} \mathcal{A}^{Ea} \nabla u^{\eps} \cdot \nabla \varphi^{Er} \dx - \int_{\Omega^{*}} \mathcal{A} \nabla \tilde{u} \cdot \nabla \varphi \dx } \\
\leq & \; \abs{\int_{\Omega^{*}} \nabla (u^{\eps} - \tilde{u}) \cdot \mathcal{A}^{T} \nabla \varphi \dx} + \norm{\mathcal{A}}_{L^{\infty}(\Omega^{*})} \int_{\Omega^{*}} (1-\xi_{\eps}) \abs{\nabla u^{\eps}} \abs{\nabla \varphi} \dx \\
+ & \; \abs{\int_{\Omega \setminus \Omega^{*}} \xi_{\eps} \nabla u^{\eps} \cdot (\mathcal{A}^{Ea})^{T} \nabla \varphi^{Er} \dx} \to 0 \text{ as } \eps \to 0.
\end{aligned}
\end{equation}

Next, by (\ref{convergence:constext}), we have
\begin{equation}\label{surfaceRHS}
\begin{aligned}
\int_{\Omega} \delta_{\eps} g^{Ec} \psi^{Ec} \dx & = \frac{1}{2} \int_{\Omega} \delta_{\eps} \left ( \abs{g^{Ec} + \psi^{Ec}}^{2} - \abs{g^{Ec}}^{2} - \abs{\psi^{Ec}}^{2} \right ) \dx \\
& \to \frac{1}{2} \int_{\Gamma} (\abs{g + \psi}^{2} - \abs{g}^{2} - \abs{\psi}^{2} ) \dHaus = \int_{\Gamma} g \psi \dHaus \text{ as } \eps \to 0.
\end{aligned}
\end{equation}
By (\ref{apriori:coupledbulk}), we have
\begin{align*}
\norm{v^{\eps}}_{1,\delta_{\eps}}^{2} \leq \norm{(u^{\eps},v^{\eps})}_{\mathcal{X}_{\eps}}^{2} \leq C_{f,g},
\end{align*}
and so the conditions of Lemma \ref{lem:H1deltabddconv} are satisfied.  Thus, there exists a function $\tilde{v} \in H^{1}(\Gamma)$ such that
\begin{equation}\label{surfacelowerorderandgradient}
\begin{aligned}
\int_{\Omega} \delta_{\eps} b^{Ec} v^{\eps} \psi^{Ec} \dx & \to \int_{\Gamma} b\tilde{v} \psi \dHaus, \\
\int_{\Omega} \delta_{\eps} \mathcal{B}^{Ec} \nabla v^{\eps} \cdot \nabla \psi^{Ec} \dx & \to \int_{\Gamma} \mathcal{B} \surf \tilde{v} \cdot \surf \psi \dHaus,
\end{aligned}
\end{equation}
as $\eps \to 0$.  Moreover, using the fact that $\psi^{Ec} - \varphi^{Er} \in H^{1}(\Omega)$, we obtain from $(\ref{H1deltabdd:convstatement})_{1}$ with $b^{Ec} = b = 1$, 
\begin{align}\label{surfacecrossterm}
\int_{\Omega} \delta_{\eps} v^{\eps}(\psi^{Ec} - \varphi^{Er}) \dx \to \int_{\Gamma} \tilde{v} (\psi - \gamma_{0}(\varphi) ) \dHaus \text{ as } \eps \to 0.
\end{align}

In summary, from (\ref{uepstildeutrace}), (\ref{bulkRHS}), (\ref{bulkloworder}), (\ref{bulkgradient}), (\ref{surfaceRHS}), (\ref{surfacelowerorderandgradient}), and (\ref{surfacecrossterm}), we see that passing to the limit $\eps \to 0$ in (\ref{CDDweakconv}) leads to
\begin{align*}
& \; \int_{\Omega^{*}} \mathcal{A} \nabla \tilde{u} \cdot \nabla \varphi + a \tilde{u} \varphi \dx + \int_{\Gamma} \mathcal{B} \surf \tilde{v} \cdot \surf \psi + b \tilde{v} \psi \dHaus \\
+ & \; \int_{\Gamma} K (\tilde{v} - \gamma_{0}(\tilde{u}))(\psi - \gamma_{0}(\varphi)) \dHaus - \int_{\Omega^{*}} f \varphi \dx - \int_{\Gamma} \beta g \psi \dHaus = 0.
\end{align*}
In particular $(\tilde{u}, \tilde{v})$ is a weak solution of (CSI).  But by the well-posedness of (CSI), we must have that $\tilde{u} = u$ and $\tilde{v} = v$, and that the whole sequence converge as $\eps \to 0$.

We remark on how to prove Theorem \ref{thm:weakconv} for the other problems:
\begin{itemize}
\item For (SDD), it follows directly from (\ref{surfaceRHS}) and (\ref{surfacelowerorderandgradient}).
\item For (RDD), it follows from (\ref{bulkRHS}), (\ref{bulkloworder}), (\ref{bulkgradient}), (\ref{uepstildeutrace}) with $\psi^{Ec} - \varphi^{Er}$ replaced with $\varphi^{Er}$, and using Corollary \ref{cor:deltaL2GammaH1} gives
\begin{align*}
\int_{\Omega} \delta_{\eps} g^{Ec} \varphi^{Er} \dx \to \int_{\Gamma} g \gamma_{0}(\varphi) \dHaus \text{ as } \eps \to 0.
\end{align*}
\item For (NDDH), it follows from (\ref{bulkloworder}), (\ref{bulkgradient}) and (\ref{bulkRHS}) with $f^{Ea}$ replaced with $f^{Ea} + \nabla \cdot (\mathcal{A}^{Er} \nabla h^{Er}) - a^{Ea} h^{Er}$.
\item For (DDDH), we consider an arbitrary $\phi \in H^{1}_{0}(\Omega^{*})$ and extend it to $\phi^{Er} \in H^{1}_{\Gamma,0}(\Omega)$ by the reflection method.  Then, by Lemma \ref{lem:H1zeroimpliesL2deltaeps-1}, $\phi^{Er} \in \mathcal{W}_{\eps}$ is a valid test function.  From (\ref{apriori:Dirichlet}) and Corollary \ref{cor:CompactnessWeps}, we obtain, along a subsequence,
\begin{align*}
w^{\eps}_{D} \mid_{\Omega^{*}} & \rightharpoonup \tilde{w}_{D} \in H^{1}(\Omega^{*}) \text{ as } \eps \to 0,
\end{align*}
for some $\tilde{w}_{D} \in H^{1}_{0}(\Omega^{*})$.  Moreover, by Cauchy--Schwarz inequality, (\ref{apriori:Dirichlet}), and (\ref{H1zeroL2deltaeps-1conv}) applied to $\phi^{Er}$, we see that
\begin{align*}
\abs{ \int_{\Omega} \frac{1}{\eps} \delta_{\eps} w^{\eps}_{D} \phi^{Er} \dx } & \leq \norm{w^{\eps}_{D}}_{0,\tfrac{1}{\eps} \delta_{\eps}} \norm{\phi^{Er}}_{0,\frac{1}{\eps} \delta_{\eps}} \\
 & \leq C(\norm{f^{Ea}}_{L^{2}(\Omega)}, \norm{\tilde{g}^{Er}}_{H^{1}(\Omega)}) \norm{\phi^{Er}}_{0,\frac{1}{\eps} \delta_{\eps}} \to 0 \text{ as } \eps \to 0.
\end{align*}
Then, together with (\ref{bulkRHS}), (\ref{bulkloworder}), (\ref{bulkgradient}), and the above, we see that $\tilde{w}_{D}$ satisfies the weak formulation for $(\mathrm{DSIH})$.
\end{itemize}

\subsection{Strong convergence}
We can choose $\varphi = u \in H^{1}(\Omega)$ and $\psi = v \in H^{1}(\Gamma)$, where $(u,v) \in H^{1}(\Omega^{*}) \times H^{1}(\Gamma)$ denote the unique weak solution to (CSI).  Then, the extension $u^{Er} \in H^{1}(\Omega)$ by the extension theorem \cite[Theorem 1, p. 254]{book:Evans}, and $\psi^{Ec} \in H^{1}(\Omega)$ as in Corollary \ref{cor:constextH1} are admissible test functions in the weak formulation of (CDD).  Due to the coercivity of the bilinear form (\ref{defn:aCDD}) (see (\ref{CDDcoercivity})) and that $\xi_{\eps} \geq \frac{1}{2}$ in $\Omega^{*}$, we obtain
\begin{equation}\label{strongconv:computation}
\begin{aligned}
& \; l_{CDD}((u^{\eps} - u^{Er}, v^{\eps} - v^{Ec})) - a_{CDD}((u^{Er}, v^{Ec}), (u^{\eps} - u^{Er}, v^{\eps} - v^{Ec})) \\
= & \; a_{CDD} ((u^{\eps} - u^{Er}, v^{\eps} - v^{Ec}), (u^{\eps} - u^{Er}, v^{\eps} - v^{Ec})) \\
\geq & \; C(\theta_{i},K) \left ( \norm{u^{\eps} - u^{Er}}_{1, \xi_{\eps}}^{2} + \norm{u^{\eps} - u^{Er}}_{0,\delta_{\eps}}^{2} + \norm{v^{\eps} - v^{Ec}}_{1,\delta_{\eps}}^{2} \right ) \\
\geq & \; C(\theta_{i},K) \left ( \norm{u^{\eps} \mid_{\Omega^{*}} - u}_{H^{1}(\Omega^{*})}^{2} + \norm{u^{\eps} - u^{Er}}_{0,\delta_{\eps}}^{2} + \norm{v^{\eps} - v^{Ec}}_{1,\delta_{\eps}}^{2} \right ).
\end{aligned}
\end{equation}
We claim that the left hand side converges to zero as $\eps \to 0$.  Indeed, by (\ref{bulkRHS}) with $\varphi^{Er} = u^{Er}$, (\ref{bulkloworder}) with $a^{Ea} = a = 1$, and $\varphi^{Er}$ replaced with $f^{Ea}$, we have
\begin{align}\label{strongconv:f}
\int_{\Omega} \xi_{\eps} f^{Ea} (u^{\eps} - u^{Er} ) \dx \to \int_{\Omega^{*}} f (u - u) \dx = 0 \text{ as } \eps \to 0.
\end{align}
Similarly, by (\ref{surfaceRHS}) with $\psi^{Ec} = v^{Ec}$ and (\ref{H1deltabdd:vepsg}), we have
\begin{align}\label{strongconv:g}
\int_{\Omega} \delta_{\eps} g^{Ec} (v^{\eps} - v^{Ec}) \dx \to \int_{\Gamma} g(v - v) \dHaus = 0 \text{ as } \eps \to 0.
\end{align}
Thus, 
\begin{align*}
l_{CDD}((u^{\eps} - u^{Er}, v^{\eps} - v^{Er})) = \int_{\Omega} \xi_{\eps} f^{Ea}(u^{\eps} - u^{Er}) + \delta_{\eps} \beta g^{Ec}(v^{\eps} - v^{Ec}) \dx \to 0 \text{ as } \eps \to 0.
\end{align*}
Meanwhile, using (\ref{bulkRHS}) with $f^{Ea} \varphi^{Er}$ replaced with $a^{Ea} \abs{u^{Er}}^{2} $ and with $\mathcal{A}^{Ea} \nabla u^{Er} \cdot \nabla u^{Er}$, (\ref{bulkloworder}) with $\varphi^{Er} = u^{Er}$ and (\ref{bulkgradient}) with $\varphi^{Er} = u^{Er}$, we have
\begin{align}\label{strongconv:bulk}
\int_{\Omega} \xi_{\eps} (\mathcal{A}^{Ea} \nabla u^{Er} \cdot \nabla (u^{\eps} - u^{Er}) + a^{Ea} u^{Er} (u^{\eps} - u^{Er})) \dx \to 0 \text{ as } \eps \to 0.
\end{align}
Next, by (\ref{convergence:constextbB}) with $f^{Ec} = v^{Ec}$, (\ref{H1deltabdd:convstatement}) with $\varphi = \psi^{Ec} = v^{Ec}$, we have
\begin{align}\label{strongconvsurface}
\int_{\Omega} \delta_{\eps} (\mathcal{B}^{Ec} \nabla v^{Ec} \cdot \nabla (v^{\eps} - v^{Ec}) + b^{Ec} v^{Ec} (v^{\eps} - v^{Ec})) \dx \to 0 \text{ as } \eps \to 0.
\end{align}
Finally, by (\ref{convergetodirac}) applied to $f =  u^{Er}(2 v^{Ec} - u^{Er}) \in W^{1,1}(\Omega)$, (\ref{convergencetotrace}) with $\varphi = v^{Ec} - u^{Er} \in H^{1}(\Omega)$, $(\ref{convergence:constextbB})_{1}$ with $b^{Ec} = 1$ and $f^{Ec} = v^{Ec}$, (\ref{H1deltabdd:vepsg}) with $g^{Ec} = v^{Ec}$, and $(\ref{H1deltabdd:convstatement})_{1}$ with $b^{Ec} = 1$ and $\varphi = u^{Er}$, we obtain
\begin{align*}
& \; \int_{\Omega} \delta_{\eps} (v^{Ec} - u^{Er})(v^{\eps} - u^{\eps} - v^{Ec} + u^{Er}) \dx \\
= & \; \int_{\Omega}  \delta_{\eps} \left (u^{Er} (2v^{Ec} - u^{Er}) - u^{\eps}(v^{Ec} - u^{Er}) - \abs{v^{Ec}}^{2} + v^{\eps} v^{Ec} - v^{\eps} u^{Er} \right ) \\
\to & \; \int_{\Gamma} \gamma_{0}(u)(2v - \gamma_{0}(u)) - \gamma_{0}(u)(v - \gamma_{0}(u)) - \abs{v}^{2} + \abs{v}^{2} - v \gamma_{0}(u) \dHaus = 0 \text{ as } \eps \to 0.
\end{align*}

Thus, 
\begin{align*}
& \; a_{CDD}((u^{Er}, v^{Ec}), (u^{\eps} - u^{Er}, v^{\eps} - v^{Ec})) \\
= & \; \int_{\Omega} \xi_{\eps}  (\mathcal{A}^{Ea} \nabla u^{Er} \cdot \nabla (u^{\eps} - u^{Er}) + a^{Ea} u^{Er} (u^{\eps} - u^{Er})) \dx \\
+ & \; \int_{\Omega} \delta_{\eps} (\mathcal{B}^{Ec} \nabla v^{Ec} \cdot \nabla (v^{\eps} - v^{Ec}) + b^{Ec} v^{Ec} (v^{\eps} - v^{Ec})) \dx \\
+ & \; K \int_{\Omega} \delta_{\eps} (v^{Ec} - u^{Er})(v^{\eps} - u^{\eps} - v^{Ec} + u^{Er}) \dx \to 0 \text{ as } \eps \to 0.
\end{align*}
From (\ref{strongconv:computation}), this implies that
\begin{align}\label{bulkandsurfacenormconvergence}
 \norm{u^{\eps} \mid_{\Omega^{*}} - u}_{H^{1}(\Omega^{*})}^{2} + \norm{u^{\eps} - u^{Er}}_{0,\delta_{\eps}}^{2} + \norm{v^{\eps} - v^{Ec}}_{1,\delta_{\eps}}^{2}  \to 0 \text{ as } \eps \to 0,
\end{align}
which is the first assertion of Theorem \ref{thm:strongconv} for (CDD).

Furthermore, by the triangle inequality, (\ref{bulkandsurfacenormconvergence}), and (\ref{convergence:constext}) with $f^{Ec} = v^{Ec}$, we obtain
\begin{align*}
\abs{ \norm{v^{\eps}}_{1,\delta_{\eps}} - \norm{v}_{H^{1}(\Gamma)} } & \leq \abs{ \norm{v^{\eps}}_{1,\delta_{\eps}} - \norm{v^{Ec}}_{1,\delta_{\eps}}} + \abs{\norm{v^{Ec}}_{1,\delta_{\eps}} - \norm{v}_{H^{1}(\Gamma)}} \\
& \leq \norm{v^{\eps} - v^{Ec}}_{1,\delta_{\eps}} + \abs{\norm{v^{Ec}}_{1,\delta_{\eps}} - \norm{v}_{H^{1}(\Gamma)}}  \to 0 \text{ as } \eps \to 0.
\end{align*}
Hence, we obtain the normal convergence
\begin{align*}
\norm{v^{\eps}}_{1,\delta_{\eps}} \to \norm{v}_{H^{1}(\Gamma)} \text{ as } \eps \to 0.
\end{align*}
Similarly, by (\ref{convergetodirac}) with $f = \abs{u^{Er}}^{2}$, we have $\norm{u^{Er}}_{0,\delta_{\eps}} \to \norm{\gamma_{0}(u)}_{L^{2}(\Gamma)}$ as $\eps \to 0$.  Together with the triangle inequality and (\ref{bulkandsurfacenormconvergence}), we have
\begin{align*}
\norm{u^{\eps}}_{0,\delta_{\eps}} \to \norm{\gamma_{0}(u)}_{L^{2}(\Gamma)} \text{ as } \eps \to 0.
\end{align*}

We remark on how to prove Theorem \ref{thm:strongconv} for the other problems:  Recall that $v_{S}^{\eps} \in H^{1}(\Omega, \delta_{\eps})$, $u_{R}^{\eps} \in \mathcal{V}_{\eps}$, $w_{N}^{\eps} \in H^{1}(\Omega, \xi_{\eps})$ and $w_{D}^{\eps} \in \mathcal{W}_{\eps}$ denote the unique weak solutions to (SDD), (RDD), (NDDH) and (DDDH), respectively, and $v_{S} \in H^{1}(\Gamma)$, $u_{R}, w_{N} \in H^{1}(\Omega^{*})$ and $w_{D} \in H^{1}_{0}(\Omega^{*})$ denote the corresponding unique weak solutions to (SSI), (RSI), (NSIH) and (DSIH).  We consider the extensions $v_{S}^{Ec}$ as constructed in Corollary \ref{cor:constextH1}, $u_{R}^{Er}, w_{N}^{Er}$ by the method of reflection and $w_{D}^{E0} \in H^{1}_{\Gamma,0}(\Omega)$ obtained from extending $w_{D}$ by zero outside $\Omega^{*}$.  Then, by Lemma \ref{lem:constextH1delta}, Lemma \ref{lem:H1impliesL2delta}, Lemma \ref{lem:H1zeroimpliesL2deltaeps-1}, and that $\xi_{\eps} \leq 1$ in $\Omega$, we have that $v_{S}^{Ec} \in H^{1}(\Omega, \delta_{\eps})$, $u_{R}^{Er}, w_{N}^{Er} \in \mathcal{V}_{\eps}$ and $w_{D}^{E0} \in \mathcal{W}_{\eps}$ are valid test functions for the weak formulation of (SDD), (RDD), (NDDH) and (DDDH).
\begin{itemize}
\item For (SDD), it follows from the above analysis by setting $v^{\eps} = v_{S}^{\eps}$, $v^{Ec} = v_{S}^{Ec}$, $\beta = 1$, $K = 0$, $\mathcal{A}^{Ea} = \bm{0}^{n \times n}$, $a^{Ea} = 0$ and $f^{Ea} = 0$.
\item For (RDD), it follows from setting $u^{\eps} = u_{R}^{\eps}$, $u^{Er} = u_{R}^{Er}$, $\mathcal{B}^{Ec} = \bm{0}^{n \times n}$, $b^{Ec} = 0$, $v^{\eps} = v = 0$, $K = \beta$ and using (\ref{strongconv:f}) and (\ref{deltaL2GammaH1}) to see that
\begin{align*}
l_{RDD}(u_{R}^{\eps} - u_{R}^{Er}) := \int_{\Omega} \xi_{\eps} f^{Ea}(u_{R}^{\eps} - u_{R}^{Er}) + \beta \delta_{\eps} g^{Ec}(u_{R}^{\eps} - u_{R}^{Er}) \dx \to 0 \text{ as } \eps \to 0.
\end{align*}
\item For (NDDH), it follows from setting $u^{\eps} = w_{N}^{\eps}$, $u^{Er} = w_{N}^{Er}$, $\mathcal{B}^{Ec} = \bm{0}^{n \times n}$, $b^{Ec} = 0$, $K = 0$, $g^{Ec} = 0$ and replacing $f^{Ea}$ with $f^{Ea} + \nabla \cdot (\mathcal{A}^{Er} \nabla h^{Er}) - a^{Ea} h^{Er}$.
\item For (DDDH), it follows from setting $u^{\eps} = w_{D}^{\eps}$, $u^{Er} = w_{D}^{E0}$, $\mathcal{B}^{Ec} = \bm{0}^{n \times n}$, $b^{Ec} = 0$, $K = \frac{1}{\eps}$, $g^{Ec} = 0$, and we use (\ref{strongconv:f}) and (\ref{strongconv:bulk}) to deduce
\begin{align*}
& \; l_{DDDH}(w_{D}^{\eps} - w_{D}^{Er}) \\
:= & \; \int_{\Omega} \xi_{\eps} (f^{Ea} (w_{D}^{\eps} - w_{D}^{Er}) - \mathcal{A}^{Ea} \nabla \tilde{g}^{Er} \cdot \nabla (w_{D}^{\eps} - w_{D}^{Er}))\dx\\
 - & \; \int_{\Omega} \xi_{\eps} a^{Ea} \tilde{g}^{Er} (w_{D}^{\eps} - w_{D}^{Er}) \dx \to 0 \text{ as } \eps \to 0,
\end{align*}
and use Cauchy--Schwarz inequality, (\ref{apriori:Dirichlet}), and (\ref{H1zeroL2deltaeps-1conv}) to deduce
\begin{align*}
& \; \abs{ \int_{\Omega} \frac{1}{\eps} \delta_{\eps} w_{D}^{E0} (w_{D}^{\eps} - w_{D}^{E0}) \dx} \leq \norm{w_{D}^{E0}}_{0,\frac{1}{\eps} \delta_{\eps}} \left ( \norm{w_{D}^{\eps}}_{0,\frac{1}{\eps} \delta_{\eps}} + \norm{w_{D}^{E0}}_{0,\frac{1}{\eps} \delta_{\eps}} \right ) \\
\leq & \; C\norm{w_{D}^{E0}}_{0,\frac{1}{\eps} \delta_{\eps}} \left ( \norm{f^{Ea}}_{L^{2}(\Omega)} + \norm{\tilde{g}^{Er}}_{H^{1}(\Omega)} + \norm{w_{D}^{E0}}_{H^{1}(\Omega)} \right ) \to 0 \text{ as } \eps \to 0.
\end{align*}
\end{itemize}

\section{Discussion}\label{sec:Discussion}
\subsection{Alternate approximations for the Dirichlet problem}\label{sec:AltDirichlet}
In the derivation of $(\mathrm{DDDH})$, we substituted $\beta= \frac{1}{\eps}$.  However, an equally valid diffuse domain approximation to $(\mathrm{DSIH})$ is $(\mathrm{DDDH})_{m}$:
\begin{equation*}
\begin{alignedat}{2}
- \nabla \cdot (\xi_{\eps} \mathcal{A}^{E} \nabla w^{\eps}) +  \xi_{\eps} a^{E} w^{\eps} + \frac{1}{\eps^{m}} \delta_{\eps} w^{\eps} & = \xi_{\eps} f^{E} + \nabla \cdot (\xi_{\eps} \mathcal{A}^{E} \nabla \tilde{g}^{E}) - \xi_{\eps} a^{E} \tilde{g}^{E} && \text{ in } \Omega, \\
\mathcal{A}^{E} \nabla w^{\eps} \cdot \nu &= 0 && \text{ on } \pd \Omega,
\end{alignedat}
\end{equation*}
where $m > 0$.  We see that $(\mathrm{DDDH})$ is the case $m = 1$.  Let us introduce the weighted Sobolev space
\begin{align*}
\mathcal{W}_{\eps}^{m} & := \{ f : \Omega \to \R \text{ measurable s.t. } f \mid_{\Omega_{\eps}} \in W^{1,2}(\Omega_{\eps}, \xi_{\eps}) \} \text{ with } \\
\inner{f}{g}_{\mathcal{W}_{\eps}^{m}} & := \int_{\Omega} \xi_{\eps} (fg + \nabla f \cdot \nabla g) + \frac{1}{\eps^{m}} \delta_{\eps} fg \dx.
\end{align*}
One can use the Lax--Milgram theorem to show that $(\mathrm{DDDH})_{m}$ is well-posed in $\mathcal{W}_{\eps}^{m}$ with unique weak solution $w_{D,m}^{\eps}$.

For $g \in C^{1}(\overline{\Omega})$ such that $g \mid_{\Gamma} = 0$, we see that for any $m \in (0,2)$, by (\ref{property:delta:ptwiseconv}) with $q = 1+m$, and (\ref{convergetodirac:zerotrace}) with $q = \infty$,
\begin{align*}
\int_{\Omega} \frac{1}{\eps^{m}} \delta_{\eps} \abs{g}^{2} \dx \to 0 \text{ as } \eps \to 0.
\end{align*}
However, due to (\ref{convergetodirac:zerotrace}) with $q = 2$, for any $f \in H^{1}_{\Gamma,0}(\Omega)$, we can only show that there exists a constant $C > 0$, independent of $f$ and $\eps \in (0,1]$, such that
\begin{align*}
\int_{\Omega} \frac{1}{\eps^{m}} \delta_{\eps} \abs{f}^{2} \dx \leq C \norm{f}_{H^{1}(\Omega)}^{2} \text{ for } m \in (0,1].
\end{align*}

Thus, for any $m \in (0,1]$, using the techniques presented here, we can show
\begin{align*}
\norm{w_{D,m}^{\eps} \mid_{\Omega^{*}} - w_{D}}_{H^{1}(\Omega^{*})} \to 0 \text{ as } \eps \to 0,
\end{align*}
where $w_{D}$ is the unique weak solution to (DSIH).  We point out that it is not known if the formal asymptotic methods used in \cite{article:AbelsGarckeGrun11, article:GarckeLamStinner14, article:LiLowengrubRatzVoigt09} can be applied to non-integer $m$.  However, this result may be beneficial for numerical computation, where high values of $m$ and small values of $\eps$ usually result in stiff computations, and thus are computationally expensive.

\subsection{Issue of uniform estimates for (NDD)}\label{sec:Neumann}
Let us define
\begin{align*}
a_{NDD}(u,\phi) := a_{B}^{\eps}(u,\phi), \quad l_{NDD}(\phi) := l_{B}^{\eps}(f^{Ea},\phi) + l_{S}^{\eps}(g^{Ec},\phi).
\end{align*}
Suppose we seek weak solutions in $H^{1}(\Omega,\xi_{\eps})$, then $a_{NDD}(\cdot,\cdot)$ is bounded and coercive uniformly in $\eps$.  But, by (\ref{XiDeltaRelation}), the boundedness of $l_{NDD}$ over $H^{1}(\Omega, \xi_{\eps})$ scales with $\frac{1}{\eps}$.  Thus, we cannot obtain uniform estimate similar to (\ref{apriori:Neumann}).  If we seek solutions in $\mathcal{V}_{\eps}$, then $l_{NDD}$ is bounded uniformly in $\eps$.  But $a_{NDD}$ is not coercive in $\mathcal{V}_{\eps}$ uniformly in $\eps$.  Hence, to obtain uniform estimates in $\eps$, we have to remove the term involving $\delta_{\eps}$, which amounts to transforming (NSI) to (NSIH) through Theorem \ref{thm:conormal}.  This, then motivates Assumption \ref{assump:NSI} and the corresponding diffuse domain approximation (NDDH) now has a weak solution in $H^{1}(\Omega, \xi_{\eps})$ with uniform estimates in $\eps$.

We point out that if a trace-type inequality holds:  
\begin{align*}
\int_{\Omega} \delta_{\eps} \abs{u}^{2} \dx \leq C \int_{\Omega} \xi_{\eps}(\abs{u}^{2} + \abs{\nabla u}^{2}) \dx,
\end{align*}
for some $C > 0$ independent of $\eps$, then the Hilbert spaces $\mathcal{V}_{\eps}$ and $H^{1}(\Omega, \xi_{\eps})$ are equivalent uniformly in $\eps$.  Moreover, (NDD) will be well-posed over $H^{1}(\Omega, \xi_{\eps}) = \mathcal{V}_{\eps}$, and Assumption \ref{assump:NSI} is not required.  However, stronger assumptions than Assumptions \ref{assump:Xi} and \ref{assump:delta} are required, see for instance \cite[Theorem 4.2]{preprint:BurgerElvetunSchlottbom14} or Section \ref{sec:Burger}.

\subsection{Regularity of $\Gamma$ for (CSI) and (SSI)}
In the proof of Lemma \ref{lem:H1deltabddconv}, we used that the components $g_{\eps}^{ij}$, $1 \leq i, j \leq n$, of the inverse of the metric tensor $\mathcal{G}_{\eps}^{-1}$ are $C^{1}$ functions.  Then, from the definition, we require that $\Gamma$ is of class $C^{3}$.  Since we apply Lemma \ref{lem:H1deltabddconv} only to (CDD) and (SDD), for the other problems, a $C^{2}$ regularity for $\Gamma$ is sufficient. 

We point out that using (\ref{ansatz}) and (\ref{Geps}), we can consider the following splitting in (\ref{surfgradientIBP}):
\begin{equation}\label{splitting}
\begin{aligned}
(\surfzloc \tilde{V}^{\eps})_{r} & = \sum_{j=1}^{n-1} g_{\eps}^{jr} \pd_{s_{j}} \tilde{V}^{\eps} \pd_{s_{r}} G_{\eps}\\
& = \sum_{j=1}^{n-1} g_{0}^{jr} \pd_{s_{j}} \tilde{V}^{\eps} \pd_{s_{r}} \alpha +  \eps z g_{0}^{jr} \pd_{s_{j}} \tilde{V}^{\eps} \pd_{s_{r}} \alpha + \mathcal{E}_{jr} \pd_{s_{j}} \tilde{V}^{\eps} \pd_{s_{r}} G_{\eps}.
\end{aligned}
\end{equation}

If we have some control over the components $(\pd_{s_{j}} \tilde{V}^{\eps})_{1 \leq j \leq n-1}$, then it is sufficient to apply integration by parts only on the first term.  This in turn implies that we can potentially drop the required regularity of $\Gamma$ from $C^{3}$ to $C^{2}$.  However, the hypothesis that $\norm{v^{\eps}}_{1,\delta_{\eps}}$ is bounded uniformly in $\eps$ seems to be not sufficient to give any control over the components $(\pd_{s_{j}} \tilde{V}^{\eps})_{1 \leq j \leq n-1}$.

\subsection{Comparison with the results in \cite{article:ElliottStinner09}}
In a time-independent setting, by choosing $q = 1$, $\rho_{\eps}(s,z) = \overline{\rho}(z) = \delta(z)$, the function spaces $B$ and $X$ employed in \cite{article:ElliottStinner09} are equivalent to $L^{2}(\mathrm{Tub}^{\eps}(\Gamma), \delta_{\eps})$ and $H^{1}(\mathrm{Tub}^{\eps}(\Gamma), \delta_{\eps})$, respectively.  Moreover, by comparing with the notation and results in Section \ref{sec:ScaledNbd}, the results of \cite{article:ElliottStinner09} in our notation are
\begin{align*}
c_{\eps}(s,z) \rightarrow c(s) \in B \Longleftrightarrow \norm{c_{\eps} - c}_{L^{2}(\mathrm{Tub}^{\eps}(\Gamma),\delta_{\eps})} \to 0,
\end{align*}
and 
\begin{align*}
c_{\eps}(s,z) \rightharpoonup c(s) \in X \Longleftrightarrow \int_{\mathrm{Tub}^{\eps}(\Gamma)} \delta_{\eps}(c_{\eps} - c)\psi \dx \to 0, \, \int_{\mathrm{Tub}^{\eps}(\Gamma)} \delta_{\eps} \nabla (c_{\eps} - c) \cdot \nabla \psi \dx \to 0,
\end{align*}
for $\psi \in H^{1}(\mathrm{Tub}^{\eps}(\Gamma), \delta_{\eps})$.  Thus, in the limit $\eps \to 0$, $c_{\eps}$ converges weakly to a function $c$ defined only on the surface $\Gamma$.

We point out that, in the proof of Lemma \ref{lem:H1deltabddconv}, it is crucial that (\ref{pdzpsiezero}) holds, i.e., the test function $\psi$ is extended constantly in the normal direction.  Otherwise, for an arbitrary test function $\lambda \in H^{1}(\Omega, \delta_{\eps})$ with representation $\Lambda_{\eps}$ in the $(p,z)$ coordinate system, a similar calculation to (\ref{vepsbddH1delta}) yields that
\begin{align*}
\chi_{(-\eta_{\eps,k}, \eta_{\eps,k})}(z) \frac{1}{\eps} \pd_{z} \Lambda_{\eps} \text{ is bounded in } L^{2}(\Gamma \times \R, \delta) \text{ for all } \eps \in (0,1].
\end{align*}
When computing (\ref{transformBnablavepsnablapsi}), we have an additional term of the form
\begin{align}\label{extraintegralH1deltatest}
\int_{\eta_{\eps,k}}^{\eta_{\eps,k}} \int_{\Gamma} \delta(z) \frac{1}{\eps^{2}} \pd_{z}V^{\eps} \pd_{z} \Lambda_{\eps} \abs{1 + \eps z C_{H}(p) + C_{R}(p, \eps z)} \dHaus \dz,
\end{align}
and by (\ref{vepsbddH1delta}) and Cauchy--Schwarz inequality, we obtain only the uniform boundedness of (\ref{extraintegralH1deltatest}) in $\eps$, and we are unable to show that (\ref{extraintegralH1deltatest}) converges to zero as $\eps \to 0$. 

Thus, we are unable to enrich the space of test functions of the weak convergence in Theorem \ref{thm:weakconv} and Lemma \ref{lem:H1deltabddconv} from $C_{\Gamma}$ (\ref{defn:CGamma}) to $H^{1}(\Omega, \delta_{\eps})$.  However, in \cite[Proof of Theorem 4.1]{article:ElliottStinner09}, to show that the limit $c$ satisfies the correct weak formulation of an advection diffusion surface PDE, the authors considered a test function $\chi \in X$ with $\pd_{z} \chi = 0$, which is similar to what we do in Lemma \ref{lem:H1deltabddconv}.

\subsection{Comparison with the results in \cite{preprint:BurgerElvetunSchlottbom14}}\label{sec:Burger}
The framework developed in \cite{preprint:BurgerElvetunSchlottbom14} bears the most resemblance to the framework presented here.  In particular, weighted Sobolev spaces are employed for the well-posedness of the diffuse domain approximation, and the authors also studied tubular neighbourhoods and derived results similar to those in Section \ref{sec:TubNbd}.  In our notation, the object of study in \cite{preprint:BurgerElvetunSchlottbom14} is the Robin boundary value problem
\begin{equation}\label{Burger:Robin}
\begin{aligned}
- \nabla \cdot (A \nabla u ) + cu = f & \text{ in } \Omega^{*}, \\
A \nabla u \cdot \nu + bu = g & \text{ on } \Gamma.
\end{aligned}
\end{equation}
With the choice $\delta_{\eps} = \abs{\nabla \xi_{\eps}}$, the diffuse domain approximation in weak formulation is given as
\begin{align}\label{Burger:RDD}
\int_{\Omega_{\eps}} \xi_{\eps} (A \nabla u^{\eps} \cdot \nabla v + c u^{\eps} v) + \abs{\nabla \xi_{\eps}} b u^{\eps} v \dx = \int_{\Omega_{\eps}} \xi_{\eps} f v + \abs{\nabla \xi_{\eps}} g v \dx
\end{align}
for all $v \in H^{1}(\Omega_{\eps}, \xi_{\eps})$, where $\Omega_{\eps}$ is defined in (\ref{defn:OmegaepsGammaeps}).  Moreover, the authors can show a trace-type inequality:  There exists a $C > 0$ such that for all $v \in W^{1,p}(\Omega_{\eps},\xi_{\eps})$, $1 \leq p < \infty$, and $\eps$ sufficiently small,
\begin{align}\label{Burger:trace}
\int_{\Omega_{\eps}} \abs{v}^{p} \abs{\nabla \xi_{\eps}} \dx \leq C \int_{\Omega_{\eps}} \xi_{\eps} (\abs{v}^{p} + \abs{\nabla v}^{p}) \dx.
\end{align}
Additionally, it is assumed that $\xi_{\eps}$ satisfies the following behaviour:  There exists $\zeta_{1}, \zeta_{2} > 0$ and $\alpha > 0$ such that for all $x \in \Omega_{\eps}$, 
\begin{align}\label{Burger:assumpXi}
\zeta_{1} \left ( \frac{\dist(x, \pd \Omega_{\eps})}{\eps} \right )^{\alpha} \leq \xi_{\eps}(x) \leq \zeta_{2} \left ( \frac{\dist(x, \pd \Omega_{\eps})}{\eps} \right )^{\alpha}.
\end{align}
Then, the authors can show continuous and compact embeddings from $W^{1,p}(\Omega_{\eps}, \xi_{\eps})$ into $L^{q}(\Omega_{\eps}, \xi_{\eps})$, as well as a Poincar\'{e} type inequality: There exists a $C_{P} > 0$ such that for all $v \in W^{1,p}(\Omega_{\eps},\xi_{\eps})$, $1 \leq p < \infty$ and $\eps$ sufficiently small,
\begin{align}\label{Burger:Poincare}
\norm{v}_{L^{p}(\Omega_{\eps},\xi_{\eps})}^{p} \leq C_{P} \left ( \norm{\nabla v}_{L^{p}(\Omega_{\eps}, \xi_{\eps})}^{p} + \int_{\Omega_{\eps}} \abs{v}^{p} \abs{\nabla \xi_{\eps}} \dx \right ).
\end{align}
Using these results, together with the fact that $\Omega_{\eps} = \Omega^{*} \cup \Gamma_{\eps}$ and $\abs{\Gamma_{\eps}} \leq C \eps$, the authors deduced that
\begin{align*}
\abs{\int_{\Omega_{\eps}} \xi_{\eps} h \dx - \int_{\Omega^{*}} h \dx} \leq \begin{cases}
C \norm{h}_{L^{p}(\Gamma_{\eps}, \xi_{\eps})} \eps^{1 - \frac{1}{p}}, & \text{ for }  h \in L^{p}(\Omega_{\eps}, \xi_{\eps}), \; 1 < p \leq \infty, \\
C \norm{h}_{W^{1,p}(\Gamma_{\eps}, \xi_{\eps})} \eps^{2 - \frac{1}{p}}, & \text{ for } h \in W^{1,p}(\Omega_{\eps}, \xi_{\eps}), \; 1 \leq p \leq \infty, \\
\to 0 \text{ as } \eps \to 0, & \text{ for } h \in L^{1}(\Omega).
\end{cases}
\end{align*}
We note that the last case is similar to Lemma \ref{lem:XiDCT}, and both proofs use the dominated convergence theorem and the pointwise limit of $\xi_{\eps}$.

By the application of (\ref{Burger:trace}), (\ref{Burger:Poincare}) and the Lax--Milgram theorem, the diffuse domain approximation is well-posed with a unique weak solution $u^{\eps} \in H^{1}(\Omega_{\eps}, \xi_{\eps})$.  The chief results of \cite{preprint:BurgerElvetunSchlottbom14} is the following error estimate:
\begin{align}\label{Burger:ErrorEst}
\norm{u - u^{\eps}}_{W^{1,2}(\Omega_{\eps}, \xi_{\eps})} \leq C \eps^{\frac{1}{2} - \frac{1}{p}},
\end{align}
where $u \in W^{1,p}(\Omega^{*})$, $2 < p \leq 2^{*}_{\alpha}$ is the weak solution to (\ref{Burger:Robin}), $u^{\eps} \in H^{1}(\Omega_{\eps}, \xi_{\eps})$ is the unique weak solution to the diffuse domain approximation, $C$ is independent of $\eps$, $\alpha$ is the exponent in (\ref{Burger:assumpXi}), and
\begin{align*}
p^{*}_{\alpha} = \frac{p(n+\alpha)}{n + \alpha - p} \text{ if } p < n + \alpha, \quad p^{*}_{\alpha} = \infty \text{ if } p \geq n + \alpha.
\end{align*}
We point out the differences between our work and the work of \cite{preprint:BurgerElvetunSchlottbom14}:
\begin{itemize}
\item  \cite{preprint:BurgerElvetunSchlottbom14} utilises a double-obstacle regularisation and the fact that $\abs{\Gamma_{\eps}} \leq C \eps$ for some constant $C$ to deduce (\ref{Burger:ErrorEst}).  In our work, we cover both the double-well and double obstacle regularisations.
\item We do not have a rate of convergence for the error $\norm{u^{\eps}_{R} - u_{R}^{Er}}_{1,\xi_{\eps}}$, but \cite{preprint:BurgerElvetunSchlottbom14} requires $W^{1,p}$ regularity of the original problem to deduce (\ref{Burger:ErrorEst}).  This follows in the same spirit as (\ref{convergetodirac:power}), where for more regular functions, we are able to deduce a rate of convergence.
\item For the choice $\delta_{\eps} = \abs{\nabla \xi_{\eps}}$, the trace-type inequality (\ref{Burger:trace}) allows \cite{preprint:BurgerElvetunSchlottbom14} to handle non-homogeneous Neumann boundary problems.  Our setting is more general, where $\delta_{\eps}$ and $\xi_{\eps}$ can be unrelated, as long as  (\ref{property:Xidelta}) is satisfied.  But as discussed in Section \ref{sec:Neumann}, we require stronger assumptions on the data to compensate for the generality.
\item In \cite{preprint:BurgerElvetunSchlottbom14}, the non-homogeneous Dirichlet problem is handled with a penalization method in the same spirit as Lemma \ref{lem:DirichletfromRobin}.  The corresponding diffuse domain approximation is (\ref{Burger:RDD}) but with $\frac{1}{\eps^{m}} \abs{\nabla \xi_{\eps}}$ replacing $\abs{\nabla \xi_{\eps}}$ with $m > 0$.  The error estimate for the Dirichlet problem becomes
\begin{align*}
\norm{u_{D}^{\eps}  \mid_{\Omega^{*}} - u_{D}}_{H^{1}(\Omega^{*})} \leq C (\eps^{m} + \eps^{\frac{1}{2} - \frac{1}{p} - m}),
\end{align*}
where $u \in W^{1,p}(\Omega^{*})$, $2 \leq p \leq 2^{*}_{\alpha}$, is the unique solution to (DSI), and $u_{D}^{\eps} \in H^{1}(\Omega_{\eps}, \xi_{\eps})$ is the unique solution to the diffuse domain approximation.  An optimal choice is $m = \frac{1}{4} - \frac{1}{2p}$, leading to a rate of convergence of order $\mathcal{O}(\eps^{\frac{1}{4} - \frac{1}{2p}})$.  In contrast, we homogenise the problem and use Lemma \ref{lem:DirichletfromRobin} to derive a family of approximation $(\mathrm{DDDH})_{m}$, $m > 0$.  Although we do not have any rate of convergence, we showed strong convergence for any $m \in (0,1]$ (see Section \ref{sec:AltDirichlet}).
\item Lastly, in this work, we are also able to consider partial differential equations purely on the boundary $\Gamma$ or coupled with the bulk equation, whereas \cite{preprint:BurgerElvetunSchlottbom14} only focus on a bulk elliptic equation with Robin, Neumann and Dirichlet boundary conditions.
\end{itemize}


\appendix \section{Derivation of the distributional forms}
A key result we use to derive an equivalent distributional form is that of Alt \cite[\S 2.7 and Theorem 2.8]{article:Alt09} (see also \cite[Theorem 6.1]{article:GarckeLamStinner14}).  We shall give a time independent version:
\begin{thm}\label{thm:Alt}
Given an open set $\Omega \subset \R^{n}$ containing an open subset $\Omega^{*}$ such that $\Gamma = \pd \Omega^{*} \subset \Omega$.  For $x \in \Gamma$, let $\nu(x) \in (T_{x} \Gamma)^{\perp}$ be the external unit normal of $\Gamma$.  Denote by $\chi_{\Omega^{*}}$ and $\delta_{\Gamma}$ the following distributions:
\begin{align*}
\int_{\Omega} f \dd \chi_{\Omega^{*}} = \int_{\Omega^{*}} f \dx, \quad \int_{\Omega} f \dd \delta_{\Gamma} = \int_{\Gamma} f \dHaus.
\end{align*}
Then, for smooth functions $q_{j}^{b}, f^{b} : \overline{\Omega^{*}} \to \R$, $q_{j}^{s}, f^{s} : \Gamma \to \R$, $1 \leq j \leq n$, the distributional law,
\begin{align*}
\nabla \cdot (\chi_{\Omega^{*}} \bm{q}^{b} + \delta_{\Gamma} \bm{q}^{s}) = \chi_{\Omega^{*}} f^{b} + \delta_{\Gamma} f^{s} \text{ in } \mathcal{D}'(\Omega),
\end{align*}
is equivalent to the following:
\begin{alignat*}{2}
\nabla \cdot \bm{q}^{b} &= f^{b} && \text{ in } \Omega^{*}, \\
\bm{q}^{s}(p) \in T_{p} \Gamma, \quad \surf \cdot \bm{q}^{s} & = f^{s} + \bm{q}^{b} \cdot \nu && \text{ on } \Gamma, 
\end{alignat*}
where $\surf (\cdot)$ denotes the tangential gradient on $\Gamma$.
\end{thm}

To derive the equivalent distributional form for (CSI), we set
\begin{align*}
\bm{q}^{b} = - \mathcal{A} \nabla u, \quad \bm{q}^{s} = \bm{0}, \quad f^{b} = f - a u, \quad f^{s} = K(v - u).
\end{align*}
Then, by Theorem \ref{thm:Alt}, 
\begin{align*}
-\nabla \cdot (\chi_{\Omega^{*}} \mathcal{A} \nabla u) + \chi_{\Omega^{*}} a u = \chi_{\Omega^{*}}f + \delta_{\Gamma} K(v-u) \text{ in } \mathcal{D}'(\Omega)
\end{align*}
is the equivalent distributional form for
\begin{align*}
-\nabla \cdot (\mathcal{A} \nabla u) + au = f & \text{ in } \Omega^{*}, \\
\mathcal{A} \nabla u \cdot \nu = K(v-u) & \text{ on } \Gamma.
\end{align*}
Now, setting
\begin{align*}
\bm{q}^{b} = \bm{0}, \quad \bm{q}^{s} = - \mathcal{B} \nabla v^{e}, \quad f^{b} = 0, \quad f^{s} = \beta g - bv - K(v-u),
\end{align*}
where $v^{e}$ is the extension of $v$ constantly in the normal direction as outlined in Section \ref{sec:ConstExt} (see also (\ref{ConstExt:samegradients})).  Then,
\begin{align*}
-\nabla \cdot (\delta_{\Gamma} \mathcal{B} \nabla v^{e}) + \delta_{\Gamma} bv = \delta_{\Gamma} \beta g - \delta_{\Gamma} K(v-u) \text{ in } \mathcal{D}'(\Omega)
\end{align*}
is the equivalent distribution form for
\begin{align*}
-\surf \cdot (\mathcal{B} \surf v) + bv + K(v-u) = \beta g \text{ on } \Gamma.
\end{align*}
Hence, the equivalent distributional form for (CSI) is
\begin{equation}\label{CSIreform}
\begin{aligned}
-\nabla \cdot (\chi_{\Omega^{*}} \mathcal{A} \nabla u) + \chi_{\Omega^{*}} a u = \chi_{\Omega^{*}}f + \delta_{\Gamma} K(v-u) & \text{ in } \mathcal{D}'(\Omega), \\
-\nabla \cdot (\delta_{\Gamma} \mathcal{B} \nabla v^{e}) + \delta_{\Gamma} bv = \delta_{\Gamma} \beta g - \delta_{\Gamma} K(v-u) & \text{ in } \mathcal{D}'(\Omega).
\end{aligned}
\end{equation}

\section{Properties of the trace operator}
We use the following result from \cite[Theorem 1.5.1.2]{book:Grisvard} regarding the continuous right inverse of the trace operator for Sobolev functions (see also \cite[Theorem, p. 211]{book:Triebel} and \cite[Theorem 3.37]{book:McLean}):
\begin{thm}\label{thm:trace}
Let $k \geq 0$ be an integer and $\Omega^{*}$ is a bounded open subset of $\R^{n}$ with $C^{k,1}$ boundary $\Gamma$.  Assume that $s - \frac{1}{p} > 0$ is not an integer with $s \leq k+1$.  Then the trace operator $\gamma_{0} : W^{s,p}(\Omega^{*}) \to W^{s - \frac{1}{2},p}(\Gamma)$ has a continuous right inverse.
\end{thm}

The derivative $\mathcal{A} \nabla u \cdot \nu$ on $\Gamma$ is defined by
\begin{align*}
\mathcal{A} \nabla u \cdot \nu := \sum_{j=1}^{n} \nu_{j} \gamma_{0}(\mathcal{A} \nabla u)_{j} \text{ on } \Gamma,
\end{align*}
where $\nu$ is the outward unit normal to $\Gamma$.  This object is well-defined for $u \in H^{2}(\Omega^{*})$ and $\mathcal{A} \in (C^{0}(\overline{\Omega^{*}}))^{n \times n}$.  The analogous result to Theorem \ref{thm:trace} for the derivative $\mathcal{A} \nabla u \cdot \nu$, is a consequence of the solvability and regularity of solutions to non-homogeneous Neumann problems (see \cite[Theorem 4.18 (ii)]{book:McLean} and \cite[Theorem 4]{article:Savare98}).  We also refer the reader to \cite[Theorem 1.6.1.3]{book:Grisvard} for a similar result with stronger assumptions on $\mathcal{A}$.

\begin{thm}\label{thm:conormal}
Let $k \geq 0$ be an integer and $\Omega^{*} \subset \R^{n}$ is a bounded open set, with boundary $\Gamma \in C^{k+1,1}$ and outward unit normal $\nu$.  Assume that the matrix $\mathcal{A}(x)$ is uniformly elliptic and each coefficient belongs to $C^{k,1}(\overline{\Omega^{*}})$.  Then, if $\tilde{f} \in H^{k}(\Omega^{*})$, $\tilde{g} \in H^{k+\frac{1}{2}}(\Gamma)$, the non-homogeneous Neumann problem (with $\lambda > 0$)
\begin{alignat*}{2}
-\nabla \cdot (\mathcal{A} \nabla u) + \lambda u & = \tilde{f} && \text{ in } \Omega^{*}, \\
\mathcal{A} \nabla u \cdot \nu & = \tilde{g} && \text{ on } \Gamma,
\end{alignat*}
admits a unique solution $u \in H^{k+2}(\Omega^{*})$ with
\begin{align*}
\norm{u}_{H^{k+2}(\Omega^{*})} \leq C(\norm{u}_{H^{1}(\Omega^{*})} + \norm{\tilde{f}}_{H^{k}(\Omega^{*})} + \norm{\tilde{g}}_{H^{k+\frac{1}{2}}(\Gamma)}).
\end{align*}
\end{thm}

We will use the following compactness property of the trace operator $\gamma_{0}(\cdot)$ from \cite[Theorem A 6.13, p. 257]{book:Alt} (see also \cite[Theorem 3.8.5, p. 167]{book:Demengel} or \cite[Theorem 6.2, p. 103]{book:Necas}):
\begin{thm}\label{thm:compacttrace}
Let $\Omega \subset \R^{n}$ be an open bounded domain with Lipschitz boundary, and $1 \leq p < \infty$.  For $k \in \N$, let $u_{k},u \in W^{1,p}(\Omega)$.  Then, as $k \to \infty$,
\begin{align*}
u_{k} \rightharpoonup u \text{ in } W^{1,p}(\Omega) \Rightarrow \gamma_{0}(u_{k}) \rightarrow \gamma_{0}(u) \text{ in } L^{p}(\pd \Omega).
\end{align*}
\end{thm}

\bibliography{DiffuseDomain}
\bibliographystyle{plain}
\end{document}